\documentclass[a4paper, 10pt, reqno]{amsart}

\usepackage[utf8]{inputenc}
\usepackage[T1]{fontenc}
\usepackage{amsthm}
\usepackage{amsmath}
\usepackage{amssymb}
\usepackage[inline]{enumitem}
\usepackage[dvips]{graphicx}
\usepackage{color}
\usepackage{comment}
\usepackage{hyperref}
\usepackage{url}
\usepackage{stackrel}
\usepackage{fancyhdr}
\usepackage{mathrsfs}
\usepackage{stmaryrd}

\newtheorem{theorem}{Theorem}[section]
\newtheorem*{main*}{Main Theorem}
\newtheorem{lemma}[theorem]{Lemma}
\newtheorem{proposition}[theorem]{Proposition}
\newtheorem{corollary}[theorem]{Corollary}

\newtheorem{claim}{\textsc{Claim}}

\newtheorem*{claim*}{\textsc{Claim}}

\theoremstyle{definition}

\newtheorem*{def*}{Definition}
\newtheorem{definition}[theorem]{Definition}

\newtheorem{remark}[theorem]{\textbf{Remark}}

\definecolor{blue-url}{RGB}{0,0,100}

\hypersetup{
  pdftitle={Power monoids},
  pdfauthor={Yushuang Fan and Salvatore Tringali},
  pdfmenubar=false,
  pdffitwindow=true,
  pdfstartview=FitH,
  colorlinks=true,
  linkcolor=blue-url,
  citecolor=green,
  urlcolor=blue-url
}

\DeclareMathSymbol{\widehatsym}{\mathord}{largesymbols}{"62}

\renewcommand{\emptyset}{\varnothing}

\renewcommand{\setminus}{\smallsetminus}

\providecommand{\cat}{{\sf Ca}}

\providecommand{\LLc}{\mathscr{L}}
\providecommand{\LLs}{\mathsf{L}}
\providecommand{H}{H}
\providecommand{\NNb}{{\mathbf{N}}}
\providecommand{\ord}{{\rm ord}}

\providecommand{\PPc}{\mathcal{P}}

\providecommand{\UUc}{\mathscr{U}}

\providecommand\llb{\llbracket}
\providecommand\rrb{\rrbracket}

\providecommand\fin{{\rm fin}}
\providecommand\fun{{{\rm fin},\times}}
\providecommand\fuN{{{\rm fin},0}}

\newcommand{\dual}[1]{%
 \mathcal P_{\rm fin}({#1})%
}

\newcommand{\fixed}[2][1]{%
 \begingroup
 \spaceskip=#1\fontdimen2\font minus \fontdimen4\font
 \xspaceskip=0pt\relax
 #2%
 \endgroup
}

\newcommand{\BF}{\text{\rm BF}}

\begin{document}
\title{Power Monoids: A Bridge between \\ Factorization Theory and Arithmetic Combinatorics}

\author{Yushuang Fan}
\address{Mathematical College, China University of Geosciences | Haidian District, Beijing, China}
\email{fys@cugb.edu.cn}
\urladdr{http://yushuang-fan.weebly.com/}
\author{Salvatore Tringali}
\address{Institute for Mathematics and Scientific Computing, University of Graz, NAWI Graz | Heinrichstr. 36, 8010 Graz, Austria}
\email{salvatore.tringali@uni-graz.at}
\urladdr{http://imsc.uni-graz.at/tringali}
%
%\thanks{Last version: \textbf{\today{} at \currenttime}}
%
% 11-XX: NUMBER THEORY
% 11Bxx: **Sequences and sets**
% 11B13: Additive bases, including sumsets
% 11B30: Arithmetic combinatorics; higher degree uniformity
% 11Rxx: **Algebraic number theory: global fields**
% 11R27: Units and factorization
% 11Pxx: **Additive number theory; partitions**
% 11P70: Inverse problems of additive number theory, including sumsets
%
% 13-XX: COMMUTATIVE ALGEBRA
% 13A05: Divisibility; factorizations
% 13F15: Rings defined by factorization properties (e.g., atomic, factorial, half-factorial)
%
% 16-XX: ASSOCIATIVE RINGS AND ALGEBRAS
% 16Uxx: **Conditions on elements**
% 16U30: Divisibility, noncommutative UFDs
%
% 20-XX: GROUP THEORY AND GENERALIZATIONS
% 20Mxx: **Semigroups**
% 20M13: Arithmetic theory of monoids
% 20M25: Semigroup rings, multiplicative semigroups of rings

\thanks{Y.F. was supported by the National Natural Science Foundation of China (NSFC), Project No. 11401542, and the China Scholarship Council (CSC). S.T. was supported by the Austrian Science Fund (FWF), Project No. M 1900-N39.}
\begin{abstract}
\noindent{}
We extend a few fundamental aspects of the classical theory of non-unique factorization, as presented in Geroldinger and Halter-Koch's 2006 monograph on the subject, to a non-commutative and non-cancellative setting, in the same spirit of Baeth and Smertnig's work on the factorization theory of non-commutative, but cancellative monoids [J. Algebra \textbf{441} (2015), 475--551].
Then, we bring in power monoids and, applying the abstract machinery developed in the first part, we undertake the study of their arithmetic.

More in particular, let $H$ be a multiplicatively written monoid. The set $\mathcal P_{\rm fin}(H)$ of all non-empty finite subsets of $H$ is naturally made into a monoid, which we call the power monoid of $H$ and is non-cancellative unless $H$ is trivial, by endowing it with the operation $(X,Y) \mapsto \{xy: (x,y) \in X \times Y\}$.
Power monoids are, in disguise, one of the primary objects of interest in arithmetic combinatorics, and here for the first time we tackle them from the perspective of factorization theory.
Proofs lead to consider various properties of finite subsets of $\mathbf N$ that can or cannot be split into a sumset in a non-trivial way, giving rise to a rich interplay with additive number theory.
\end{abstract}
\maketitle
\thispagestyle{empty}

\section{Introduction}
\label{sec:intro}
From the classical point of view, factorization theory is all about the study of phenomena arising from
the non-uniqueness of factorization in atomic monoids and rings, and the classification of these phenomena
by a variety of algebraic, arithmetic, or combinatorial invariants.

The theory grew up out of algebraic number theory and has so far been centered on rings and monoids, where the structures in play are cancellative. The subject has become more and more popular since the publication of Geroldinger and Halter-Koch's 2006 monograph \cite{GeHK06}, which is entirely devoted to the commutative and cancellative case: A more accurate overview of the field is beyond the scope here, but further information and background can be found in the conference proceedings \cite{And97,ChGl00,Ch05,ChFoGeOb16}, in the surveys \cite{BaCh11,BaWi13,Ge16c}, or in the volumes \cite{Nark74,FoHoLu13}.

It is, indeed, the main objective of the present work to extend fundamental aspects of factorization theory to \textit{arbitrary} monoids (in a more systematic way than done in the past) and, as an application, to inquire into the arithmetic properties of a new class of ``highly non-cancellative'' structures we refer to as power monoids (notations and terminology will be explained later, see, in particular, \S\S{} \ref{sec:monoids} and \ref{sec:power_monoid}).

Our motivation is twofold.
On the one hand, there has been a mounting interest for possible generalizations of factorization theory to monoid-like structures that need no longer be commutative or cancellative \cite{BaSm, ChAn13, FGKT, Ge16c, Ge13, GeSc17, Sm13}.
On the other, power monoids are both an effective test bed for these generalizations and, in disguise, one of the primary objects of study in arithmetic combinatorics, a very active area of research, which has undergone tremendous developments in recent years, rapidly expanding from the classical bases of additive number theory \cite{Nath1, Nath2} (where the focus is on the integers)
to much more abstract settings involving non-commutative groups or semigroups \cite{Gry13, Rus, TV06}: In particular, power monoids can serve as a medium for arithmetic combinatorics
to benefit, in the long run, from the in\-ter\-ac\-tion with fac\-tor\-i\-za\-tion theory, much in the same way as the latter has, in its own right, drawn enormous benefits from the former, see \cite{Ge-Ru09,Tr18} and references therein.

For a basic example of the kind of connections alluded to in the previous paragraph, assume that $G$ is an additively written, finite group. A set $X \subseteq G$ is called \textit{irreducible} if there do not exist $A, B \subseteq G$ with $|A|, |B| \ge 2$ such that $X$ is the sumset of $A$ and $B$, namely, $X = \{a+b: (a,b) \in X \times Y\}$. This notion is related to deep questions in arithmetic combinatorics, see, e.g., \cite{Al07, AGU10, Sa12, GyKoSa13, GySa15}; and it follows from the definitions in \S{ }\ref{subsub:sets-of-lengths} and points \ref{it:prop:basic-properties-of-power-monoids(ii)} and \ref{it:prop:basic-properties-of-power-monoids(iii)} of Proposition \ref{prop:basic-properties-of-power-monoids} that a subset of $G$ is irreducible if and only if it is an atom in the power monoid of $G$.
\subsection{Plan of the paper and background.}\label{subsec:plan}
With these ideas in mind, we organize the paper as follows. In \S{ }\ref{sec:monoids}, we first extend a few fundamental aspects of the classical theory of non-unique factorization to a non-commutative and non-cancellative setting, in the same spirit of Baeth and Smertnig's work on the factorization theory of non-commutative, cancellative monoids \cite{BaSm} (see Remarks \ref{rem:where-is-0}--\ref{rem:philosophy-of-equimorphisms}, \ref{rem:decreasing-cat-degree}, \ref{rem:modelled-after}, and \ref{rem:Baeth-Smertning-permutable-distance} 
for a critical comparison).
More specifically, we introduce notions of factorization, distance, and catenary degree, along with a generalization of weak transfer homomorphisms we refer to as equimorphisms, and we prove a number of properties related to these notions:
In particular, we establish that equimorphisms preserve factorization lengths and do not increase the catenary degree (Theorem \ref{th:cotransfer-hom}). Moreover, we give conditions for a unit-cancellative monoid to be atomic (Theorem \ref{th:normable_implies_atomic}) and obtain a characterization of BF-monoids in terms of the existence of a length function (Corollary \ref{cor:BFness-of-unit-cancellative-monoids}), thus improving on analogous results of Smertnig in the cancellative setting \cite[Proposition 3.1]{Sm13}, and Geroldinger, Kainrath, and the authors in the commutative setting \cite[Lemma 3.1(1)]{FGKT}.

Then we bring in power monoids (Definition \ref{def:power-monoids}) and, applying the abstract machinery developed in the former part, undertake the study of their arithmetic. More in detail, let $H$ be a monoid. We denote the power monoid of $H$ by $\mathcal P_\fin(H)$, and show 
that $\mathcal P_{\rm fin}(H)$ is a {\rm BF}-monoid if $H$ is linearly orderable and \textrm{BF} (Proposition \ref{prop:lin-orderable-H}). In addition, we obtain that, if $H$ is a Dedekind-finite, non-torsion monoid,
then $\mathcal P_{\rm fin}(H)$ is not equimorphic to a cancellative monoid (in particular, is not a transfer Krull monoid), and that the union of the sets of lengths of $\mathcal P_{\rm fin}(H)$ containing $k$ is $\mathbf N_{\ge 2}$ for every integer $k \ge 2$; the set of distances (or delta set) is $\mathbf N^+$; and the set of catenary degrees is either $\mathbf N^+ \cup \{\infty\}$ or $\mathbf N^+$, the latter being the case if $H$ is a linearly orderable BF-monoid (Proposition \ref{prop:not_a_transfer_Krull_monoid} and Theorem \ref{th:main-theorem}, respectively).

It is probably worth stressing that we are talking here of several \textit{different} results, insofar as unions of sets of lengths, sets of distances, and sets of catenary degrees are, in principle, ``independent objects'', in the sense that, even in the commutative cancellative setting, none of them can be determined from the knowledge of the other two.

As for the proofs, we use transfer principles (see Remark \ref{rem:philosophy-of-equimorphisms} and Theorems \ref{th:cotransfer-hom} and \ref{th:transfer}) to reduce the kind of arithmetic properties we are considering to corresponding properties of finite subsets of $\mathbf N$ than can or cannot be written as
a sumset in a non-trivial way.
%, which leads to an intriguing interplay between factorization theory and additive number theory.

Analogous contributions have been made by many authors in the cancellative setting. In particular, it follows by work of Kainrath \cite[Theorem 1]{Ka90} that the delta set of a commutative Krull monoid with infinite class group in which
every class contains a prime divisor, is equal to $\mathbf N^+$, see also \cite[Theorem 17]{Ge16c}. The same is true, by \cite[Theorem 9]{Fr13}, for the monoid (under multiplication) of integer-valued polynomials with rational coefficients; and more generally, by \cite[Corollary 4.1]{FrNaRi17}, for the monoid of $D$-valued polynomials with coefficients in the fraction field of a Dedekind domain $D$ with infinitely many maximal ideals, all of which have finite index.

In a similar vein, Hassler has established that the set of distances of certain commutative Krull monoids with infinite class group (where every class is a sum of a bounded number of classes containing prime divisors) is infinite, see \cite[Theorem 1]{Has02}, while Smertnig has proved in \cite[Theorem 1.2]{Sm13} that, if $H$ is the multiplicative monoid of the non-zero elements of certain maximal orders in a simple central algebra over a number field, then $H$ is not necessarily a transfer Krull monoid, but the delta set of $H$ is still equal to $\mathbf N^+$ and the union of sets of lengths of $H$ containing $k$ is either $\mathbf N_{\ge 2}$ or $\mathbf N_{\ge 3}$ for every $k \ge 3$.
% (note that $H$ need not be commutative).

On a related note, Geroldinger and Schmid have obtained in \cite{GeSch16} that for every non-empty finite set $\Delta \subseteq \mathbf N^+$ with $\min \Delta = \gcd \Delta$ there is
a finitely generated, commutative Krull monoid whose set of distances is $\Delta$, while Geroldinger and Yuan had previously shown \cite[Theorem 1.1]{GeYu12} that the delta set of a commutative Krull monoid having prime divisors in all classes is either empty or a (discrete) interval whose minimum is equal to $1$. The latter result has been subsequently generalized by Geroldinger and Zhong to certain commutative, seminormal, weakly Krull monoids \cite[Theorem 1.1]{GeZh16}, while a non-commutative analogue was established by Smertnig in \cite[Theorem 1.1]{Sm13}. 
Further contributions to this line of research have been made, among others, by Chapman, Gotti, and Pelayo \cite{ChGoPe14}, Garc\'ia-Garc\'ia, Moreno-Fr\'ias, and Vigneron-Tenorio \cite{GaMoVi}, and Chapman, Garc\'ia-S\'anchez, Llena, Ponomarenko, and Rosales \cite{ChGaLlPoRo}.

As for the set of catenary degrees, this was also considered in a couple of recent papers by Fan and Geroldinger \cite{FG016} and O'Neill, Ponomarenko, Tate, and Webb \cite{OPTW16}, with the former focused on commutative Krull monoids and the latter on finitely generated, cancellative, commutative monoids.

We conclude the paper with a probably challenging problem that will stimulate, it is our hope, further research in the topic (see \S{ }\ref{sec:future} for details).
\subsection{Generalities}
\label{subsec:notations}
Unless noted otherwise, we reserve the letters $\ell$, $m$, $n$, and $r$ (with or without subscripts) for positive integers, and the letters $i$, $j$, and $k$ for non-negative integers. We use $\mathbf R$ for the reals, $\mathbf Z$ for the integers, and $\mathbf N$ for the non-negative integers.

A \textit{monoid} is a pair $(H, \otimes)$ consisting of a set $H$ (called the \textit{ground set} of the monoid and systematically identified with it if there is no risk of ambiguity) and an associative (binary) operation $\otimes: H \times H \to H$ for which there exists a (provably unique) element $e \in H$ (the \textit{identity} of the monoid) such that $e \otimes x = x \otimes e = x$ for all $x \in H$.
We assume that monoid homomorphisms preserve the identity, and for $X, Y \subseteq H$ we set
$
X \fixed[-0.2]{\text{ }} \otimes Y := \{x \otimes y: (x,y) \in X \times Y\}$. Note also that, if not stated otherwise, we will systematically use multiplicative notation for arbitrary monoids.

Given a set $\mathscr U$, we denote by $\mathscr{F}^\ast(\mathscr{U})$ the \textit{free monoid with basis $\mathscr{U}$}.
We write $\mathscr{F}^\ast(\mathscr{U})$ multiplicatively, and we adopt the symbol $\ast$ for its operation (so, for instance, if $\mathfrak z \in \mathscr F^\ast(\mathscr U)$, then $\mathfrak z^2 := \mathfrak z \ast \mathfrak z$).
We refer to the elements of $\mathscr{F}^\ast(\mathscr{U})$ as \textit{$\mathscr{U}$-$\fixed[0.2]{\text{ }}$words}, and to the identity, $1_{\mathscr{F}^\ast(\mathscr{U})}$, of $\mathscr{F}^\ast(\mathscr{U})$ as the empty word.

Let $\mathfrak z$ be a $\mathscr{U}$-$\fixed[0.2]{\text{ }}$word. We set $\|\mathfrak z\|_\mathscr{U} := 0$ if $\mathfrak{z} = 1_{\mathscr{F}^\ast(\mathscr{U})}$. Otherwise, there are determined $z_1, \ldots, z_n \in \mathscr{U}$ such that $\mathfrak z = z_1 \ast \cdots \ast z_n$. So we take $\|\mathfrak z\|_\mathscr{U} := n$, and we define $\mathfrak z \fixed[0.15]{\text{ }} z_n^{-1} := z_1^{-1} \mathfrak z := 1_{\mathscr{F}(\mathscr{U})}$ if $n = 1$, and $\mathfrak z \fixed[0.15]{\text{ }} z_n^{-1} := z_1 \ast \cdots \ast z_{n-1}$ and $z_1^{-1} \mathfrak z := z_2 \ast \cdots \ast z_n$ if $n \ge 2$. In both cases, we call $\|\mathfrak z\|_\mathscr{U}$ the (\textit{word}) \textit{length} of $\mathfrak z$ (relative to $\mathscr U$). It is clear that
$
\|\mathfrak z_1 \ast \mathfrak z_2\|_{\mathscr U} = \|\mathfrak z_1\|_{\mathscr U} + \| \mathfrak z_2\|_{\mathscr U}$ for all $\mathfrak z_1, \mathfrak z_2 \in \mathscr F^\ast(\mathscr U)$.

If $a,b \in \mathbf{R} \cup \{\pm \infty\}$, we let $\llb a, b \rrb := \{x \in \mathbf Z: a \le x \le b \}$
stand for the (discrete) interval between $a$ and $b$.
If $\lambda \in \mathbf R$ and $X, Y \subseteq \mathbf R$, we denote by $X^+$ the positive part of $X$ (so, $\mathbf N^+$ is the set of positive integers), and we define the sumset of $X$ and $Y$ by $X + Y := \{x+y: (x,y) \in X \times Y\}$, the $n$-fold sumset of $X$ by $nX := \{x_1 + \cdots + x_n: x_1, \ldots, x_n \in X\}$, and the $\lambda$-dilation of $X$ by
$\lambda \cdot \fixed[-0.3]{\text{ }} X := \{\lambda x: x \in X\}$.

If $X$, $Y$, and $Z$ are sets and $\mathscr{C}$ is an equivalence (relation) on $X$, we denote by $\mathcal P(X)$ the power set of $X$ and by $\llb x \rrb_{\mathscr{C}}$ the (equivalence) class of a fixed element $x \in X$ in the quotient $X/\mathscr{C}$, and we write $X = Y \uplus Z$ to mean that $Y \cap Z = \emptyset$ and $X = Y \cup Z$.

We say that a finite sequence $x_1, \ldots, x_n$ is the natural enumeration of a non-empty set $X \subseteq \mathbf R$ if $X = \{x_1, \ldots, x_n\}$ and $x_i < x_{i+1}$ for every $i \in \llb 1, n-1 \rrb$. Lastly, we assume $\sup(\emptyset) := 0$ and $\inf(\emptyset) := \infty$, and we let $\mathfrak{S}_n$ be the group of permutations of $\llb 1, n \rrb$.

Further notations and terminology, if not explained, are standard or should be clear from the context.
\section{Factorization theory}
\label{sec:monoids}
In this section, we fix some definitions that are at the center of our interest, and we prove some fundamental results that will be used later (in \S\S{} \ref{sec:power_monoid} and \ref{sec:the_case_of_integers}) to investigate the structure of power monoids.
\subsection{Basic definitions and arithmetic invariants.}
\label{subsec:basic}
Throughout, we let $H$ be a (multiplicatively written) monoid with identity $1_H$, and we denote by $H^\times$ the \textit{set of units} (or \textit{invertible elements}) of $H$.

Note that $H$ need not have any special property (e.g., commutativity), unless a statement to the contrary is made. Also, we will systematically drop the subscript `$H$\fixed[0.2]{\text{ }}' from the notations we are going to introduce whenever $H$ is implied from the context and there is no risk of ambiguity.

We say that $H$ is \textit{reduced} if $H^\times = \{1_H\}$; \textit{cancellative} if $xz = yz$ or $zx = zy$, for some $x, y, z \in H$, implies $x = y$; \textit{Dedekind-finite} if $xy = 1_H$ yields $yx = 1_H$; \textit{unit-cancellative} (respectively, \textit{strongly unit-cancellative}) provided that $xy = x$ or $yx = x$ only if $y \in H^\times$ (respectively, $y = 1_H$); \textit{divisible} if, for all $n \in \mathbf N^+$ and $x \in H$, there exists $y \in H$ with $x = y^n$; and \textit{non-torsion} if $\ord_H(x) = \infty$ for some $x \in H$, where $\ord_H(x)$ is the \textit{order} of $x$ (in $H$), i.e., the cardinality of the set $\{x^n: n \in \mathbf N^+\}$.
\begin{remark}
Factorization in unit-cancellative monoids is the subject of recent work by Geroldinger, Kainrath, and the authors in the commutative and finitely generated case \cite{FGKT}, and by Geroldinger and Schwab in the finitely presented case \cite{GeSc17}. In turn, Dedekind-finite monoids, sometimes also referred to as directly finite, weakly $1$-finite, inverse symmetric, or von Neumann-finite monoids,
form a fairly large class, which includes, among many others, the multiplicative monoid of
Artinian or Noetherian rings \cite[Proposition 4.6.6 and Theorem 4.6.7(iii)]{Cohn05},
algebraic algebras over a field \cite[Exercise 1.13]{Lam}, right and left self-injective rings \cite[Corollary 1.1]{Fa03}, and the group ring of a (possibly non-abelian or infinite) group over a field of characteristic zero \cite[Theorem 2.3]{Fa03},
not to mention trivial examples such as commutative or cancellative monoids (and submonoids, direct products, and direct limits of Dedekind-finite monoids).

Both unit-cancellative and Dedekind-finite monoids play a central role in the present paper, though most of the basic definitions and results are worked out in greater generality at no additional cost.
Of course, all cancellative monoids are strongly unit-cancellative, and the latter are unit-cancellative: What is slightly less obvious is that unit-cancellative monoids are Dedekind-finite (Proposition \ref{prop:Dedekind-finite}).
\end{remark}
Given $x, y \in H$, we
write $x \mid_H y$ if $uxv = y$ for some $u, v \in H$, cf. \cite[Definition 5.2(1)]{BaSm}. Moreover, we use $x \simeq_H y$, and we say that $x$ is \textit{associate} to $y$, if $y \in H^\times x H^\times$. Lastly, we take a submonoid $M$ of $H$ to be \textit{divisor-closed} if $x \in M$ whenever $x \mid_H y$ and $y \in M$.
\subsubsection{Atoms and lengths}
\label{subsub:sets-of-lengths}
We let $\mathscr{A}(H)$ stand for the \textit{set of atoms} (or \textit{irreducible elements}) of $H$, where $a \in H$ is an atom if $a \notin H^\times$ and there do not exist $x, y\in H \setminus H^\times$ such that $a = xy$ (note that, in general, the product of two non-units can be a unit, so the first condition does not follow from the second, cf. Lemma \ref{lem:products_and_units}\ref{it:lem:products_and_units(i)} and Proposition \ref{prop:Dedekind-finite}).

We set, for every $x \in H$, $\LLs_H(x) := \{k \in \mathbf N^+: x = a_1 \cdots a_k\text{ for some }a_1, \ldots, a_k \in \mathscr{A}(H)\}$ if $x \ne 1_H$, and $\LLs_H(x) := \{0\} \subseteq \mathbf N$ otherwise: We call an element of $\LLs_H(x)$ a \textit{\textup{(}factorization\textup{)} length} of $x$, and $\LLs_H(x)$ the \textit{set of lengths} of $x$.
Consequently, we say that $H$ is \textit{atomic} (respectively, a \textit{{\rm BF}-monoid}) if $\LLs_H(x)$ is non-empty (respectively, non-empty and finite) for all $x \in H \setminus H^\times$.
\begin{lemma}
\label{lem:basic-properties-atoms-units}
Let $H$ be a monoid. The following hold:
\begin{enumerate}[label={\rm (\roman{*})}]
\item\label{it:lem:basic-properties-atoms-units(i)} If $u, v \in H^\times$, then $uv \in H^\times$. Moreover, the converse is true if $H = H^\times$ or $\mathscr{A}(H) \ne \emptyset$.
\item\label{it:lem:basic-properties-atoms-units(ii)} If $a \in \mathscr{A}(H)$ and $u \in H^\times$, then $u \fixed[0.05]{\text{ }} a, au \in \mathscr{A}(H)$.
\item\label{it:lem:basic-properties-atoms-units(iii)} ${\sf L}_H(u) = \emptyset$ for every $u \in H^\times \setminus \{1_H\}$.
\item\label{it:lem:basic-properties-atoms-units(iv)} ${\sf L}_H(x) = {\sf L}_H(uxv)$ for all $x \in H \setminus H^\times$ and $u, v \in H^\times$.
\end{enumerate}
\end{lemma}
\begin{proof}
\ref{it:lem:basic-properties-atoms-units(i)} The first part is trivial and well known; in particular, if $u, v \in H^\times$, then $uv$ is invertible and $(u \fixed[0.1]{\text{ }} v)^{-1} = v^{-1} u^{-1}$. As for the converse, the claim is obvious if $H = H^\times$. Otherwise, pick $a \in \mathscr{A}(H)$ and suppose for a contradiction that there are $x, y \in H$ such that $xy \in H^\times$, but $x \notin H^\times$ or $y \notin H^\times$. We can assume (by symmetry) that $x \notin H^\times$.
Then $xyz = 1_H$ for some $z \in H$, which yields $a=x(yza)$. So $yza$ must be a unit, since $x$ is not and $a$ is an atom. In particular, $yzav=1_H$ for some $v \in H$. This shows that $yz$ is both left- and right-invertible, hence is invertible.
It follows $x=(yz)^{-1} \in H^\times$, a contradiction.

\ref{it:lem:basic-properties-atoms-units(ii)} Let $a \in \mathscr{A}(H)$ and $u \in H^\times$. We will prove that $a \fixed[0.05]{\text{ }} u$ is an atom (the other case is similar). Indeed, $a \fixed[0.05]{\text{ }} u$ is not a unit: Otherwise, $a = vu^{-1}$ for some $v \in H^\times$, which would imply, say, by point \ref{it:lem:basic-properties-atoms-units(i)} that $a \in H^\times$, a contradiction. Moreover, if $au = xy$ for some $x, y \in H$, then $a = x(yu^{-1})$. So, using that $a$ is an atom, we have $x \in H^\times$, or $yu^{-1} = v$ for some $v \in H^\times$ (and hence $y = vu \in H^\times$). To wit, $a \fixed[0.05]{\text{ }} u \in \mathscr{A}(H)$.

\ref{it:lem:basic-properties-atoms-units(iii)} It is evident that, if $\mathscr{A}(H)$ is empty, then so is ${\sf L}_H(x)$ for every $x \in H \setminus \{1_H\}$, and we are done. Otherwise, we get from point \ref{it:lem:basic-properties-atoms-units(i)} that the units of $H$ cannot be factored into a non-empty product of atoms of $H$, and consequently ${\sf L}_H(u) = \emptyset$ for every $u \in H^\times \setminus \{1_H\}$.

\ref{it:lem:basic-properties-atoms-units(iv)} Let $x \in H \setminus H^\times$ and $u, v \in H^\times$. Since $x = u^{-1}(uxv) v^{-1}$, it is sufficient to prove that $\mathsf L_H(x) \subseteq \mathsf L_H(uxv)$. If $\mathsf L_H(x)$ is empty, this is obvious. Otherwise, pick $k \in \mathsf L_H(x)$. Then $k \ge 1$ (because $x \ne 1_H$), and there exist $a_1, \ldots, a_k \in \mathscr{A}(H)$ such that $x = a_1 \cdots a_k$. It follows $uxv = b_1 \cdots b_k$, with $b_i := ua_i u^{-1}$ for $i \in \llb 1, k-1 \rrb$ and $b_k := ua_n v$. So we conclude from \ref{it:lem:basic-properties-atoms-units(ii)} that $k \in \mathsf L_H(uxv)$, and we are done.
\end{proof}
\begin{remark}
\label{rem:where-is-0}
It is perhaps worth noting that $0 \in {\sf L}_H(x)$ for some $x \in H$ only if $x = 1_H$, in contrast to the standard convention that the set of lengths of \textit{any} unit of $H$ is equal to $\{0\}$. As a matter of fact, we disagree with this convention,
since it looks no longer fit for the non-commutative setting
(cf. Remark \ref{remark:comparison-factorization-monoids}), and all the more in the light of Lemma \ref{lem:basic-properties-atoms-units}\ref{it:lem:basic-properties-atoms-units(iii)}.
\end{remark}
\begin{remark}
\label{rem:factorizations-of-1H}
By Lemma \ref{lem:basic-properties-atoms-units}\ref{it:lem:basic-properties-atoms-units(iii)}, $1_H$ cannot be expressed as a non-empty product of atoms of $H$. This yields that, for all $x, y \in H$,
$\LLs_H(x) + \LLs_H(y) \subseteq \LLs_H(xy)$
and $\sup \LLs_H(x) + \sup \LLs_H(y) \le \sup \LLs_H(xy)$.
\end{remark}
We let $\LLc(H) := \{\LLs_H(x): x \in H\} \setminus \{\emptyset\} \subseteq \PPc(\mathbf N)$. We refer to $\LLc(H)$ as the \textit{system of sets of lengths} of $H$.
Then, for each $k \in \mathbf N$ we denote by $\UUc_k(H)$ the union of all $L \in \mathscr{L}(H)$ with $k \in L$.
It is clear that $\mathscr{U}_0(H) = \{0\}$; and if $\mathscr{A}(H)$ is non-empty then $\mathscr{U}_1(H) = \{1\}$ and $k \in \mathscr{U}_k(H)$ for all $k \in \mathbf N$, otherwise $\mathscr{U}_1(H) = \mathscr{U}_2(H) = \cdots = \emptyset$.

We take $\Delta(H) := \bigcup_{L \in \LLc(H)} \Delta(L)$, where for $L \subseteq \mathbf{Z}$ we let $\Delta(L)$ be the
set of all $d \in \mathbf N^+$ such that $L \cap \llb l, l+d\fixed[0.2]{\text{ }} \rrb = \{l, l+d\}$ for some $l \in L$. We call $\Delta(H)$ the \textit{set of distances} (or \textit{delta set}) of $H$.

Sets of lengths, along with a number of invariants derived from them (e.g., unions of sets of lengths and sets of distances), are by and large the best tools so far available to describe the arithmetic of \BF-monoids, see \cite{Ge16c} for further discussion on this point.

\subsubsection{Factorizations}
\label{sec:factorizations}
We let $\pi_H$ be the unique monoid homomorphism $\mathscr{F}^\ast(H) \to H$ such that $\pi_H(x) = x$ for all $x \in H$,
and $\mathscr{C}_H$ the smallest monoid con\-gru\-ence on $\mathscr{F}^\ast(\mathscr{A}(H))$ for which the following holds:
\begin{enumerate}[label={$\bullet$}]
\item If $\mathfrak a = a_1 \ast \cdots \ast a_m$ and $\mathfrak b = b_1 \ast \cdots \ast b_n$ are, re\-spec\-tive\-ly, non-empty $\mathscr{A}(H)$-$\fixed[0.2]{\text{ }}$words of length $m$ and $n$, then $(\mathfrak a, \mathfrak b) \in \mathscr{C}_H$ if and only if $\pi_H(\mathfrak a) = \pi_H(\mathfrak b)$, $m = n$, and $a_1 \simeq_H b_{\sigma(1)}, \ldots, a_n \simeq_H b_{\sigma(n)}$ for some permutation $\sigma \in \mathfrak S_n$.
\end{enumerate}
We call $\pi_H$ the \textit{factorization homomorphism} of $H$, and the quotient $\mathsf Z(H) := \mathscr{F}^\ast(\mathscr{A}(H))/\mathscr{C}_H$ the \textit{factorization monoid} of $H$. We continue denoting the operation of $\mathsf{Z}(H)$ by the same symbol as the operation of $\mathscr{F}^\ast(\mathscr{A}(H))$, and we observe that,
if $H$ is a reduced commutative monoid and $\mathfrak a = a_1 \ast \cdots \ast a_n$ is a non-empty $\mathscr{A}(H)$-word of length $n$, then
$$
\llb \mathfrak a \rrb_{\mathscr{C}_H} = \bigl\{a_{\sigma(1)} \ast \cdots \ast a_{\sigma(n)}: \sigma \in \mathfrak{S}_n\bigr\}.
$$
Accordingly, we abuse notation and identify $\llb \mathfrak a \rrb_{\mathscr{C}_H}$ with $\mathfrak a$ whenever $H$ is commutative and $H^\times = \{1_H\}$.

Also, we notice that $\pi_H(\mathcal A) = \{\pi_H(\mathfrak a)\}$ for all $\mathcal A \in \mathsf Z(H)$ and $\mathfrak a \in \mathcal A$, and we define, for every $x \in H$,
$$
\mathsf{Z}_H(x) := \fixed[-0.15]{\text{ }}\bigl\{\llb \mathfrak a \rrb_{\mathscr{C}_H}: \mathfrak a \in \mathscr{F}^\ast(\mathscr{A}(H)) \text{ and } \pi_H(\mathfrak a) = x\bigr\} \fixed[-0.15]{\text{ }} \subseteq \mathsf{Z}(H)
$$
and
\begin{equation*}
\label{equ:factorizations-of-x}
\mathcal{Z}_H(x) := \pi_H^{-1}(x) \cap \mathscr{F}^\ast(\mathscr{A}(H)) = \bigcup \mathsf{Z}_H(x) \subseteq \mathscr{F}^\ast(\mathscr{A}(H)).
\end{equation*}
From here it is easy to see that
\begin{equation}
\label{equ:relation-lengths-factorization-monoid}
{\sf L}_H(x) = \bigl\{\|\mathfrak a\|_H: \mathfrak a \in \mathcal{Z}_H(x)\bigr\}, \text{ for all } x \in H.
\end{equation}
We refer to the elements of $\mathsf Z_H(x)$ as the \textit{factorization classes} of $x$, and to the $\mathscr{A}(H)$-words in $\mathcal{Z}_H(x)$ as the \textit{factorizations} of $x$. Then we have the following:
\begin{lemma}\label{lem:products-of-factorizations}
	Let $H$ be a monoid, and pick $x \in H \setminus \mathscr{A}(H)$ such that $x \ne 1_H$. Then
	\begin{equation}\label{equ:Z=Z'}
	\mathcal{Z}_H(x) = \bigcup_{y,\fixed[0.1]{\text{ }} z \fixed[0.15]{\text{ }}\in H \setminus H^\times:\fixed[0.25]{\text{ }} x \fixed[0.15]{\text{ }} = \fixed[0.15]{\text{ }} yz} \{\mathfrak{a} \ast \mathfrak{b}: (\mathfrak{a}, \mathfrak{b}) \in \mathcal{Z}_H(y) \times \mathcal{Z}_H(z)\}.
	\end{equation}
\end{lemma}
\begin{proof}
	Let $\mathcal{Z}_H^\prime(x)$ denote the set on the right-hand side of equation \eqref{equ:Z=Z'}. It is clear that $\mathcal{Z}_H^\prime(x) \subseteq \mathcal{Z}_H(x)$. As for the opposite inclusion, this is obvious if $\mathcal{Z}_H(x) = \emptyset$. Otherwise, $\mathcal{Z}_H(x)$ is a non-empty subset of $\mathscr{F}^\ast(\mathscr{A}(H)) \setminus \bigl\{1_{\mathscr{F}^\ast(\mathscr{A}(H))}\bigr\}$. Accordingly, let $\mathfrak a := a_1 \ast \cdots \ast a_n \in \mathcal{Z}_H(x)$. Then $n \ge 2$ (since $x$ is not an atom), and hence $\mathfrak a = \mathfrak b \ast \mathfrak c$, where $\mathfrak b := a_1 \ast \cdots \ast a_{n-1}$ and $\mathfrak c := a_n$ are non-empty $\mathscr{A}(H)$-words. But this implies $\mathfrak a \in \mathcal{Z}_H^\prime(x)$, because it is evident from the above that $\mathscr A(H)$ is non-empty, and therefore $y := \pi_H(\mathfrak b)$ and $z := \pi_H(\mathfrak c)$ are non-units of $H$, by Lemma \ref{lem:basic-properties-atoms-units}\ref{it:lem:basic-properties-atoms-units(i)}.
\end{proof}
Now we take a break for some highlights, to put things in perspective and contrast our approach to the study of the arithmetic of monoids with what has been done so far in the existing body of literature.
\begin{remark}
\label{remark:comparison-factorization-monoids}
Our definition of the factorization monoid $\mathsf Z(H)$ is, in general, inconsistent with
analogous definitions from the literature on factorization theory, and it is probably useful to explain why
this inconsistency is not necessarily bad.

Our terms for comparison will be the classical definition of the factorization monoid
(for the case when $H$ is commutative and cancellative) and Smertnig's definition of the monoid of rigid factorizations (for cancellative monoids), for which we use, respectively, the notation $\mathsf Z_{\rm GeH}(H)$ and ${\sf Z}_{\rm Sm}(H)$, and we refer, respectively, to \cite[Definition 1.2.6]{GeHK06} and \cite[\S{ }3]{Sm13} (see also Remarks \ref{rem:permutable_factorizations} and \S{ }\ref{sec:distances}).

To start with, it is worth stressing that a ``full comparison'' between $\mathsf Z_{\rm GeH}(H)$ and $\mathsf Z(H)$, whatever it may mean, is just impossible. Not only because $\mathsf Z_{\rm GeH}(H)$ is not defined for non-commutative monoids (cancellativity has no active role in this regard, see \cite[\S{ }3]{FGKT}),
but also, and more importantly, because there seems to be no meaningful way to carry over the definition of $\mathsf Z_{\rm GeH}(H)$ to a non-commutative setting: $\mathsf Z_{\rm GeH}(H)$ is the free \textit{abelian} monoid with basis $\mathscr{A}(H_{\rm red})$, where $H_{\rm red}$ is the quotient $H/H^\times$. Thus, a naive attempt to generalize the classical definition to the case when $H$ may not be commutative, would be to take the quotient of $H$ by the monoid congruence $\mathscr{C}_{\rm red}$ generated by the relation $\simeq_H$ and to let the factorization monoid of $H$ equal to $\mathscr{F}^\ast(\mathscr{A}(H / \mathscr{C}_{\rm red}))$.
But this approach has a major drawback: If $H$ is commutative, then $\mathscr{C}_{\rm red}$ and $\simeq_H$ coincide. Otherwise, $\simeq_H$ need not be a congruence and $\mathscr{C}_{\rm red}$ can be ``much larger'' than $\simeq_H$,
with the result that $H /\mathscr{C}_{\rm red}$ is ``too small'' for carrying any interesting information about the arithmetic of $H$ (cf. \cite[Remarks 3.3.1]{Sm13}).

In a similar vein, a full comparison between $\mathsf Z_{\rm Sm}(H)$ and $\mathsf Z(H)$ is also unfeasible, since the definition of $\mathsf Z_{\rm Sm}(H)$ is phrased in the language of categories, while the present paper is entirely focused on monoids (though a large part of this section can be abstracted to the level of categories without much trouble).

So, we have no choice but to restrict the comparison between $\mathsf Z_{\rm GeH}(H)$ and $\mathsf Z(H)$ to the commutative setting, and the comparison between $\mathsf Z_{\rm Sm}(H)$ and $\mathsf Z(H)$ to the case when the former is specialized to monoids (no further comment will be made on this point in the sequel).
\vskip 0.1cm
\textit{Round 1:} $\mathsf Z_{\rm GeH}(H)$ vs $\mathsf Z(H)$. Assume that $H$ is commutative, and denote by $\mathscr{C}_H^{\fixed[0.2]{\text{ }}\prime}$
  the smallest monoid con\-gru\-ence on $\mathscr{F}^\ast(\mathscr{A}(H))$ for which the following holds:
  \begin{itemize}
  \item If $\mathfrak a = a_1 \ast \cdots \ast a_m$ and $\mathfrak b = b_1 \ast \cdots \ast b_n$ are, re\-spec\-tive\-ly, non-empty $\mathscr{A}(H)$-words of length $m$ and $n$, then $(\mathfrak a, \mathfrak b) \in \mathscr{C}_H^{\fixed[0.2]{\text{ }}\prime}$ if and only if $m = n$ and $a_1 \simeq_H b_{\sigma(1)}, \ldots, a_n \simeq_H b_{\sigma(n)}$ for some $\sigma \in \mathfrak{S}_n$.
  \end{itemize}
  It is readily checked that $\mathsf Z_{\rm GeH}(H)$ is isomorphic (as a monoid) to the quotient $\mathsf Z_{\rm GeH}^{\fixed[0.2]{\text{ }} \prime}(H) := \mathscr{F}^\ast(\mathscr{A}(H))/\mathscr{C}_H^{\fixed[0.2]{\text{ }}\prime}$. Therefore, rather than comparing $\mathsf Z(H)$ with $\mathsf Z_{\rm GeH}(H)$, we may compare the former with $\mathsf Z_{\rm GeH}^{\fixed[0.2]{\text{ }} \prime}(H)$, which has practical advantages.

  In particular, there is a unique homomorphism $\pi_{\rm GeH}: \mathsf Z_{\rm GeH}^{\fixed[0.2]{\text{ }} \prime}(H) \to H_{\rm red}$ such that $\pi_{\rm GeH}\bigl(\llb a \rrb_{\mathscr{C}_H^{\fixed[0.2]{\text{ }}\prime}}\bigr) = aH^\times$ for all $a \in \mathscr{A}(H)$, and for every $x \in H$ we can identify the elements of the set
  $$
  \mathsf{Z}_{\rm GeH}^{\fixed[0.2]{\text{ }} \prime}(x) := \pi_{\rm GeH}^{-1}(xH^\times) \subseteq \mathsf Z_{\rm GeH}^{\fixed[0.2]{\text{ }}\prime}(H)
  $$
  with the factorizations of $x$ in the sense of \cite[Definition 1.2.6]{GeHK06}. So, taking
  \begin{equation*}
  \label{equ:classical-factorizations-of-x}
  \mathcal{Z}_{\rm GeH}^{\fixed[0.2]{\text{ }}\prime}(x) := \bigcup \mathsf{Z}_{\rm GeH}(x) \subseteq \mathscr{F}^\ast(\mathscr{A}(H))
  \end{equation*}
  and calling the $\mathscr{A}(H)$-words in $\mathcal{Z}_{\rm GeH}^{\fixed[0.2]{\text{ }}\prime}(x)$ the \textit{classical factorizations} of $x$, we end up with the conclusion that, in the multiplicative monoid of the ring of integers, the $\mathscr{A}(\mathbf P)$-words $2 \ast (-3)$ and $2 \ast 3$, where $\mathbf P$ is the set of rational primes, are both classical factorizations of $6$. Of course, there is nothing wrong or paradoxical with this inference (it is just the consequence of some definitions), though we do not find it very natural and nothing similar happens with our definitions.

  Indeed, $\mathcal{Z}_H(1_H) = \mathcal{Z}_{\rm GeH}^{\fixed[0.2]{\text{ }}\prime}(x) = \fixed[-0.2]{\text{ }} \bigl\{1_{\mathscr{F}^\ast(\mathscr{A}(H))}\bigr\}$ for $x \in H^\times$, and $\mathcal{Z}_H(x) = \emptyset$ for $x \in H^\times \setminus \{1_H\}$. Also, if $\mathfrak a = a_1 \ast \cdots \ast a_n$ is a non-empty $\mathscr{A}(H)$-word of length $n$ and $x = \pi_H(\mathfrak a)$, then
  \begin{equation}
  \label{equ:congr-class-commutative-our-approach}
  \llb \mathfrak a \rrb_{\mathscr{C}_H} = \fixed[-0.2]{\text{ }} \bigl\{\fixed[-0.1]{\text{ }}\bigl(a_{\sigma(1)} u_1\bigr)\fixed[-0.1]{\text{ }} \ast \cdots \ast \fixed[-0.1]{\text{ }}\bigl(a_{\sigma(n)} u_n\bigr)\fixed[-0.1]{\text{ }}: \sigma \in \mathfrak{S}_n,\ u_1, \ldots, u_n \in H^\times, \text{ and }u_1 \cdots u_n x = x\fixed[-0.1]{\text{ }}\bigr\}
  \end{equation}
  and
  \begin{equation*}
  \label{equ:congr-class-commutative-classical-approach}
  \llb \mathfrak a \rrb_{\mathscr{C}_H^\prime} = \fixed[-0.2]{\text{ }} \bigl\{\fixed[-0.1]{\text{ }}\bigl(a_{\sigma(1)} u_1\bigr)\fixed[-0.1]{\text{ }} \ast \cdots \ast \fixed[-0.1]{\text{ }}\bigl(a_{\sigma(n)} u_n\bigr)\fixed[-0.1]{\text{ }}: \sigma \in \mathfrak{S}_n \text{ and }u_1, \ldots, u_n \in H^\times\fixed[-0.1]{\text{ }}\bigr\}\fixed[0.1]{\text{ }}.
  \end{equation*}
  It follows that $\llb \mathfrak a \rrb_{\mathscr{C}_H} \subseteq \llb \mathfrak a \rrb_{\mathscr{C}_H^\prime}$, and the inclusion is strict if, for instance, $H$ is strongly unit-cancellative, but not reduced. The point is simply that $\mathscr{C}_H \subseteq \mathscr{C}_H^{\fixed[0.2]{\text{ }}\prime}$, and in general we do not have equality.

  In other terms, ${\sf Z}_{\rm GeH}^{\fixed[0.2]{\text{ }} \prime}(H)$ is ``coarser'' than ${\sf Z}(H)$,
  in the sense that the former embeds (as a monoid) into the latter, but the embedding is an isomorphism if and only if $H$ is reduced.
\vskip 0.1cm
\textit{Round 2:} $\mathsf Z_{\rm Sm}(H)$ vs $\mathsf Z(H)$. In the case of monoids, Smertnig's definition of ${\sf Z}_{\rm Sm}(H)$ comes down to the following: Denote by $\circ$ the binary operation on the set $\mathscr{F}_{\rm Sm}(H) := H^\times \fixed[-0.1]{\text{ }} \times \mathscr{F}^\ast(\mathscr{A}(H))$ given by
\begin{equation*}
((u, \mathfrak a), (v, \mathfrak b)) \mapsto
\left\{
\begin{array}{ll}
\!\! (u, \mathfrak a \fixed[0.2]{\text{ }} a_n^{-1} \ast (a_nv) \ast \mathfrak b)
  & \text{if } n := \|\mathfrak a\|_H \ge 1 \text{ and }\mathfrak a = a_1 \ast \cdots \ast a_n \\
\!\! (uv, \mathfrak b) & \text{otherwise}
\end{array}
\right.\!\!\!,
\end{equation*}
which is well defined by Lemma \ref{lem:basic-properties-atoms-units}\ref{it:lem:basic-properties-atoms-units(ii)} (Smertnig's original definition is restricted to the can\-cel\-la\-tive setting, where the well-definedness of $\circ$ is trivial).
Note that the pair $(\mathscr{F}_{\rm Sm}(H), \circ\fixed[0.1]{\text{ }})$ is a monoid. Accordingly, let $\mathscr{C}_{\rm Sm}$ be the smallest monoid congruence on $(\mathscr{F}_{\rm Sm}(H), \circ \fixed[0.1]{\text{ }})$ determined by the following:
\begin{itemize}
\item If $u, v \in H^\times$, and $\mathfrak a = a_1 \ast \cdots \ast a_m$ and $\mathfrak b = b_1 \ast \cdots \ast b_n$ are, respectively, non-empty $\mathscr{A}(H)$-words of length $m$ and $n$, then $((u, \mathfrak a), (v, \mathfrak b)) \in \mathscr{C}_{\rm Sm}$ if and only if $u \fixed[0.2]{\text{ }} \pi_H(\mathfrak a) = v \fixed[0.2]{\text{ }} \pi_H(\mathfrak b)$, $m = n$, and there exist $\varepsilon_1, \ldots, \varepsilon_n \in H^\times$ with $\varepsilon_n = 1_H$, $u \fixed[0.2]{\text{ }} a_1 = v \fixed[0.15]{\text{ }} b_1 \fixed[0.1]{\text{ }} \varepsilon_1^{-1}$, and $a_i = \varepsilon_{i-1} b_i \fixed[0.15]{\text{ }} \varepsilon_i^{-1}$ for $i \in \llb 2, n \rrb$.
\end{itemize}
In fact, $\mathsf{Z}_{\rm Sm}(H)$ is the quotient of $(\mathscr{F}_{\rm Sm}(H),\circ \fixed[0.1]{\text{ }})$ by the congruence $\mathscr{C}_{\rm Sm}$. In particular, if $H$ is reduced, then $\mathsf{Z}_{\rm Sm}(H) \cong \mathscr{F}^\ast(\mathscr{A}(H))$. So, contrary to what happens with $\mathsf Z(H)$, $\mathsf Z_{\rm Sm}(H)$ is not even isomorphic to $\mathsf Z_{\rm GeH}(H)$ when $H$ is reduced and commutative, cf. \cite[p. 492]{BaSm}.
Nevertheless, there are strong sim\-i\-lar\-i\-ties between the constructions of $\mathsf{Z}_{\rm Sm}(H)$ and $\mathsf{Z}(H)$, which will be further clarified by Remark \ref{rem:permutable_factorizations}.

First, both constructions involve, through the definition of the congruences $\mathscr{C}_{\rm Sm}$ and $\mathscr{C}_H$, a con\-di\-tion (in terms of the homomorphism $\pi_H$) that rules out the ``issues'' pointed out in the above in reference to the classical factorizations in the commutative setting.

Secondly, both agree on the role of $\mathscr{F}^\ast(\mathscr{A}(H))$ and the idea that factorizations, whatever they may be, are related to the quotient of $\mathscr{F}^\ast(\mathscr{A}(H))$, or something as close to $\mathscr{F}^\ast(\mathscr{A}(H))$ as $\mathscr{F}_{\rm Sm}(H)$, by a suitable congruence. But while $\mathsf{Z}_{\rm Sm}(H)$ brings in ``the $H^\times$ factor {(...)} to represent trivial factorizations of units'' (to quote Smertnig's own words from \cite[Remark 3.3.1]{Sm13}), we brush off the trivial factorizations of a unit $u \ne 1_H$ from our approach: 
This leads to a simplification of the theory,
without causing any significant loss (cf. Remarks \ref{rem:where-is-0} and \ref{rem:factorizations-of-1H}).
\end{remark}
\begin{remark}
\label{rem:permutable_factorizations}
The factorization monoid $\mathsf Z(H)$ is \textit{essentially} the same as Baeth and Smertnig's monoid, $\mathsf{Z}_p(H)$, of permutable factorizations: This may not be immediately apparent, but it follows from Lemma \ref{lem:basic-properties-atoms-units}\ref{it:lem:basic-properties-atoms-units(ii)} and a careful reading of \cite[Construction 3.3(2), Definitions 3.4(2) and 3.8(2), and Remark 3.9(2)]{BaSm}. In the notations and terminology of Remark \ref{remark:comparison-factorization-monoids}, $\mathsf{Z}_p(H)$ is, in fact, the quotient of $\mathsf Z_{\rm Sm}(H)$ by the smallest monoid congruence $\sim_p$ on $\mathsf Z_{\rm Sm}(H)$ for which the following holds:
\begin{itemize}
\item If $u, v \in H^\times$, and $\mathfrak a = a_1 \ast \cdots \ast a_m$ and $\mathfrak b = b_1 \ast \cdots \ast b_n$ are non-empty $\mathscr{A}(H)$-words of length $m$ and $n$, respectively, then $\llb (u, \mathfrak a) \rrb_{\mathscr{C}_{\rm Sm}} \sim_p \llb (v, \mathfrak b) \rrb_{\mathscr{C}_{\rm Sm}}$ if and only if $u \fixed[0.2]{\text{ }} \pi_H(\mathfrak a) = v \fixed[0.2]{\text{ }} \pi_H(\mathfrak b)$, $m = n$, and there exists $\sigma \in \mathfrak S_n$ such that $a_i \simeq_H b_{\sigma(i)}$ for all $i \in \llb 1, n \rrb$.
\end{itemize}
Consequently, $\mathsf Z(H)$ is monoid isomorphic to $\mathsf Z_p(H) \setminus \mathcal C$, where $\mathcal C$ is the the set of all congruence classes in $\mathsf Z_p(H)$ corresponding to a rigid factorization of the form $\bigl\llb \bigl(u, 1_{\mathscr{F}^\ast(\mathscr{A}(H))}\bigr) \bigr\rrb_{\mathscr{C}_{\rm Sm}}$ with $u \in H^\times \setminus \{1_H\}$. In particular, we have a monoid isomorphism between $\mathsf Z(H)$ and $\mathsf Z_p(H)$ if and only if $H$ is reduced. 
\end{remark}
\subsubsection{Distances and catenary degree}
\label{sec:distances}
Let $\mathsf{d}$ be a function $\mathscr{F}^\ast(\mathscr{A}(H)) \times \mathscr F^\ast(\mathscr{A}(H)) \to \mathbf R$. We say that $\sf d$ is a \textit{\textup{(}global\textup{)} distance} (on $H$) if, for all $\mathfrak a, \mathfrak b, \mathfrak c \in \mathscr{F}^\ast(\mathscr{A}(H))$, the following hold:
\begin{enumerate}[label={\rm (\textsc{d}\arabic{*})}]
	\item\label{it:distance-definition(i)} $\mathsf d(\mathfrak a, \mathfrak b) = 0$ whenever $(\mathfrak a, \mathfrak b) \in \mathscr{C}_H$.
	\item\label{it:distance-definition(ii)} $\mathsf d(\mathfrak a, \mathfrak b) = \mathsf d(\mathfrak b, \mathfrak a)$.
	\item\label{it:distance-definition(iii)} $\mathsf d(\mathfrak a, \mathfrak b) \le \mathsf d(\mathfrak a, \mathfrak c) + \mathsf d(\mathfrak c, \mathfrak b)$.
	\item\label{it:distance-definition(v)} $\bigl|\|\mathfrak a\|_H - \|\mathfrak b\|_H\bigr| \fixed[-0.1]{\text{ }} \le \mathsf d(\mathfrak a, \mathfrak b) \le \max\bigl(\|\mathfrak a\|_H, \|\mathfrak b\|_H\bigr)$.
\end{enumerate}
We refer to $\mathsf d$ as a $\mathscr{C}_H$-\textit{metric} if it is a distance and, in addition, $\mathsf d(\mathfrak a, \mathfrak b) = 0$ for some $\mathfrak a, \mathfrak b \in \mathscr F^\ast(\mathscr{A}(H))$ implies $(\mathfrak a, \mathfrak b) \in \mathscr{C}_H$. Moreover, we call $\mathsf d$ \textit{subinvariant} (on $H$) if we have:
\begin{enumerate}[label={\rm (\textsc{d}\arabic{*})}, resume]
	\item\label{it:distance-definition(iv)} $\mathsf d(\mathfrak c \ast \mathfrak a \ast \mathfrak d, \mathfrak c \ast \mathfrak b \ast \mathfrak d) \le \mathsf d(\mathfrak a, \mathfrak b)$ for all $\mathfrak a, \mathfrak b, \mathfrak c, \mathfrak d \in \mathscr F^\ast(\mathscr{A}(H))$.
\end{enumerate}
In a similar vein, we say that $\mathsf d$ is \textit{invariant} (on $H$) if \ref{it:distance-definition(iv)} holds with equality, namely:
\begin{enumerate}[label={\rm (\textsc{d}\arabic{*})}, resume]
	\item\label{it:distance-definition(vi)} $\mathsf d(\mathfrak c \ast \mathfrak a \ast \mathfrak d, \mathfrak c \ast \mathfrak b \ast \mathfrak d) = \mathsf d(\mathfrak a, \mathfrak b)$ for all $\mathfrak a, \mathfrak b, \mathfrak c, \mathfrak d \in \mathscr F^\ast(\mathscr{A}(H))$.
\end{enumerate}
Lastly, we take $\mathsf d$ to be \textit{locally invariant} (on $H$) if it is subinvariant and
\begin{enumerate}[label={\rm (\textsc{d}\arabic{*})}, resume]
	\item\label{it:distance-definition(vii)} $\mathsf d(\mathfrak c \ast \mathfrak a \ast \mathfrak d, \mathfrak c \ast \mathfrak b \ast \mathfrak d) = \mathsf d(\mathfrak a, \mathfrak b)$ for all $\mathfrak a, \mathfrak b, \mathfrak c, \mathfrak d \in \mathscr F^\ast(\mathscr{A}(H))$ with $\pi_H(\mathfrak a) = \pi_H(\mathfrak b)$.
\end{enumerate}
\begin{remark}
\label{rem:modelled-after}
The above definitions are all modeled after \cite[Definition 3.2]{BaSm}, where $\mathscr{F}^\ast(\mathscr{A}(H))$ is replaced by the category of rigid factorizations $\mathsf Z_{\rm Sm}(H)$, distances are all $\mathbf N$-valued and invariant, and the right-most inequality in \ref{it:distance-definition(v)} has a slightly different form, for the fact that $\mathsf Z_{\rm Sm}(H)$ is designed to include the trivial fac\-tor\-i\-za\-tions of the units of $H$ (see Remark \ref{remark:comparison-factorization-monoids} for notations and further details).
\end{remark}
The interest for subinvariant distances stems in part from the next lemma, which the reader may want to compare with points (1) and (2) of \cite[Lemma 3.7]{BaSm}.
\begin{lemma}
	\label{lem:subinvariant-distances}
	Let $H$ be a monoid and $\sf d$ a subinvariant distance on $H$. Then:
	\begin{enumerate}[label={\rm (\roman{*})}]
		\item\label{it:lem:subinvariant-distances(0)} $\mathsf d(\mathfrak a, \mathfrak b) = \mathsf d(\mathfrak c, \mathfrak d)$ for all $\mathfrak a, \mathfrak b, \mathfrak c, \mathfrak d \in \mathscr{F}^\ast(\mathscr{A}(H))$ with $(\mathfrak a, \mathfrak c), (\mathfrak b, \mathfrak d) \in \mathscr{C}_H$.
		\item\label{it:lem:subinvariant-distances(i)} $\mathsf d(\mathfrak a \ast \mathfrak c, \mathfrak b \ast \mathfrak d) \le \mathsf d(\mathfrak a, \mathfrak b) + \mathsf d(\mathfrak c, \mathfrak d)$ for all $\mathfrak a, \mathfrak b, \mathfrak c, \mathfrak d \in \mathscr{F}^\ast(\mathscr{A}(H))$.
		\item\label{it:lem:subinvariant-distances(ii)} The binary relation $\sim_{\sf d}$ on $\mathscr{F}^\ast(\mathscr{A}(H))$, defined by taking $\mathfrak a \sim_{\sf d} \mathfrak b$, for some $\mathfrak a, \mathfrak b \in \mathscr{F}^\ast(\mathscr{A}(H))$, if and only if $\pi_H(\mathfrak a) = \pi_H(\mathfrak b)$ and $\mathsf d(\mathfrak a, \mathfrak b) = 0$, is a monoid congruence, with $\llb \mathfrak c \rrb_{\mathscr{C}_H} \subseteq \llb \mathfrak c \rrb_{\sim_{\sf d}}$ for every $\mathfrak c \in \mathscr{F}^\ast(\mathscr{A}(H))$.
	\end{enumerate}
\end{lemma}
\begin{proof}
	\ref{it:lem:subinvariant-distances(0)} Let $\mathfrak a, \mathfrak b, \mathfrak c \in \mathscr{F}^\ast(\mathscr{A}(H))$ with $(\mathfrak a, \mathfrak c) \in \mathscr{C}_H$. By \ref{it:distance-definition(ii)}, it is sufficient to show that $\mathsf d(\mathfrak a, \mathfrak b) \le \mathsf d(\mathfrak c, \mathfrak b)$, which is straightforward, because $\mathsf d(\mathfrak a, \mathfrak b) \le \mathsf d(\mathfrak a, \mathfrak c) + \mathsf d(\mathfrak c, \mathfrak b)$ by \ref{it:distance-definition(iii)} and $\mathsf d(\mathfrak a, \mathfrak c) = 0$ by \ref{it:distance-definition(i)}.
	
	\ref{it:lem:subinvariant-distances(i)} Recall from the above that $\mathsf d$ is non-negative. Then, consider that, for all $\mathfrak a, \mathfrak b, \mathfrak c, \mathfrak d \in \mathscr{F}^\ast(\mathscr{A}(H))$,
	$$
	\mathsf d(\mathfrak a \ast \mathfrak c, \mathfrak b \ast \mathfrak d) \fixed[-0.5]{\text{ }} \stackrel{\ref{it:distance-definition(iii)}}{\le} \fixed[-0.5]{\text{ }} \mathsf d(\mathfrak a \ast \mathfrak c, \mathfrak b \ast \mathfrak c) + \mathsf d(\mathfrak b \ast \mathfrak c, \mathfrak b \ast \mathfrak d) \fixed[-0.5]{\text{ }} \stackrel{\ref{it:distance-definition(iv)}}{\le} \fixed[-0.5]{\text{ }} \mathsf d(\mathfrak a, \mathfrak b) + \mathsf d(\mathfrak c, \mathfrak d),
	$$
	\ref{it:lem:subinvariant-distances(ii)} Let $\mathfrak a, \mathfrak b, \mathfrak c, \mathfrak d \in \mathscr{F}^\ast(\mathscr{A}(H))$. If $\mathfrak a \sim_{\sf d} \mathfrak b$ and $\mathfrak b \sim_{\sf d} \mathfrak c$, then $\mathsf d(\mathfrak a, \mathfrak c) \le \mathsf d(\mathfrak a, \mathfrak b) + \mathsf d(\mathfrak b, \mathfrak c) = 0$ by \ref{it:distance-definition(iii)}, whence it is easy to check that $\sim_{\sf d}$ is an equivalence relation. To show that $\sim_{\sf d}$ is actually a congruence, assume $\mathfrak a \sim_{\sf d} \mathfrak b$ and $\mathfrak c \sim_{\sf d} \mathfrak d$, i.e., $\pi_H(\mathfrak a) = \pi_H(\mathfrak b)$, $\pi_H(\mathfrak c) = \pi_H(\mathfrak d)$, and $\mathsf d(\mathfrak a, \mathfrak b) = \mathsf d (\mathfrak c, \mathfrak d) = 0$. We need to prove that $\mathfrak a \ast \mathfrak c \sim_{\sf d} \mathfrak b \ast \mathfrak d$, which is immediate, since, on the one hand, $\pi_H$ being a homomorphism $\mathscr{F}^\ast(H) \to H$ yields
	\begin{equation*}
	\label{equ:piH-of-sets}
	\pi_H(\mathfrak a \ast \mathfrak c) = \pi_H(\mathfrak a) \ast \pi_H(\mathfrak c) = \pi_H(\mathfrak b) \ast \pi_H(\mathfrak d) = \pi_H(\mathfrak b \ast \mathfrak d),
	\end{equation*}
	and on the other, we obtain from \ref{it:lem:subinvariant-distances(i)} that $
	\mathsf d(\mathfrak a \ast \mathfrak c, \mathfrak b \ast \mathfrak d) \le \mathsf d(\mathfrak a, \mathfrak b) + \mathsf d (\mathfrak c, \mathfrak d) = 0$. So $\sim_{\mathsf{d}}$ is a con\-gru\-ence, and the rest is trivial by \ref{it:distance-definition(i)}.
\end{proof}
By Lemma \ref{lem:subinvariant-distances}\ref{it:lem:subinvariant-distances(ii)}, every subinvariant distance $\mathsf d$ on $H$ yields a corresponding notion of factorization class, by looking at the quotient $\mathsf Z_{\sf d}(H)$ of $\mathscr{F}^\ast(\mathscr{A}(H))$ by the congruence $\sim_{\sf d}$, i.e., by identifying two words $\mathfrak a, \mathfrak b \in \mathscr{F}^\ast(\mathscr{A}(H))$ if and only if $\mathsf d(\mathfrak a, \mathfrak b) = 0$, cf. \cite[Definition 3.8(1)]{BaSm}.
However, we will not pursue this direction here, as it would take us too far from our main goals.
Instead,
we note that, by Lemma \ref{lem:subinvariant-distances}\ref{it:lem:subinvariant-distances(ii)}, $\mathsf Z_{\mathsf d}(H) = \mathsf Z(H)$ whenever $\mathsf d$ is a $\mathscr{C}_H$-metric, and we proceed to introduce \textit{the} distance we are going to use in the sequel of the paper and to show that it is, in fact, a $\mathscr{C}_H$-metric, cf. \cite[Proposition 1.2.5]{GeHK06}.

\begin{definition}
We set
$
\mathscr{A}^\ast(H) := \bigl\{H^\times a H^\times: a \in \mathscr{A}(H)\bigr\}$, and
given $\mathfrak A \in \mathscr{A}^\ast(H)$ and $\mathfrak z \in \mathscr{F}^\ast(H)$, we let
$$
\mathsf{v}_{H}(\mathfrak z\fixed[0.22]{\text{ }}; \mathfrak A) :=
\left\{
\begin{array}{ll}
\!\! \bigl|\bigl\{i \in \llb 1, n \rrb: z_i \in \mathfrak A \bigr\}\bigr| & \text{if }n := \|\mathfrak z\|_H \ge 1 \text{ and } \mathfrak z = z_1 \ast \cdots \ast z_n \\
\!\! 0 & \text{otherwise}
\end{array}
\right.\!\!\!,
$$
cf. \cite[Definition 1.1.9.1]{GeHK06}.
Then, for all $\mathfrak a, \mathfrak b \in \mathscr{F}^\ast(\mathscr{A}(H))$ we take
$$
\delta_H(\mathfrak a, \mathfrak b) :=
\left\{
\begin{array}{ll}
\!\! 0 & \text{if }\pi_H(\mathfrak a) = \pi_H(\mathfrak b) \\
\!\! \frac{1}{2} & \text{otherwise}
\end{array}
\right.
$$
and
\begin{equation}
\label{equ:definition-of-the-logical-and}
\mathfrak a \land_H \mathfrak b := {\max\bigr(\|\mathfrak a\|_H, \|\mathfrak b\|_H\bigl)} - {{\sum}_{\mathfrak A \in \mathscr{A}^\ast(H)} \min(\mathsf{v}_{H}(\mathfrak a\fixed[0.22]{\text{ }}; \mathfrak A), \mathsf{v}_{H}(\mathfrak b\fixed[0.22]{\text{ }}; \mathfrak A))}.
\end{equation}
Lastly, we let the \textit{matching distance} of $H$ be the function
$$
\mathsf{d}_H: \mathscr{F}^\ast(\mathscr{A}(H)) \times \mathscr{F}^\ast(\mathscr{A}(H)) \to \mathbf R: (\mathfrak a, \mathfrak b) \mapsto \max(\delta_H(\mathfrak a, \mathfrak b), \mathfrak a \wedge_H \mathfrak b).
$$
\end{definition}
It turns out that $\mathsf d_H$ provides a natural way to measure how different two factorizations of a fixed element are from each other, especially when related to our definition of the factorization monoid $\mathsf Z(H)$.
\begin{lemma}
	\label{lem:fundamental-lemma-for-distance}
	Let $H$ be a monoid and pick $\mathfrak a, \mathfrak b \in \mathscr{F}^\ast(\mathscr{A}(H))$. Then
	\begin{equation}
	\label{equ:identity-for-the-wedge}
	\mathfrak a \wedge_H \mathfrak b = \frac{1}{2}\fixed[0.2]{\text{ }}{\sum}_{\mathfrak{A} \in \mathscr{A}^\ast(H)} \bigl|
	\mathsf{v}_{H}(\mathfrak a\fixed[0.22]{\text{ }}; \mathfrak A) - \mathsf{v}_{H}(\mathfrak b\fixed[0.22]{\text{ }}; \mathfrak A)\bigr| + \frac{1}{2}\fixed[0.2]{\text{ }}\bigl|\|\mathfrak a\|_H - \|\mathfrak b\|_H\bigr|.
	\end{equation}
	In particular, if $(\mathfrak a, \mathfrak c), (\mathfrak b, \mathfrak d) \in \mathscr{C}_H$, then $\mathfrak a \wedge_H \mathfrak b = \mathfrak c \wedge_H \mathfrak d$.
\end{lemma}
\begin{proof}
	It is not difficult to see that
	\begin{equation}
	\label{equ:unravelling-definition-of-norm}
	\|\mathfrak z\|_H = \sum_{\mathfrak A \in \mathscr{A}^\ast(H)} \mathsf v_H(\mathfrak z\fixed[0.22]{\text{ }}; \mathfrak A) = \sideset{}{'}\sum_{\mathfrak A \in \mathscr A^\ast(H)} \mathsf v_H(\mathfrak z\fixed[0.22]{\text{ }}; \mathfrak A),\quad \text{for every }\mathfrak z \in \mathscr{F}^\ast(\mathscr{A}(H)).
	\end{equation}
	Here, the prime in the sum means that the summation is over all $\mathfrak A \in \mathscr A^\ast(H)$ such that $\mathsf v_H(\mathfrak z\fixed[0.22]{\text{ }}; \mathfrak A) \ne 0$.
	
	As a consequence, the claim is trivial if $\mathfrak a$ or $\mathfrak b$ is the empty word.
	Otherwise, write $\mathfrak a = a_1 \ast \cdots \ast a_m$ and $\mathfrak b = b_1 \ast \cdots \ast b_n$, where $a_1, \ldots, a_m, b_1, \ldots, b_n \in \mathscr{A}(H)$. Because
	$$
	\mathfrak a \land_H \mathfrak b =
	\fixed[-0.2]{\text{ }} {\bigr(a_{\sigma(1)} \ast \cdots \ast a_{\sigma(m)}\bigl)} \land_H \fixed[-0.1]{\text{ }} {\bigr(b_{\tau(1)} \ast \cdots \ast b_{\tau(n)}\bigl)},
	$$
	for all $\sigma \in \mathfrak S_m$ and $\tau \in \mathfrak S_n$, there is no loss of generality in assuming that there exists $k \in \llb 0, \min(m,n) \rrb$ with $a_i \simeq_H b_i$ for $i \in \llb 1, k \rrb$, but $a_i \not \simeq_H b_j$ for every $i \in \llb k+1, m \rrb$ and $j \in \llb k+1, n \rrb$.
	Accordingly, set $\mathfrak a_0 := \mathfrak a$ and $\mathfrak b_0 := \mathfrak b$,
	and for each $i \in \llb 1, k \rrb$ define $\mathfrak a_i := a_i^{-1}\mathfrak a_{i-1}$ and $\mathfrak b_i := b_i^{-1} \mathfrak b_{i-1}$.
	Then
	\begin{equation}
	\label{equ:unravelling-the-distance}
	\begin{split}
	\mathfrak a \land_H \mathfrak b  = \max\bigl(\|\mathfrak a_k\|_H, \|\mathfrak b_k\|_H\bigr) \fixed[-0.2]{\text{ }} & =
	\frac{1}{2} \Big(\|\mathfrak a_k\|_H + \|\mathfrak b_k\|_H\Big)\fixed[-0.2]{\text{ }} + \frac{1}{2}\bigl|\|\mathfrak a_k\|_H - \|\mathfrak b_k\|_H\bigr|,
	\end{split}
	\end{equation}
	where for the second equality we have used that $2\max(x,y) = x+y+|x-y|$ for all $x, y \in \mathbf R$. On the other hand, it is clear from our definitions that
	\begin{equation}\label{equ:equal-differences-of-norms}
	\|\mathfrak a_k\|_H - \|\mathfrak b_k\|_H = \|\mathfrak a\|_H - \|\mathfrak b\|_H,
	\end{equation}
	and it is easily checked that, for each $\mathfrak A \in \mathscr{A}^\ast(H)$, we have
	\begin{equation}\label{equ:primed-sums}
	\|\mathfrak a_k\|_H + \|\mathfrak b_k\|_H \stackrel{\eqref{equ:unravelling-definition-of-norm}}{=} \sideset{}{'}\sum_{\mathfrak A \in \mathscr A^\ast(H)} \mathsf v_H(\mathfrak a_k; \mathfrak A) + \sideset{}{'}\sum_{\mathfrak A \in \mathscr A^\ast(H)} \mathsf v_H(\mathfrak b_k; \mathfrak A) = \sum_{\mathfrak A \in \mathscr A^\ast(H)} |\mathsf v_H(\mathfrak a_k; \mathfrak A) - \mathsf v_H(\mathfrak b_k; \mathfrak A)|,
	\end{equation}
	where the last equality follows from considering that $\mathsf v_H(\mathfrak a_k; \mathfrak A) \ne 0$, for some $\mathfrak A \in \mathscr A^\ast(H)$, if and only if $\mathsf v_H(\mathfrak b_k; \mathfrak A) = 0$ (by construction of $\mathfrak a_k$ and $\mathfrak b_k$). Then \eqref{equ:primed-sums}, together with \eqref{equ:unravelling-the-distance} and \eqref{equ:equal-differences-of-norms}, leads to \eqref{equ:identity-for-the-wedge}, because
	$$
	\mathsf v_H(\mathfrak a_k; \mathfrak A) - \mathsf v_H(\mathfrak b_k; \mathfrak A) = \mathsf v_H(\mathfrak a; \mathfrak A) - \mathsf v_H(\mathfrak b; \mathfrak A), \quad \text{for all } \mathfrak A \in \mathscr A^*(H).
	$$
	The ``In particular'' part of the statement is now immediate, since $\|\mathfrak c\|_H = \|\mathfrak d\|_H$ and $\mathsf v_H(\mathfrak c \fixed[0.22]{\text{ }}; \mathfrak A) = \mathsf v_H(\mathfrak d \fixed[0.22]{\text{ }}; \mathfrak A)$ for all $(\mathfrak c, \mathfrak d) \in \mathscr{C}_H$ and $\mathfrak A \in \mathscr{A}^\ast(H)$.
\end{proof}
\begin{proposition}
	\label{prop:dH-a-locally-invariant-metric}
	$\mathsf d_H$ is a locally invariant $\mathscr{C}_H$-metric and has the additional property that:
	\begin{enumerate}[label={\rm (\roman{*})}]
		\item $\mathsf d_H(\mathfrak a, \mathfrak b)$ is a non-negative integer for every $(\mathfrak a, \mathfrak b) \in \mathscr{C}_H$;
		\item $\mathsf d_H(\mathfrak a^k, \mathfrak b^k) = k \fixed[0.2]{\text{ }} \mathsf d_H(\mathfrak a, \mathfrak b)$ for all $\mathfrak a, \mathfrak b \in \mathscr{F}^\ast(\mathscr{A}(H))$ and $k \in \mathbf N$.
	\end{enumerate}
	Moreover, $\mathsf d_H(\mathfrak a, \mathfrak b) = \frac{1}{2}$ for some non-empty $\mathscr{A}(H)$-words $\mathfrak a = a_1 \ast \cdots \ast a_m$ and $\mathfrak b = b_1 \ast \cdots \ast b_n$ of length $m$ and $n$, respectively, if and only if $\pi_H(\mathfrak a) \ne \pi_H(\mathfrak b)$, $m = n$, and there exists a permutation $\sigma \in \mathfrak S_n$ such that $a_i \simeq_H b_{\sigma(i)}$ for every $i \in \llb 1, m \rrb$.
\end{proposition}
\begin{proof}
	\ref{it:distance-definition(i)} and \ref{it:distance-definition(ii)} are trivial, the rest is a consequence of \eqref{equ:unravelling-definition-of-norm}, Lemma \ref{lem:fundamental-lemma-for-distance}, the triangle inequality for the absolute value, and the fact that $\delta_H(\mathfrak a, \mathfrak b) \le \delta_H(\mathfrak a, \mathfrak c) + \delta_H(\mathfrak c, \mathfrak b)$ for all $\mathfrak a, \mathfrak b, \mathfrak c \in \mathscr{F}^\ast(\mathscr{A}(H))$, with equality if and only if $(\mathfrak a, \mathfrak c) \in \mathscr{C}_H$ or $(\mathfrak c, \mathfrak b) \in \mathscr{C}_H$ (we encourage the reader to fill in the details).
\end{proof}
\begin{remark}
	\label{rem:Baeth-Smertning-permutable-distance}
	Up to the technical details highlighted in Remarks \ref{rem:permutable_factorizations} and \ref{rem:modelled-after}, $\mathsf{d}_H$ is no different from the \textit{permutable distance} introduced by Baeth and Smertnig in the cancellative setting, cf. \cite[Definition 3.4(2) and Construction 3.3(2)]{BaSm}. In particular, it follows from Remark \ref{rem:permutable_factorizations} and \cite[Remark 3.5(1)]{BaSm} that $\mathsf d_H$ is essentially the same as the distance of \cite[Definition 1.2.4 and p. 14]{GeHK06} on the level of reduced, cancellative, commutative monoids.
\end{remark}
We conclude the section with the definition of another arithmetic invariant that has played a prominent role in recent developments of factorization theory, as it provides more accurate information about factorizations than just their lengths.
\begin{definition}
	Let $H$ be a monoid.
	We take the \textit{catenary degree} of an element $x \in H$, denoted by ${\sf c}_H(x)$, to be the infimum of the set of all $d \in \mathbf N$ for which the following condition is verified:
\begin{itemize}
\item For all $\mathfrak a, \mathfrak b \in \mathcal{Z}_H(x)$ there are factorizations $\mathfrak{c}_0, \ldots, \mathfrak{c}_n \in \mathcal{Z}_H(x)$ with $\mathfrak{c}_0 = \mathfrak{a}$ and $\mathfrak{c}_n = \mathfrak b$ such that ${\sf d}_H(\mathfrak c_{i-1},\mathfrak c_i) \le d$ for every $i \in \llb 1, n \rrb$.
\end{itemize}
It is seen that ${\sf c}_H(x) = 0$, for a given $x \in H$, if and only if $|\mathsf{Z}_H(x)| \le 1$. Consequently, we take
\begin{equation*}
\label{equ:set-of-catenary-degrees}
\cat(H) := \{{\sf c}_H(x): x \in H\} \setminus \{0\} \subseteq \mathbf N^+ \cup \{\infty\},
\end{equation*}
and we call $\cat(H)$ the \textit{set of catenary degrees} (or \textit{catenary set}) of $H$.
\end{definition}
It is clear that $\cat(H) \subseteq \mathbf N^+$ if $H$ is a \BF-monoid, but this need not be true in general.
%it is easy to construct examples for which $\infty \in \cat(H)$.
\begin{remark}
\label{rem:comparing-catenary-degrees}
Let $x \in H \setminus H^\times$. It follows by Remark \ref{rem:Baeth-Smertning-permutable-distance} that, if $H$ is atomic and cancellative, ${\sf c}_H(x)$ coincides with the catenary degree of $x$ associated, according to \cite[Definition 4.1(3)]{BaSm}, to the permutable distance of Baeth and Smertnig. In particular, if $H$ is atomic, cancellative, and commutative, ${\sf c}_H(x)$ has the same value as the catenary degree of $x$ in the classical theory, cf. \cite[Definition 1.6.1.2]{GeHK06}.
\end{remark}
Note that, in the same spirit of \cite[\S{ }4]{BaSm}, every subinvariant distance $\sf d$ on $H$ gives rise to a corresponding notion of catenary degree. However, this is something beyond the scope of the present work.

\subsubsection{Equimorphisms}\label{subsub:equimorphisms}
The kind of arithmetic properties we consider in this paper, are often studied by reduction to suitable families of atomic monoids that are, in a certain way, less problematic than others.
This is achieved by means of transfer techniques (cf. Remark \ref{rem:philosophy-of-equimorphisms}), as per Halter-Koch's notion of transfer homomorphism in the commutative and cancellative setting, see \cite[Lemma 5.4]{HK97a}; or Baeth and Smertnig's notion of weak transfer homomorphism, see \cite[Definition 2.1]{BaSm} and \cite[\S{ }4]{Ge16c}.

\begin{definition}\label{def:equimorphisms}
Let $H$ and $K$ be multiplicatively written monoids, and let $\varphi$ a homomorphism $H \to K$.
We denote by $\varphi^\ast$ the unique (monoid) homomorphism $\mathscr{F}^\ast(H) \to \mathscr{F}^\ast(K)$ such that $\varphi^\ast(x) = \varphi(x)$ for all $x \in H$, 
and we refer to $\varphi$ as a (\textit{monoid}) \textit{equimorphism} (from $H$ to $K$) if:
\begin{enumerate}[label={\rm (\textsc{e}\arabic{*})}]
\item\label{covariant-transfer(1)}
%$\varphi(x) = 1_K$ for some $x \in H$ only if $x \in H^\times$, that is,
$\varphi^{-1}(K^\times) \subseteq H^\times$ (or equivalently $\varphi^{-1}(K^\times) = H^\times$).
\item\label{covariant-transfer(2)} $\varphi$ is \textit{atom-preserving}, i.e., $\varphi(a) \in \mathscr{A}(K)$ for all $a \in \mathscr{A}(H)$.
\item\label{covariant-transfer(3)} If $x \in H$, $\mathfrak b \in \mathcal{Z}_K(\varphi(x)) \ne \emptyset$, and $\|\mathfrak b\|_K \ne 0$, then $\varphi^\ast(\mathfrak a) \in \llb \mathfrak b \rrb_{\mathscr{C}_K}$ for some $\mathfrak{a} \in \mathcal{Z}_H(x)$.
\end{enumerate}
Moreover, we call $\varphi$ a \textit{weak transfer homomorphism} if it is an equimorphism and $K = K^\times \varphi(H) K^\times$.

Then, we say that $H$ is \textit{equimorphic} to $K$ if there exists an equimorphism from $H$ to $K$; and that $H$ is a \textit{transfer Krull monoid} if there is a weak transfer homomorphism from $H$ to a monoid of zero-sum sequences over an abelian group $G$ with support in a subset $G_0 \subseteq G$ (see \cite[Definition 2.5.5]{GeHK06} for further details and terminology).
\end{definition}
%Of course, every weak transfer homomorphism is an equimorphism, but the converse need not hold, not even in the cancellative commutative setting.
%Some remarks on these definitions are in order.
%
\begin{remark}
\label{rem:covariant-transfer(1)}
In \ref{covariant-transfer(3)}, the $K$-word $\varphi^\ast(\mathfrak a)$ is actually an $\mathscr{A}(K)$-word by condition \ref{covariant-transfer(2)}. In addition, the $\mathscr{A}(H)$-word $\mathfrak a$ is non-empty, since  $\|\mathfrak a\|_H = \|\varphi^*(\mathfrak a)\|_K$ and, on the other hand, $\varphi^*(\mathfrak a) \in \llb \mathfrak b \rrb_{\mathscr C_K}$ implies $\|\varphi^*(\mathfrak a)\|_K = \|\mathfrak b\|_K \ne 0$. 
Accordingly, write $\mathfrak a = a_1 \ast \cdots \ast a_m$ and $\mathfrak{b} = b_1 \ast \cdots \ast b_n$, with $a_1, \ldots, a_m \in \mathscr{A}(H)$ and $b_1, \ldots, b_n \in \mathscr{A}(K)$. Then $\varphi^\ast(\mathfrak{a}) \in \llb \mathfrak b \rrb_{\mathscr{C}_K}$ is equivalent to having that $\pi_K(\mathfrak b) = \pi_K(\varphi^\ast(\mathfrak a))$, $m = n$, and there exists a permutation $\sigma \in \mathfrak S_n$ such that $b_{\sigma(i)} \simeq_K \varphi(a_i)$ for every $i \in \llb 1, n \rrb$.
\end{remark}
\begin{remark}
\label{rem:covariant-transfer(3)}
Condition \ref{covariant-transfer(2)} cannot be proved from \ref{covariant-transfer(1)} and \ref{covariant-transfer(3)}. Indeed, let $H$ (respectively, $K)$ be the monoid of non-negative integers (respectively, non-negative real numbers) under addition, and let $\varphi$ be the canonical embedding. Clearly, $\varphi$ satisfies \ref{covariant-transfer(1)} and \ref{covariant-transfer(3)}, because $H^\times = K^\times = \{0\}$ and $\mathscr{A}(K) = \emptyset$. But $1 \in \mathscr{A}(H)$, so $\varphi$ cannot satisfy \ref{covariant-transfer(2)}.
\end{remark}
\begin{remark}
\label{rem:comparison-between-equis-and-weak-transfer-homs}
In Baeth and Smertnig's original definition of a weak transfer homomorphism $\varphi: H \to K$, it is assumed that $H$ is cancellative and $K$ is atomic, which implies that $\varphi$ is atom-preserving. 
By Remark \ref{rem:covariant-transfer(3)}, this need not hold for an arbitrary equimorphism, which is the reason for having included condition \ref{covariant-transfer(2)} in the above definitions. In particular, it follows from here and Remark \ref{rem:covariant-transfer(1)} that every weak transfer homomorphism in the sense of Baeth and Smertnig is also a weak transfer homomorphism in our sense, and hence an equimorphism.
\end{remark}
\begin{remark}
\label{rem:philosophy-of-equimorphisms}
The rationale behind the introduction of transfer techniques in factorization theory is as follows: We have some kind of monoid homomorphism $\varphi: H \to K$, and we want to understand properties of one of $H$ or $K$ by looking at corresponding properties of the other. To this end, we use $\varphi$ to shift information from $H$ to $K$ (as we do here with equimorphisms, see Theorems \ref{th:cotransfer-hom} and \ref{th:transfer}), if $H$ is, in a sense, easier to study than $K$; or to pull it back from $K$ to $H$ (as is commonly the case with transfer and weak transfer homomorphisms), if it is the other way around.
\end{remark}
\subsection{Abstract arithmetic results}
Now that we have introduced most of the basic notions we need and clarified, we hope, some subtle aspects of the theory, we are ready to prove a couple of results extending some pieces of
\cite[Proposition 1.2.11.1]{GeHK06} and \cite[Lemma 11]{Ge16c}, respectively, to the general setting of this work: As is true for the largest number of results from the present section, they will be used later, in \S{}\S{ }\ref{sec:power_monoid} and \ref{sec:the_case_of_integers}, to study the arithmetic of power monoids.
\begin{proposition}
\label{prop:divisor-closed-sub}
Let $H$ be a monoid, and assume that $M$ is a divisor-closed submonoid of $H$. Then $M^\times = H^\times$ and $\mathscr{A}(M) = \mathscr{A}(H) \cap M$. In addition, ${\sf L}_M(x) = {\sf L}_H(x)$, $\mathsf{Z}_M(x) = \mathsf{Z}_H(x)$, and ${\sf c}_M(x) = {\sf c}_H(x)$ for all $x \in M$, and con\-se\-quent\-ly $\mathscr{L}(M) \subseteq \mathscr{L}(H)$, $\Delta(M) \subseteq \Delta(H)$, and $\cat(M) \subseteq \cat(H)$.
\end{proposition}
\begin{proof}
Of course, $M^\times \subseteq H^\times$. On the other hand, $u \in H^\times$ only if $u \mid_H 1_H$, and since $1_H = 1_M$ and $M$ is a divisor-closed submonoid of $H$, this implies $H^\times \subseteq M^\times$. To wit, $M^\times = H^\times$.

Consequently, it is clear that $\mathscr{A}(H) \cap M \subseteq \mathscr{A}(M)$. To prove the opposite inclusion, let $a \in \mathscr{A}(M)$, and write $a = xy$ for some $x, y \in H$. Then $x, y \in M$, using again that $M$ a divisor-closed submonoid of $H$. So, $x$ or $y$ is a unit of $M$, and hence $a \in \mathscr{A}(H)$, because $M^\times = H^\times$. Therefore, given $x, y \in M$, it is immediate that $\mathcal{Z}_M(x) = \mathcal{Z}_H(x)$, and $x \simeq_M y$ if and only if $x \simeq_H y$. This yields $\mathsf Z_M(x) = \mathsf Z_H(x)$ for every $x \in M$, and the rest is obvious.
\end{proof}
\begin{theorem}
\label{th:cotransfer-hom}
Let $H$ and $K$ be monoids, and $\varphi: H \to K$ an equimorphism. The following hold:
\begin{enumerate}[label={\rm (\roman{*})}]
\item\label{it:prop:cotransfer-hom(1)} ${\sf L}_H(x) = {\sf L}_K(\varphi(x))$ for each $x \in H \setminus H^\times$.
\item\label{it:prop:cotransfer-hom(3)} For each $\mathfrak A \in \mathscr A^\ast(H)$ there exists a unique $\mathfrak B \in \mathscr A^\ast(K)$ with $\varphi(\mathfrak A) \subseteq \mathfrak B$.
\item\label{it:prop:cotransfer-hom(4a)} $\varphi^\ast(\mathfrak a) \land_K \varphi^\ast(\mathfrak b) \le \mathfrak a \land_H \mathfrak b$ for all $\mathfrak a, \mathfrak b \in \mathscr F^\ast(\mathscr A(H))$ with $\|\mathfrak a\|_H = \|\mathfrak b\|_H$.
\item\label{it:prop:cotransfer-hom(4)} ${\sf c}_K(\varphi(x)) \le {\sf c}_H(x)$ for all $x \in H$.
\item\label{it:prop:cotransfer-hom(5)} $\varphi(H)$ is a divisor-closed submonoid of $K$ only if $\varphi(H^\times) = K^\times$.
\item\label{it:prop:cotransfer-hom(6)} If $\varphi(H^\times) = K^\times$ and $K$ is atomic, then $\varphi$ is surjective.
\item\label{it:prop:cotransfer-hom(7)} If $K$ is atomic, then so is $H$.
\end{enumerate}
In particular, $\mathscr{L}(H) \subseteq \mathscr{L}(K)$ and $\Delta(H) \subseteq \Delta(K)$.
\end{theorem}
\begin{proof}
\ref{it:prop:cotransfer-hom(1)} Pick $x \in H \setminus H^\times$, and set $L := {\sf L}_H(x)$ and $L^\prime := {\sf L}_K(\varphi(x))$. Since $x$ is not a unit, it is clear from condition \ref{covariant-transfer(1)} of Definition \ref{def:equimorphisms} that $\varphi(x) \notin K^\times$, which yields $L, L^\prime \subseteq \mathbf N^+$ (see Remark \ref{rem:where-is-0}).

Accordingly, assume $L \ne \emptyset$ and let $k \in L$. Then $x = a_1 \cdots a_k$ for some $a_1, \ldots, a_k \in \mathscr{A}(H)$. Therefore $\varphi(x) = \varphi(a_1) \cdots \varphi(a_k)$, and hence $k \in L^\prime$, because $\varphi$ is an atom-preserving homomorphism.

Conversely, assume $L^\prime \ne \emptyset$ and pick $k \in L^\prime$. Then $\varphi(x) = b_1 \cdots b_k$ for some $b_1, \ldots, b_k \in \mathscr{A}(K)$, and by Remark \ref{rem:covariant-transfer(1)} there are $a_1, \ldots, a_k \in \mathscr{A}(H)$ such that $x = a_1 \cdots a_k$, with the result that $k \in L$.

So, putting it all together, we can conclude that $L = L^\prime$. The ``In particular'' part of the statement (on systems of sets of lengths and delta sets) is then an obvious consequence.

\ref{it:prop:cotransfer-hom(3)} Given $\mathfrak A \in \mathscr{A}^\ast(H)$, let $a \in \mathfrak A$ and define $\mathfrak B := K^\times \varphi(a) K^\times$. Then $\mathfrak A = H^\times a H^\times$, and since $\varphi$ is a ho\-mo\-mor\-phism, we have that $\varphi(\mathfrak A) = \varphi(H^\times) \fixed[0.2]{\text{ }} \varphi(a) \fixed[0.2]{\text{ }} \varphi(H^\times) \subseteq K^\times \varphi(a) K^\times = \mathfrak B$; in addition, $\mathfrak B \in \mathscr A^\ast(K)$, because $\varphi$ is atom-preserving.
The rest is trivial, in that $\mathscr A^\ast(K)$ is the quotient set of $\mathscr A(K)$ under the restriction of the (equivalence) relation $\simeq_K$ to the atoms of $K$; so each class in $\mathscr A^\ast(K)$ is non-empty, and pairwise distinct classes are disjoint (by the general properties of equivalences).

\ref{it:prop:cotransfer-hom(4a)} Let $\mathfrak a, \mathfrak b \in \mathscr F^\ast(\mathscr A(H))$ with $\|\mathfrak a\|_H = \|\mathfrak b\|_H$. Then $\|\varphi^\ast(\mathfrak a)\|_K = \|\varphi^\ast(\mathfrak b)\|_K$, hence it is clear from \eqref{equ:definition-of-the-logical-and} that $\varphi^\ast(\mathfrak a) \land_K \varphi^\ast(\mathfrak b) \le \mathfrak a \land_H \mathfrak b$ if and only if
\begin{equation}\label{equ:inequality-between-valuations}
{\sum_{\mathfrak A \in \mathscr{A}^\ast(H)} \min(\mathsf{v}_{H}(\mathfrak a\fixed[0.22]{\text{ }}; \mathfrak A), \mathsf{v}_{H}(\mathfrak b\fixed[0.22]{\text{ }}; \mathfrak A))} \le {\sum_{\mathfrak B \in \mathscr{A}^\ast(K)} \min(\mathsf{v}_{K}(\varphi^\ast(\mathfrak a)\fixed[0.22]{\text{ }}; \mathfrak B), \mathsf{v}_{H}(\varphi^\ast(\mathfrak b)\fixed[0.22]{\text{ }}; \mathfrak B))}.
\end{equation}
Denote by $\mathfrak B^{\ast}$, for every $\mathfrak B \in \mathscr A^\ast(K)$, the set of all $\mathfrak A \in \mathscr A^\ast(H)$ such that $\varphi(\mathfrak A) \subseteq \mathfrak B$. Then we see from \ref{it:prop:cotransfer-hom(3)} that $\mathscr A^\ast(H) = \biguplus \{\mathfrak B^{\ast}: \mathfrak B \in \mathscr A^\ast(K)\}$. Therefore, a sufficient condition for \eqref{equ:inequality-between-valuations} to hold is that
\begin{equation}
\label{equ:simplified-inequality-on-valuations}
{\sum_{\mathfrak A \in \mathfrak{B}^{\ast}} \min(\mathsf{v}_{H}(\mathfrak a\fixed[0.22]{\text{ }}; \mathfrak A), \mathsf{v}_{H}(\mathfrak b\fixed[0.22]{\text{ }}; \mathfrak A))} \le \min(\mathsf{v}_{K}(\varphi^\ast(\mathfrak a)\fixed[0.22]{\text{ }}; \mathfrak B), \mathsf{v}_{H}(\varphi^\ast(\mathfrak b)\fixed[0.22]{\text{ }}; \mathfrak B)), \quad\text{for all } \mathfrak B \in \mathscr{A}^\ast(K).
\end{equation}
On the other hand, it is easily verified that, for all $a_1, b_1, \ldots, a_n, b_n \in \mathbf R$,
$$
\min(a_1, b_1) + \cdots + \min(a_n, b_n) \le \min(a_1 + \cdots + a_n, b_1 + \cdots + b_n).
$$
So, for \eqref{equ:simplified-inequality-on-valuations} to be true it is enough to check that $\sum_{\mathfrak A \in \mathfrak B^{\ast}} \mathsf{v}_{H}(\mathfrak c\fixed[0.22]{\text{ }}; \mathfrak A) \le \mathsf{v}_{K}(\varphi^\ast(\mathfrak c)\fixed[0.22]{\text{ }}; \mathfrak B)$ for every $\mathfrak B \in \mathscr{A}^\ast(K)$ and every non-empty $\mathscr{A}(H)$-word $\mathfrak c = c_1 \ast \cdots \ast c_n$,
which is now trivial, because $c_i \in \mathfrak A$ for some $\mathfrak A \in \mathfrak B^{\ast}$ and $i \in \llb 1, n \rrb$ only if $\varphi(c_i) \in \mathfrak B$.

\ref{it:prop:cotransfer-hom(4)} Let $x \in H$. The inequality is obvious if ${\sf c}_K(\varphi(x)) = 0$.
Otherwise, $\mathcal{Z}_K(\varphi(x)) \ne \emptyset$ and $x \ne 1_H$. So, pick $\mathfrak a^{\fixed[0.2]{\text{ }}\prime}, \mathfrak b^{\fixed[0.2]{\text{ }}\prime} \in \mathcal{Z}_K(\varphi(x))$. By condition \ref{covariant-transfer(3)}, $\varphi^\ast(\mathfrak a) \in \llb \mathfrak a^{\fixed[0.2]{\text{ }}\prime} \rrb_{\mathscr{C}_K}$ and $\varphi^\ast(\mathfrak b) \in \llb \mathfrak b^{\fixed[0.2]{\text{ }}\prime} \rrb_{\mathscr{C}_K}$ for some $\mathfrak a, \mathfrak b \in \mathcal{Z}_H(x)$. Consequently, there are factorizations $\mathfrak c_0, \ldots, \mathfrak c_n \in \mathcal{Z}_H(x)$ with $\mathfrak c_0 = \mathfrak a$, $\mathfrak c_n = \mathfrak b$, and ${\sf d}_H(\mathfrak c_{i-1}, \mathfrak c_i) \le {\sf c}_H(x)$ for each $i \in \llb 1, n \rrb$.
Set $\mathfrak c_0^{\fixed[0.2]{\text{ }}\prime} := \mathfrak a^{\fixed[0.2]{\text{ }}\prime}$, $\mathfrak c_n^{\fixed[0.2]{\text{ }}\prime} := \mathfrak b^{\fixed[0.2]{\text{ }}\prime}$, and $\mathfrak c_i^{\fixed[0.2]{\text{ }}\prime} := \varphi^\ast(\mathfrak c_i)$ for $i \in \llb 1, n-1 \rrb$. By Lemma \ref{lem:fundamental-lemma-for-distance}, 
$$
\mathfrak c_0^{\fixed[0.2]{\text{ }}\prime} \land_K \mathfrak c_1^{\fixed[0.2]{\text{ }}\prime} = \varphi^\ast(\mathfrak a) \land_K \mathfrak c_1^{\fixed[0.2]{\text{ }}\prime} 
= \varphi^\ast(\mathfrak c_0) \land_K \mathfrak c_1^{\fixed[0.2]{\text{ }}\prime}
\quad\text{and}\quad
\mathfrak c_{n-1}^{\fixed[0.2]{\text{ }}\prime} \land_K \mathfrak c_n^{\fixed[0.2]{\text{ }}\prime} = \mathfrak c_{n-1}^{\fixed[0.2]{\text{ }}\prime} \land_K \varphi^\ast(\mathfrak b)
= \mathfrak c_{n-1}^{\fixed[0.2]{\text{ }}\prime} \land_K \varphi^\ast(\mathfrak c_n).
$$
Therefore, we conclude from \ref{it:prop:cotransfer-hom(4a)} that, for every $i \in \llb 1, n \rrb$,
$$
{\sf d}_K(\mathfrak c_{i-1}^{\fixed[0.2]{\text{ }}\prime}, \mathfrak c_i^{\fixed[0.2]{\text{ }}\prime}) = \mathfrak c_{i-1}^{\fixed[0.2]{\text{ }}\prime} \land_K \mathfrak c_i^{\fixed[0.2]{\text{ }}\prime} = \varphi^\ast(\mathfrak c_{i-1}) \land_K \varphi^\ast(\mathfrak c_i) \le \mathfrak c_{i-1} \land_H \mathfrak c_i = \mathsf d_H(\mathfrak c_{i-1}, \mathfrak c_i) \le {\sf c}_H(x),
$$
which implies that the catenary degree of $\varphi(x)$ in $K$ is bounded above by ${\sf c}_H(x)$.

\ref{it:prop:cotransfer-hom(5)}
It is straightforward from Proposition \ref{prop:divisor-closed-sub}, when considering that $\varphi(H^\times) = \varphi(H)^\times$ (again, by the fact that $\varphi$ is a homomorphism).

\ref{it:prop:cotransfer-hom(6)} %As for the rest, a
Assume that $K$ is atomic and $\varphi(H^\times) = K^\times$, and pick $y \in K \setminus K^\times$. We have to prove $y \in \varphi(H)$.
To this end, it follows from the atomicity of $K$ that $y = b_1 \cdots b_n$ for some $b_1, \ldots, b_n \in \mathscr{A}(K)$, and since $\varphi$ is an equimorphism, we get from Remark \ref{rem:covariant-transfer(1)} (and the assumption that every unit of $K$ is the image under $\varphi$ of some unit of $H$) that there are $a_1, \ldots, a_n \in \mathscr{A}(H)$, $u_1, v_1, \ldots, u_n, v_n \in H^\times$, and $\sigma \in \mathfrak{S}_n$ such that $b_i = \varphi(u_i) \fixed[0.2]{\text{ }}\varphi(a_{\sigma(i)}) \fixed[0.2]{\text{ }}\varphi(v_i)$ for every $i \in \llb 1, n \rrb$. So $y = \varphi(x)$, where $x := u_1 a_{\sigma(1)} v_1 \cdots u_n a_{\sigma(n)} v_n$.

\ref{it:prop:cotransfer-hom(7)} Suppose that $K$ is atomic, and let $x \in H \setminus H^\times$. Then $\varphi(x) \notin K^\times$ (by condition \ref{covariant-transfer(1)} of Definition \ref{def:equimorphisms}). Hence $\mathsf L_H(x) = \mathsf L_K(\varphi(x)) \ne \emptyset$, by Theorem  \ref{th:cotransfer-hom}\ref{it:prop:cotransfer-hom(1)} and the atomicity of $K$. Viz., $H$ is atomic.
\end{proof}
\begin{remark}
\label{rem:counterexample-to-divisor-closedness}
%The second part of
The conclusion of Theorem \ref{th:cotransfer-hom}\ref{it:prop:cotransfer-hom(6)} need not be true if $K$ is not atomic. Indeed, let $K$ be the monoid of non-negative real numbers under addition, and $H$ the submonoid of $K$ consisting of the rational numbers. Both $H$ and $K$ are reduced, divisible, cancellative, commutative monoids, and hence $\mathscr{A}(H) = \mathscr{A}(K) = \emptyset$, since every non-unit of $H$ (respectively, of $K$) is the square of a non-unit. It follows that the canonical embedding $H \to K$ is an injective equimorphism. Yet, $H$ is not a divisor-closed sub\-monoid of $K$.
\end{remark}
\begin{remark}
\label{rem:decreasing-cat-degree}
It follows from \cite[p. 506]{BaSm} and Remark \ref{rem:Baeth-Smertning-permutable-distance} that, if $\varphi: H \to K$ is a transfer homomorphism in the sense of \cite[Definition 2.1(1)]{BaSm} and $K$ is atomic, then $\mathsf c_K(\varphi(x)) \le \mathsf c_H(x)$ for all $x \in H \setminus H^\times$, which is a special case of Theorem \ref{th:cotransfer-hom}\ref{it:prop:cotransfer-hom(4)}, because $\varphi$ is a weak transfer ho\-mo\-mor\-phism by \cite[p. 483]{BaSm}, and hence an equimorphism by Remark \ref{rem:comparison-between-equis-and-weak-transfer-homs}.

Sharper results are available in the literature under stronger assumptions or for particular classes of monoids; see, e.g., \cite[Theorem 3.2.5.4]{GeHK06} for the case when $H$ and $K$ are commutative, cancellative, atomic monoids and $\varphi$ is a transfer homomorphism, or \cite[Proposition 4.8]{BaSm} for the case of an \textit{isoatomic} weak transfer homomorphism from an atomic cancellative monoid to another ($\varphi$ is isoatomic if $\varphi(a) \simeq_K \varphi(b)$, for some $a, b \in \mathscr{A}(H)$, implies $a \simeq_H b$).
\end{remark}
Our next step is to have some convenient criteria for a unit-cancellative monoid to be atomic or BF. We start with a couple of definitions, the second of which extends \cite[Definition 1.1.3.2]{GeHK06} from the case when $H$ is cancellative and commutative, but is more restrictive than analogous definitions given by Geroldinger \cite[p. 533]{Ge13} and Smertnig \cite[p. 364]{Sm16} in the cancellative setting.
\begin{definition}
	A monoid $H$ satisfies the \textit{ascending chain condition} (shortly, ACC) on principal right ideals (respectively, on principal left ideals) if, for every sequence $(a_n)_{n \ge 1}$ of elements of $H$ such that $a_n H \subseteq a_{n+1} H$ (respectively, $H a_n \subseteq H a_{n+1}$) for all $n$, there is an index $v \in \mathbf N^+$ for which $a_n H = a_v H$ (respectively, $H a_v = H a_n$) when $n \ge v$. In addition, we say that $H$ satisfies the ACCP if it satisfies the ACC on both principal right and principal left ideals.
\end{definition}
\begin{definition}
	Let $H$ be a monoid and $\lambda$ a function $H \to \mathbf{N}$. We say that $\lambda$ is a \textit{length function} (on $H$) if $\lambda(x) < \lambda(y)$ for all $x, y \in H$ such that $y = uxv$ for some $u, v \in H$ with $u \notin H^\times$ or $v \notin H^\times$.
\end{definition}
Then, we proceed to
our first theorem, which is an all-embracing generalization of \cite[Proposition 1.1.4 and Corollary 1.3.3]{GeHK06} and \cite[Lemma 3.1(1)]{FGKT}.
\begin{lemma}
\label{lem:products_and_units}
Let $H$ be a unit-cancellative monoid, and let $x, y \in H$. We have that:
\begin{enumerate}[label={\rm (\roman{*})}]
\item\label{it:lem:products_and_units(i)} $xy \in H^\times$ if and only if $x, y \in H^\times$.
\item\label{it:lem:products_and_units(ib)} If $xy = xu$ \textup{(}respectively, $yx = ux$\textup{)} for some $u \in H^\times$, then $y \in H^\times$.
\item\label{it:lem:products_and_units(ii)} $xH = yH$ \textup{(}respectively, $Hx = Hy$\textup{)} if and only if $x \in y \fixed[0.2]{\text{ }}H^\times$ \textup{(}respectively, $x \in H^\times y$\textup{)}.
\item\label{it:lem:products_and_units(iii)} If $H$ satisfies the \textup{ACC} on principal right \textup{(}respectively, principal left\textup{)} ideals and $x \in H \setminus H^\times$, then $x \in \mathscr{A}(H) \cdot H$ \textup{(}respectively, $x \in H \cdot \mathscr{A}(H)$\textup{)}.
\end{enumerate}
\end{lemma}
\begin{proof}
\ref{it:lem:products_and_units(i)} The ``if'' part is trivial, see also Lemma \ref{lem:basic-properties-atoms-units}\ref{it:lem:basic-properties-atoms-units(i)}.
As for the other direction, assume $xy$ is a unit, and let $u \in H$ such that $xyu = uxy = 1_H$. This means that $x$ is right-invertible and $y$ is left-invertible (a right inverse of $x$ being given by $yu$, and a left inverse of $y$ by $ux$). Moreover, we have $xyux = x$, which implies, by the unit-cancellativity of $H$, that $v := yux$ is a unit, and hence
$(v^{-1} yu) x = y(uxv^{-1}) = 1_H$.
So, in conclusion, we see that both $x$ and $y$ are right- and left-invertible, hence are invertible.

\ref{it:lem:products_and_units(ib)} Suppose that $xy = xu$ for some $u \in H^\times$ (the other case is similar). Then $x = xyu^{-1}$, so we get by the unit-cancellativity of $H$ and point \ref{it:lem:products_and_units(i)} that $y \in H^\times$.

\ref{it:lem:products_and_units(ii)} It is obvious that, if $x \in y \fixed[0.2]{\text{ }}H^\times$, then $xH = yH$. So assume $xH = yH$. Then $x=ya$ and $y = xb$ for some $a, b \in H$, with the result that $y = yab$. Because $H$ is unit-cancellative, it follows that $ab \in H^\times$, and hence $a, b \in H^\times$ by point \ref{it:lem:products_and_units(i)} above. This yields $x \in y \fixed[0.2]{\text{ }}H^\times$ and finishes the proof, since the analogous statement for principal left ideals can be established in a similar way (we omit details).

\ref{it:lem:products_and_units(iii)} We prove the statement only for principal right ideals, as the other case is similar. To this end, assume for a contradiction that the claim is false. Then the set
$$
\Omega := \{zH: z \in H \setminus H^\times \text{ and } z \notin \mathscr{A}(H) \cdot H\}
$$
is non-empty. So, using that $H$ satisfies the ACC on principal right ideals, $\Omega$ has a $\subseteq$-maximal element, say $\bar zH$.
Clearly, $\bar z$ is neither a unit nor an atom (of $H$), because $\bar z H \in \Omega$. Therefore, $\bar z = ab$ for some $a,b \in H \setminus H^\times$, and it is clear that $a \notin \mathscr{A}(H) \cdot H$, otherwise we would have $\bar z \in \mathscr{A}(H) \cdot H$. Thus, $aH \in \Omega$ and $\bar z H \subseteq aH$. But $\bar zH$ is a $\subseteq$-maximal element of $\Omega$, so necessarily $\bar z H = a H$.

It then follows from point \ref{it:lem:products_and_units(ii)} and the above that $ab = \bar z = a u$ for some $u \in H^\times$, which is however a contradiction, as it implies $b \in H^\times$ by point \ref{it:lem:products_and_units(ib)}.
\end{proof}
\begin{theorem}
\label{th:normable_implies_atomic}
Let $H$ be a monoid and $M$ a submonoid of $H$ with $M^\times = M \cap H^\times$. The following hold:
\begin{enumerate}[label={\rm (\roman{*})}]
\item\label{it:th:normable(i)} Suppose that $H$ is unit-cancellative and satisfies the \textup{ACCP}. Then $H$ is atomic.
\item\label{it:th:normable(ii)} If $H$ is unit-cancellative, then so is $M$.
\item\label{it:th:normable(iii)} Assume that $H$ has a length function. Then $H$ is unit-cancellative and satisfies the \textup{ACCP}.
\item\label{it:th:normable(iv)} If $H$ has a length function $\lambda$, then $M$ is a \BF-monoid and $\sup \mathsf L_M(x) \le \lambda(x)$ for all $x \in M$.
\end{enumerate}
\end{theorem}
\begin{proof}
\ref{it:th:normable(i)} We proceed along the lines of the proof of \cite[Proposition 3.1]{Sm13}. To start with, suppose for a contradiction that the set
$$
\Omega := \{Hx: x \in H \setminus H^\times \text{ and } x \text{ is not a (finite) product of atoms of }H\}
$$
is non-empty. Then, using that $H$ satisfies the ACC on principal \textit{left} ideals, $\Omega$ must have a maximal element, say $H\tilde{x}$. Of course, $\tilde{x}$ is neither a unit nor an atom (of $H$), so we get from Lemma \ref{lem:products_and_units}\ref{it:lem:products_and_units(iii)} that $\tilde x = ax$ for some $a \in \mathscr{A}(H)$ and $x \in H$, where we have used that $H$ also satisfies the ACC on principal \textit{right} ideals. This shows that $H\tilde x \subseteq Hx$, and we claim that $H\tilde x \subsetneq Hx$.

Indeed, assume to the contrary that $H\tilde x = Hx$. Then we infer from Lemma \ref{lem:products_and_units}\ref{it:lem:products_and_units(ii)} that $H^\times\tilde x = H^\times x$, and hence $ax = \tilde x = ux$ for some $u \in H^\times$. Yet this is impossible, as it implies by Lemma \ref{lem:products_and_units}\ref{it:lem:products_and_units(ib)} that $a$ is a unit, and hence not an atom, of $H$. Consequently, we see that $H \tilde{x} \subsetneq H x$ (as was claimed).

It follows that $Hx \notin \Omega$, because $H\tilde{x}$ is a $\subseteq$-maximal element of $\Omega$. But we derive from Lemma \ref{lem:basic-properties-atoms-units}\ref{it:lem:basic-properties-atoms-units(ii)} that $x$ is a not a unit of $H$, as $H \tilde x$ means, in particular, that $\tilde x$ is not an atom. Therefore, $Hx \notin \Omega$ only if $x = a_1 \cdots a_n$ for some $a_1, \ldots, a_n \in \mathscr{A}(H)$, which, however, is still a contradiction, since it implies that $\tilde x$ is a product of atoms of $H$ (recall from the above that $\tilde x = ax$).

\ref{it:th:normable(ii)} Assume $H$ is unit-cancellative, and let $x, y \in M$ such that $xy = x$ or $yx = x$ in $M$. Then $xy = x$ or $yx = x$ in $H$ (since $M$ is a submonoid of $H$), and hence $x \in H^\times$. So, using that $M^\times = M \cap H^\times$, it follows that $x \in M^\times$, and we can conclude that $M$ is unit-cancellative.

\ref{it:th:normable(iii)} Let
$\ell: H \to \mathbf{N}$ be a length function on $H$, and
suppose for a contradiction that $H$ is not unit-cancellative. Then there
exist non-units $x,y \in H$ such that $xy = x$ or $yx=x$. But this is impossible, since $xy=x$ implies
$\ell(x) = \ell(xy) > \ell(x)$, and the other case is similar.

So, it remains to prove that $H$ satisfies the ACCP. For, suppose to the contrary that there exists a sequence $(a_n)_{n \ge 1}$ of elements of $H$ such that $a_n H \subsetneq a_{n+1} H$ (respectively, $H a_n \subsetneq H a_{n+1}$) for all $n \in \mathbf N^+$. Then, for each $n \in \mathbf N^+$ we have that $a_n = a_{n+1} v_n$ (respectively, $a_n = v_n a_{n+1}$) for some $v_n \in H \setminus H^\times$, with the result that $\ell(a_{n+1}) < \ell(a_n)$. But this is impossible.

\ref{it:th:normable(iv)}  
Because $u \in M \setminus M^\times$ only if $u \notin H^\times$ (by hypothesis), it is obvious that the restriction of $\lambda$ to $M$ is a length function on $M$. So we have by point \ref{it:th:normable(iii)} that $M$ is unit-cancellative and satisfies the ACCP. Consequently, we obtain from \ref{it:th:normable(i)} that $M$ is atomic, and we are left to show that it is actually \BF.

To this end, let $x \in M \setminus M^\times$, and pick $k \in \mathbf N^+$ and $a_1, \ldots, a_k \in \mathscr{A}(M)$ such that $x = a_1 \cdots a_k$. Since it is immediate from the definition of a length function that $\lambda(a) \ge 1$ for every $a \in M \setminus M^\times$, it is seen by induction that $\lambda(x) \ge k$. Thus $\sup {\sf L}_M(x) \le \lambda(x)$, and we are done.
\end{proof}
We conclude the section with a corollary generalizing \cite[Proposition 1.3.2]{GeHK06} from cancellative, commutative monoids to arbitrary monoids, cf. \cite[Lemma 6.1]{Ge13} and \cite[Lemma 3.6]{Sm16} for some analogues in the cancellative setting; and with an elementary proposition showing that unit-cancellative monoids are Dedekind-finite.
\begin{corollary}
\label{cor:BFness-of-unit-cancellative-monoids}
Let $H$ be a monoid.
The following are equivalent:
\begin{enumerate}[label={\rm (\alph{*})}]
\item\label{it:cor:length_fncs_characterize_BF(i)} $H$ is a \BF-monoid.
\item\label{it:cor:length_fncs_characterize_BF(ii)} $\bigcap_{n \ge 1} (H \setminus H^\times)^n = \emptyset$.
\item\label{it:cor:length_fncs_characterize_BF(iii)} $H$ has a length function.
\end{enumerate}
In particular, if any of these conditions is satisfied, then $H$ is unit-cancellative.
\end{corollary}
\begin{proof}
To ease notation, set $\mathfrak m := H \setminus H^\times$ and $\mathfrak j := \bigcap_{n \ge 1} \mathfrak m^n$. It is sufficient to prove that \ref{it:cor:length_fncs_characterize_BF(i)} $\Rightarrow$ \ref{it:cor:length_fncs_characterize_BF(ii)} and \ref{it:cor:length_fncs_characterize_BF(ii)} $\Rightarrow$ \ref{it:cor:length_fncs_characterize_BF(iii)}, since \ref{it:cor:length_fncs_characterize_BF(iii)} $\Rightarrow$ \ref{it:cor:length_fncs_characterize_BF(i)} is straightforward from Theorem \ref{th:normable_implies_atomic}\ref{it:th:normable(iv)}, while the ``In particular'' part is a consequence of \ref{it:cor:length_fncs_characterize_BF(iii)} and Theorem \ref{th:normable_implies_atomic}\ref{it:th:normable(iii)}.

\ref{it:cor:length_fncs_characterize_BF(i)} $\Rightarrow$ \ref{it:cor:length_fncs_characterize_BF(ii)}: Pick $x \in H$. If $x \in \mathfrak m^n$, then $x = x_1 \cdots x_n$ for some $x_1, \ldots, x_n \in H \setminus H^\times$, which yields $\sup \mathsf L_H(x) \ge n$, because $H$ is BF, and hence $x_i$ is a non-empty product of atoms of $H$ for each $i \in \llb 1, n \rrb$. But $H$ being a BF-monoid also implies that ${\sf L}_H(x)$ is finite. Consequently, there must exist $n \in \mathbf N^+$ such that $x \notin \mathfrak m^n$, with the result that $\mathfrak j = \emptyset$.

\ref{it:cor:length_fncs_characterize_BF(ii)} $\Rightarrow$ \ref{it:cor:length_fncs_characterize_BF(iii)}: To start with, we show that $H$ is unit-cancellative. For, assume to the contrary that $x = xy$ (respectively, $x=yx$) for some $x \in H$ and $y \in H \setminus H^\times$. Then $x$ is a non-unit, and $x = xy^k \in \mathfrak m^{k+1}$ (respectively, $x = y^k x \in \mathfrak m^{k+1}$) for every $k \in \mathbf N$. So $x \in \mathfrak j \ne \emptyset$, a contradiction.

With this in hand, let $\bar x \in H$. Since $\mathfrak j = \emptyset$, there exists $v \in \mathbf N^+$ such that $\bar x \notin \mathfrak m^v$, and we claim that $\bar x \notin \mathfrak m^n$ for every $n \ge v$. Indeed, suppose that $\bar x = x_1 \cdots x_n$ for some $n \ge v$ and $x_1, \ldots, x_n \in H \setminus H^\times$, and set $y_i := x_i$ for $i \in \llb 1, v-1 \rrb$ and $y_v := x_v \cdots x_n$. Then, using that $H$ is unit-cancellative, we obtain from Lemma \ref{lem:products_and_units}\ref{it:lem:products_and_units(i)} that $y_1, \ldots, y_v \in H \setminus H^\times$, and hence $\bar x \in \mathfrak m^v$, which is again a contradiction.

It follows that the function $\lambda: H \to \mathbf N: x \mapsto \sup \fixed[-0.2]{\text{ }}\{n \in \NNb^+: x \in \mathfrak m^n\}$ is well defined, because the set $\{n \in \mathbf N^+: x \in \mathfrak m^n\}$ is finite for all $x \in H$. We want to show that $\lambda$ is a length function (on $H$).

In fact, let $u, v, x, y \in H$ with $y = uxv$. Again by Lemma \ref{lem:products_and_units}\ref{it:lem:products_and_units(i)}, it is clear that $y \in \mathfrak m^{\lambda(u) + \lambda(x) + \lambda(v)}$, where $\mathfrak m^0 := H^\times$. Therefore $\lambda(u) + \lambda(x) + \lambda(v) \le \lambda(y)$, and $\lambda(x) = \lambda(y)$ only if $\lambda(u) = \lambda(v) = 0$, which is, in turn, equivalent to $u, v \in H^\times$. This yields that $\lambda$ is a length function.
\end{proof}
As a side remark, we get from Corollary \ref{cor:BFness-of-unit-cancellative-monoids} that a monoid is BF only if it is unit-cancellative: This extends to the non-commutative setting an observation from the introduction of \cite[\S{ }3]{FGKT}.
\begin{proposition}
\label{prop:Dedekind-finite}
Let $H$ be a monoid. Then $H$ is Dedekind-finite if and only if $x, y \in H^\times$ for all $x, y \in H$ with $xy \in H^\times$. In particular, $H$ is Dedekind-finite if $\mathscr{A}(H) \ne \emptyset$ or $H$ is unit-cancellative.
\end{proposition}
\begin{proof}
Suppose first that $H$ is Dedekind-finite, and let $x,y \in H$ such that $xy \in H^\times$. Then there exists $z \in H$ for which $xyz = zxy = 1_H$. It follows that $yzx = 1_H$, which shows that $x$ and $y$ are units (since they are both left- and right-invertible).

Conversely, assume that $x, y \in H^\times$ whenever $x, y \in H$ and $xy \in H^\times$. Then $xy = 1_H$ for some $x, y \in H$ yields $x,y \in H^\times$, and hence $y = x^{-1} xy = x^{-1}\cdot 1_H = x^{-1}$, which implies $yx = 1_H$.

The ``In particular'' part is now straightforward by the above and Lemmas \ref{lem:basic-properties-atoms-units}\ref{it:lem:basic-properties-atoms-units(i)} and \ref{lem:products_and_units}\ref{it:lem:products_and_units(i)}.
\end{proof}
\section{Power monoids}
\label{sec:power_monoid}
In this and the subsequent section, we apply most of the ideas developed in \S{ }\ref{sec:monoids} to a specific class of structures that are, in a way, ``extremely non-cancellative''. To this end, we make the following:
\begin{definition}
\label{def:power-monoids}
Let $H$ be a (multiplicatively written) monoid. We use $\mathcal P_\fin(H)$ for the set of all \textit{non-empty} finite subsets of $H$, and we denote by $\cdot$ the binary operation
$$
\PPc_{\rm fin}(H) \times \PPc_{\rm fin}(H) \to \PPc_{\rm fin}(H): (X, Y) \mapsto XY,
$$
where $
XY := X \cdot Y := \{xy: (x,y) \in X \times Y\}$.
%is called the product-set of the pair $(X, Y)$,
Moreover, we define
$$
\mathcal P_{\fun}(H) := \{X \in \mathcal P_\fin(H): X \cap H^\times \ne \emptyset\}.
$$
It is trivial that $\PPc_{\rm fin}(H)$, endowed with the above operation, forms a monoid,
with the identity given by the singleton $\{1_H\}$, and $\PPc_{\fun}(H)$ is a sub\-monoid of $\PPc_{\rm fin}(H)$.
Accordingly, we call $\PPc_{\rm fin}(H)$ and $\mathcal P_{\fun}(H)$, respectively, the \textit{power monoid} and \textit{restricted power monoid} of $H$.
\end{definition}
Our goal for the remainder of the paper is, in fact, to investigate some of the algebraic and arithmetic properties of power monoids, and to link them to corresponding properties of the re\-strict\-ed power monoid of $(\mathbf N, +)$, which we will denote by $\mathcal P_{\fuN}(\mathbf{N})$ and always write additively.
We start with a few basic results.
\begin{proposition}
\label{prop:basic-properties-of-power-monoids}
Let $H$ be a monoid. The following hold:
\begin{enumerate}[label={\rm (\roman{*})}]
\item\label{it:prop:basic-properties-of-power-monoids(i)} $\PPc_{\rm fin}(H)$ and $\mathcal P_{\fun}(H)$ are cancellative if and only if $H = \{1_H\}$.
\item\label{it:prop:basic-properties-of-power-monoids(ii)} $\mathcal P_\fin(H)^\times = \mathcal P_\fun(H)^\times = \fixed[-0.15]{\text{ }} \big\{\{u\}: u \in H^\times\bigr\}$.
\item\label{it:prop:basic-properties-of-power-monoids(iv)} Let $H$ be Dedekind-finite. Then $\mathcal P_{\fun}(H)$ is a divisor-closed submonoid of $\mathcal P_\fin(H)$. 
  In particular, $\mathscr{L}(\mathcal P_{\fun}(H)) \subseteq \mathscr{L}(\mathcal P_\fin(H))$ and $\cat(\PPc_{\fun}(H)) \subseteq \cat(\PPc_{\fin}(H))$.
\item\label{it:prop:basic-properties-of-power-monoids(iii)} Let $a \in H$. Then $\{a\}$ is an atom of $\mathcal P_\fin(H)$ only if $a$ is an atom of $H$, and the converse is also true if $H$ is strongly unit-cancellative.
\item\label{it:prop:basic-properties-of-power-monoids(v)} Assume $H$ is cancellative. Then 
$\mathcal P_\fin(H)$ is atomic \textup{(}respectively, a \BF-monoid\textup{)} only if so is $H$.
\end{enumerate}
\end{proposition}
\begin{proof}
\ref{it:prop:basic-properties-of-power-monoids(i)} The ``if'' part is obvious, and of course $\mathcal P_{\fin}(H)$ is cancellative only if so is $\mathcal P_{\fun}(H)$. Therefore, suppose for a contradiction that $\mathcal P_{\fun}(H)$ is cancellative, but $H \ne \{1_H\}$. Accordingly, let $x \in H \setminus \{1_H\}$.
We have $1_H \ne x^2 \ne x$: Otherwise, $\{1_H, x\} \cdot \{1_H\} = \{1_H, x\} \cdot \{1_H, x\}$, yet $\{1_H, x\} \ne \{1_H\}$, which is not possible, by the cancellativity of $\mathcal P_{\fun}(H)$.
But this again leads to a contradiction, because it implies $\{1_H, x^2\} \ne \{1_H, x, x^2\}$, though
$\{1_H, x^2\} \cdot \{1_H, x\} = \{1_H, x, x^2\} \cdot \{1_H, x\}$.

\ref{it:prop:basic-properties-of-power-monoids(ii)} It is trivial that $\big\{\{u\}: u \in H^\times\bigr\} \fixed[-0.15]{\text{ }} \subseteq \mathcal P_\fun(H)^\times \subseteq \mathcal P_\fin(H)^\times$, and it only remains to show $\mathcal P_\fin(H)^\times \subseteq \fixed[-0.15]{\text{ }} \big\{\{u\}: u \in H^\times\bigr\}$. For, let $U \in \mathcal P_\fin(H)^\times$. Then $UV = VU = \{1_H\}$ for some $V \in \mathcal P_\fin(H)$, and hence, for every $u \in U$, there can be found $v, w \in H$ such that $vu = uw = 1_H$. That is, every element of $U$ is left- and right-invertible, and hence invertible. It is thus clear that $UV = \{1_H\}$ only if $1 = |UV| \ge |U|$. So, $U$ is a one-element subset of $H^\times$, and we are done.

\ref{it:prop:basic-properties-of-power-monoids(iv)}
Let $X \in \PPc_{\fun}(H)$ and $Y \in \PPc_\fin(H)$ such that $UYV = X$ for some $U,V \in \PPc_\fin(H)$ (namely, $Y \mid X$ in $\PPc_\fin(H)$), and using that $X \in \PPc_{\fun}(H)$, pick $x \in X \cap H^\times$. Then, $x = uyv$ for some $u \in U$, $y \in Y$, and $v \in V$. So, by Proposition \ref{prop:Dedekind-finite}, $y$ is a unit of $H$, because $H$ is Dedekind-finite.
It follows that $\PPc_{\fun}(H)$ is a divisor-closed submonoid of $\mathcal P_\fin(H)$, and the rest is a consequence of Proposition \ref{prop:divisor-closed-sub}.

\ref{it:prop:basic-properties-of-power-monoids(iii)} Let $a = xy$ for some $x, y \in H \setminus H^\times$. Then $\{a\} = \{x\} \cdot \{y\}$ in $\mathcal P_\fin(H)$, and we have by point \ref{it:prop:basic-properties-of-power-monoids(ii)} that neither $\{x\}$ nor $\{y\}$ is a unit of $\mathcal P_\fin(H)$. So, $\{a\}$ is not an atom of $\mathcal P_\fin(H)$, which proves the ``only if'' part of the claim.

Now, assume that $H$ is strongly unit-cancellative and $\{a\} = XY$ for some non-unit $X, Y \in \mathcal P_\fin(H)$. Accordingly, suppose for a contradiction that $Y \subseteq H^\times$ (the case when $X \subseteq H^\times$ is similar).
Then $|Y| \ge 2$, because every one-element subset of $H^\times$ is a unit of $\mathcal P_\fin(H)$ by point \ref{it:prop:basic-properties-of-power-monoids(ii)}.
In particular, there are $x \in H$ and $y_1, y_2 \in H^\times$ such that $y_1 \ne y_2$ and $a = x y_1 = x y_2$, viz., $xy_1 y_2^{-1} = x$. This, however, is impossible, since $H$ is strongly unit-cancellative.

So, putting it all together, neither $X$ nor $Y$ is a subset of $H^\times$, whence $a = xy$ for some $x \in X \setminus H^\times$ and $y \in Y \setminus H^\times$. To wit, $a$ is not an atom of $H$.

\ref{it:prop:basic-properties-of-power-monoids(v)} This is straightforward from points \ref{it:prop:basic-properties-of-power-monoids(ii)} and \ref{it:prop:basic-properties-of-power-monoids(iii)}, together with the fact that, if $H$ is cancellative, then $|XY| = 1$, for some $X, Y \subseteq H$, implies $|X| = |Y| = 1$.
\end{proof}
\begin{remark}
(i) Proposition \ref{prop:basic-properties-of-power-monoids}\ref{it:prop:basic-properties-of-power-monoids(i)} suggests that 
the basic goal we are pursuing in this section (that is, the study of power monoids from the perspective of factorization theory) is, except for trivial cases, entirely beyond the scope of the factorization theory of \textit{cancellative} monoids.

(ii) The power monoid of a linearly orderable monoid need not be atomic.
In fact, let $H$ be a commutative, linearly orderable, divisible monoid such that $H^\times \ne H$ (e.g., the additive monoid of the non-negative rational numbers). Then the set of atoms of $H$ is empty (cf. Remark \ref{rem:counterexample-to-divisor-closedness}), and since $H \setminus H^\times$ is non-empty, we see that $H$ is not atomic: This proves that, to some extent, Proposition \ref{prop:basic-properties-of-power-monoids}\ref{it:prop:basic-properties-of-power-monoids(v)} is sharp.
\end{remark}
With the above in mind, we look for conditions such that $\mathcal P_\fin(H)$ and $\mathcal P_{\fun}(H)$ are \BF-monoids,
and to this end we make the following:
\begin{definition}
We say that a monoid $H$ is \textit{linearly orderable} if there exists a total order $\preceq$ on $H$ such that $xz \prec yz$ and $zx \prec zy$ for all $x, y, z \in H$ with $x \prec y$, in which case we call the pair $(H, \preceq)$ a \textit{linearly ordered monoid}.
\end{definition}
Every submonoid of a linearly order\-able monoid is still a linearly orderable monoid, and the same is true of any direct product (either finite or infinite) of linearly orderable monoids.
An interesting variety of linearly orderable groups is provided by abelian torsion-free groups, as first proved by Levi in \cite{Levi13}. In a similar vein, Iwasawa \cite{Iwasa48}, Mal'tsev \cite{Malc48}, and Neumann \cite{Neum49} established, independently from each other, that torsion-free nilpotent groups are linearly orderable. Moreover, pure braid groups \cite{RZ98} and free groups \cite{Iwasa48} are linearly orderable, and so are some Baumslag-Solitar groups, which has led to interesting developments in connection to the study of sums of dilates in additive number theory, see \cite{FHLMS15, FHLMPRS16} and references therein.
Further examples are discussed in \cite[Appendix A]{Tr15} and \cite[\S{}1]{PlTr16}.
\begin{proposition}
\label{prop:lin-orderable-H}
Let $H$ be a linearly orderable monoid. Then:
\begin{enumerate}[label={\rm (\roman{*})}]
\item\label{it:prop:lin-orderable-H(0)} $|XY| \ge |X| + |Y| - 1$ for all $X, Y \in \PPc_{\rm fin}(H)$.
\item\label{it:prop:lin-orderable-H(i)} $\dual{H}$ and $\mathcal P_{\fun}(H)$ are strongly unit-cancellative monoids.
\item\label{it:prop:lin-orderable-H(ii)} $\lambda_\times: \mathcal P_{\fun}(H) \to \mathcal P_{\fun}(H): X \mapsto |X| - 1$ is a length function and $\mathcal P_{\fun}(H)$ is a \BF-monoid.
\item\label{it:prop:lin-orderable-H(iii)} $\mathcal P_\fin(H)$ is a \BF-monoid if and only if so is $H$.
\end{enumerate}
\end{proposition}
\begin{proof}
To begin, let $\preceq$ be a total order turning $H$ into a linearly ordered monoid. Given $S \in \mathcal P_{\fin}(H)$, we will denote by $S_\sharp$ and $S^\sharp$, respectively, the minimum and the maximum of $S$ relative to $\preceq$ (which are well defined, since every non-empty finite subset of a totally ordered set has a maximum and a minimum, and these are in fact unique).

\ref{it:prop:lin-orderable-H(0)} Fix $X, Y \in \PPc_{\rm fin}(H)$. The claim is immediate if $X$ or $Y$ is a singleton, as linearly orderable monoids are cancellative. Otherwise, let $x_1, \ldots, x_m$ be the unique enumeration of $X$ with $x_1 \prec \cdots \prec x_m$ and, similarly, $y_1, \ldots, y_n$ the unique enumeration of $Y$ with $y_1 \prec \cdots \prec y_n$. Then it is clear that
$$
x_1 y_1 \prec \cdots \prec x_m y_1 \prec \cdots \prec x_m y_n,
$$
with the result that $|XY| \ge m + n - 1 = |X| + |Y| -1$.

\ref{it:prop:lin-orderable-H(i)} Let $X, Y \in \PPc_{\rm fin}(H)$ such that $XY = X$ (the case when $YX = X$ is similar). Then
$
X_\sharp = (XY)_\sharp = X_\sharp \cdot Y_\sharp
$ and $X^\sharp = X^\sharp \cdot Y^\sharp$,
which is possible if and only if $Y_\sharp = Y^\sharp = 1_H$ (recall that $H$ is cancellative). To wit, $XY = X$ only if $Y = \{1_H\}$. This implies that $\mathcal P_\fin(H)$ is strongly unit-cancellative, and then so is $\mathcal P_{\fun}(H)$, since submonoids of strongly unit-cancellative monoids are strongly unit-cancellative.

\ref{it:prop:lin-orderable-H(ii)} 
Let $X, Y \in \mathcal P_{\fun}(H)$ such that $Y = UXV$, where $U, V \in \mathcal P_{\fun}(H)$ and at least one of $U$ and $V$ is not a unit. Then it follows from point \ref{it:prop:lin-orderable-H(0)} that
$$
\lambda_\times(Y) = |Y| - 1 \ge |U| + |X| + |V| - 3 = \lambda_\times(U) + \lambda_\times(X) + \lambda_\times(V) \ge \lambda_\times(X),
$$
and the last inequality is strict unless $\lambda_\times(U) = \lambda_\times(V) = 0$, that is, $|U| = |V| = 1$. So, knowing from Proposition \ref{prop:basic-properties-of-power-monoids}\ref{it:prop:basic-properties-of-power-monoids(ii)} that $\mathcal P_{\fun}(H)^\times = \big\{\{u\}: u \in H^\times\bigr\} \fixed[-0.15]{\text{ }}$, we find $\lambda_\times(Y) > \lambda_\times(X)$.

In other terms, we have shown that $\lambda_\times$ is a length function on $\mathcal P_{\fun}(H)$. Therefore, we conclude from Corollary \ref{cor:BFness-of-unit-cancellative-monoids}
that $\mathcal P_{\fun}(H)$ is a \BF-monoid.

\ref{it:prop:lin-orderable-H(iii)} The ``only if'' part is a consequence of Proposition \ref{prop:basic-properties-of-power-monoids}\ref{it:prop:basic-properties-of-power-monoids(v)}, in combination with the cancellativity of $H$. As for the other direction, assume $H$ is a \BF-monoid and let $\lambda$ be the function
$$
\mathcal P_\fin(H) \to \mathbf N: X \mapsto |X| + \sup {\sf L}_H(X^\sharp) - 1.
$$
Note that $\lambda$ is well defined, because ${\sf L}_H(x)$ is a finite subset of $\mathbf N$ for every $x \in H$ (by the assumption that $H$ is a \BF-monoid).
We want to prove that $\lambda$ is a length function on $\mathcal P_\fin(H)$, which, as in the proof of point \ref{it:prop:lin-orderable-H(ii)}, will imply that $\mathcal P_\fin(H)$ is a \BF-monoid.

Indeed, let $X, Y \in \mathcal P_\fin(H)$, and suppose that $Y = UXV$ for some $U, V \in \mathcal P_\fin(H)$ with $U \notin \mathcal P_\fin(H)^\times$ or $V \notin \mathcal P_\fin(H)^\times$. In particular, we can assume (by symmetry) that $U \notin \mathcal P_\fin(H)^\times$. By Proposition \ref{prop:basic-properties-of-power-monoids}\ref{it:prop:basic-properties-of-power-monoids(ii)}, this means that either $|U| \ge 2$ or $U = \{x\}$ for some $x \notin H^\times$, and in both cases $|U| + \sup \LLs_H(U^\sharp) \ge 2$. Moreover, it is clear that $|V| + \sup {\sf L}_H(V^\sharp) \ge 1$. Therefore, we get from Remark \ref{rem:factorizations-of-1H} and point \ref{it:prop:lin-orderable-H(0)} that
$$
\lambda(Y) \ge |U| + |X| + |V| + \sup {\sf L}_H(U^\sharp) + \sup \LLs_H(X^\sharp)+ \sup \LLs_H(V^\sharp) - 3 > \lambda(X).
$$
It follows that $\lambda$ is a length function on $\mathcal P_\fin(H)$, and this finishes the proof.
\end{proof}
\begin{corollary}
\label{cor:P_fuN(N)}
$\mathcal P_\fuN(\mathbf{N})$ is a strongly unit-cancellative, reduced, commutative \BF-monoid.
\end{corollary}
\begin{proof}
It is a straightforward consequence of Propositions \ref{prop:basic-properties-of-power-monoids}\ref{it:prop:basic-properties-of-power-monoids(ii)} and \ref{prop:lin-orderable-H}\ref{it:prop:lin-orderable-H(i)}-\ref{it:prop:lin-orderable-H(ii)}, when considering that $(\mathbf{N},+)$ is a linearly orderable, reduced, commutative monoid.
\end{proof}
\begin{remark}
\label{rem:non-atomic_power_monoids}
Let $H$ be a monoid. As a complement to Proposition \ref{prop:lin-orderable-H}\ref{it:prop:lin-orderable-H(ii)}, Antoniou and the second-named author have recently proved that $\mathcal P_{\fin,\times}(H)$ is atomic if and only if $1_H \ne x^2 \ne x$ for every $x \in H \setminus \{1_H\}$, see \cite[Theorems 3.9 and 4.9]{An-Tn-18}; and is BF if and only if $H$ is torsion-free, see \cite[Theorem 3.11(iii)]{An-Tn-18}. On the other hand, we do not know of any analogous characterization of when $\mathcal P_\fin(H)$ is atomic, beside the case covered by Proposition \ref{prop:lin-orderable-H}\ref{it:prop:lin-orderable-H(iii)}.
\end{remark}
To conclude this section, we show that it is possible to understand some properties of $\mathcal P_{\fun}(H)$, under suitable assumptions on the monoid $H$, from the study of $\mathcal P_\fuN(\mathbf N)$, with the advantage that the latter is, in a sense, easier to deal with than the former (cf. Remark \ref{rem:philosophy-of-equimorphisms}).
\begin{theorem}
\label{th:transfer}
Let $H$ be a Dedekind-finite, non-torsion monoid. Then there exists a \textup{(}monoid\textup{)} homo\-mor\-phism $\Phi: \mathcal P_\fuN(\mathbf N) \to \mathcal P_{\fun}(H)$ for which the following holds:
\begin{enumerate}[label={\rm (\textsc{c})}]
\item\label{it:th:transfer:condition(C)} Given $X \in \mathcal P_{\fuN}(\mathbf N)$ and $Y_1, \ldots, Y_n \in \mathcal P_{\fun}(H)$ with $\Phi(X) = Y_1 \cdots Y_n$, there are $X_1, \ldots, X_n \in \mathcal P_{\fuN}(\mathbf N)$ such that $X = X_1 + \cdots + X_n$ and $\Phi(X_i) \simeq_{\PPc_{\fun}(H)} Y_i$ for every $i \in \llb 1, n \rrb$.
\end{enumerate}
In particular, $\Phi$ is an injective equimorphism, and hence we have that ${\sf L}_{\PPc_\fuN(\mathbf N)}(X) = {\sf L}_{\PPc_{\fun}(H)}(\Phi(X))$ and ${\sf c}_{\PPc_{\fun}(H)}(\Phi(X)) \le {\sf c}_{\PPc_\fuN(\mathbf N)}(X)$ for every $X \in \mathcal P_{\fin,0}(\mathbf N)$.
\end{theorem}
\begin{proof}
Using that $H$ is non-torsion, fix $x_0 \in H$ with $\ord_H(x_0) = \infty$, and let $\phi$ be the unique (monoid) homomorphism from $(\mathbf N, +)$ to $H$ with $\phi(1) = x_0$.
Of course, $\phi$ is a monomorphism, because $\phi(x) = \phi(y)$ for some $x, y \in \mathbf N$ with $x < y$ would imply
$
\{x_0^k: k \in \mathbf N^+\} \subseteq \fixed[-0.25]{\text{ }}\bigl\{x_0^k: k \in \llb 0, y-1 \rrb\bigr\}$,
in contradiction to the fact that $\ord_H(x_0) = \infty$. Moreover, we can clearly lift $\phi$ to a monomorphism $\Phi: \mathcal P_\fuN(\mathbf N) \to \mathcal P_{\fun}(H)$ by taking $\Phi(X) := \{\phi(x): x \in X\}$ for every $X \in \mathcal P_\fuN(\mathbf N)$.

To see that $\Phi$ satisfies condition \ref{it:th:transfer:condition(C)}, let $X \in \mathcal P_{\fuN}(\mathbf N)$ and $Y_1, \ldots, Y_n \in \mathcal P_{\fun}(H)$ such that $\Phi(X) = Y_1 \cdots Y_n$. Since $\phi$ is a homo\-mor\-phism and $0 \in X$, there exist $u_1 \in Y_1, \ldots, u_n \in Y_n$ for which $u_1 \cdots u_n = \phi(0) = 1_H$, and we get from Proposition \ref{prop:Dedekind-finite} that $u_1, \ldots, u_n \in H^\times$ (recall that, by hypothesis, $H$ is Dedekind-finite). Set, for every $i \in \llb 1, n \rrb$,
$$
Y_i^\prime := u_0 \cdots u_{i-1} Y_i \fixed[0.1]{\text{ }} u_i^{-1} \cdots u_1^{-1},
\quad\text{with } u_0 := 1_H.
$$
It is straightforward that $\Phi(X) = Y_1^\prime \cdots Y_n^\prime$, and of course $Y_i^\prime \simeq_{\PPc_{\fun}(H)} Y_i$ for each $i \in \llb 1, n \rrb$. Further, $1_H \in \bigcap_{i=1}^n Y_i^\prime$, with the result that $Y_1^\prime, \ldots, Y_n^\prime \subseteq \Phi(X)$. Thus, since $\Phi$ is injective, there exist $X_1, \ldots, X_n \subseteq X$ with $0 \in X_i$ and $\Phi(X_i) = Y_i^\prime \simeq_{\PPc_{\fun}(H)} Y_i$ for all $i \in \llb 1, n \rrb$. So $\Phi(X) = \Phi(X_1+ \cdots + X_n)$, and hence $X = X_1 + \cdots + X_n$ (again by the injectivity of $\Phi$). To wit, $\Phi$ satisfies condition \ref{it:th:transfer:condition(C)}, as was desired.

We are left to show that $\Phi$ is an equimorphism, as all the rest will follow from points \ref{it:prop:cotransfer-hom(1)} and \ref{it:prop:cotransfer-hom(4)} of Theorem \ref{th:cotransfer-hom}. For, it is clear from the above that $\Phi$ satisfies conditions \ref{covariant-transfer(1)} and \ref{covariant-transfer(3)} of Definition \ref{def:equimorphisms}. Therefore, it will be enough to prove that $\Phi$ is atom-preserving. 

To this end, let $A \in \mathscr{A}(\mathcal P_\fuN(\mathbf N))$, and assume first that $\Phi(A) = X^\prime Y^\prime$ for some $X^\prime, Y^\prime \in \mathcal P_{\fun}(H)$. Then, we derive from condition \ref{it:th:transfer:condition(C)} that there are $X, Y \in \mathcal P_{\fuN}(\mathbf N)$ such that 
$$
\Phi(X) \simeq_{\PPc_{\fun}(H)} X^\prime,
\quad
\Phi(Y) \simeq_{\PPc_{\fun}(H)} Y^\prime,
\quad\text{and}\quad
A = X + Y, 
$$
which can only happen if one of $X$ and $Y$ is $\{0\}$, since $\mathcal P_{\fuN}(\mathbf N)$ is a reduced \BF-monoid (Corollary \ref{cor:P_fuN(N)}) and $A$ is an atom of $\mathcal P_{\fuN}(\mathbf N)$. Accordingly, $X^\prime \simeq_{\PPc_{\fun}(H)} \{1_H\}$ or $Y^\prime \simeq_{\PPc_{\fun}(H)} \{1_H\}$, and hence one of $X^\prime$ and $Y^\prime$ is a unit of $\mathcal P_\fun(H)$.

On the other hand, suppose for a contradiction that $\Phi(A) \in \mathcal P_\fun(H)^\times$. Then, we have
by Proposition \ref{prop:basic-properties-of-power-monoids}\ref{it:prop:basic-properties-of-power-monoids(ii)} that $|\Phi(A)| = 1$. So $A$ is a singleton (recall that $\Phi$ is injective), and hence $A = \{0\}$. This, however, is impossible, because $A$ is an atom of $\mathcal P_\fuN(\mathbf N)$.
\end{proof}
\section{Some additive combinatorics}
\label{sec:the_case_of_integers}
We derive from Proposition \ref{prop:basic-properties-of-power-monoids} and Theorem \ref{th:transfer} that, in many relevant cases, the arithmetic of power monoids is ``controlled by the combinatorial structure'' of the integers, as algebraically encoded by the restricted power monoid of $(\mathbf N, +)$, which we continue to denote by $\mathcal P_{\fuN}(\mathbf{N})$ and to write additively (as in \S{ }\ref{sec:power_monoid}).
As a consequence, we are led here to consider various properties of (finite) subsets of $\mathbf N$ that can or cannot be split into a sumset in a non-trivial way. We start by identifying some families of atoms of $\mathcal P_\fuN(\mathbf{N})$.
\begin{proposition}
\label{prop:valuations}
Let $X, Y_1, \ldots, Y_n \in \mathcal P_\fuN(\mathbf N)$ such that $X = Y_1 + \cdots + Y_n$. We have:
\begin{enumerate}[label={\rm (\roman{*})}]
\item\label{it:prop:valuations(i)} $Y_i \subseteq X$ for every $i \in \llb 1, n \rrb$.
\item\label{it:prop:valuations(ii)} ${\sf L}(q \cdot \fixed[-0.3]{\text{ }} X) = {\sf L}(X)$ for every $q \in \mathbf N^+$.
\item\label{it:prop:valuations(iii)} If $X^+ \ne \emptyset$, then $\min X^+ \in Y_i$ for some $i \in \llb 1, n \rrb$.
\item\label{it:prop:valuations(iv)} Every $2$-element set in $\mathcal P_\fuN(\mathbf N)$ is an atom.
\item\label{it:prop:valuations(v)} A $3$-element set $A \in \mathcal P_\fuN(\mathbf N)$ is not an atom if and only if $A = \{0, x, 2x\}$ for some $x \in \mathbf N^+$.
\end{enumerate}
\end{proposition}
\begin{proof}
	\ref{it:prop:valuations(i)}-\ref{it:prop:valuations(iii)} are straightforward (we leave it as an exercise for the reader to fill in the details).

\ref{it:prop:valuations(iv)} Let $x \in \mathbf N^+$ and assume that $\{0, x\} = X + Y$ for some $X, Y \in \mathcal P_\fuN(\mathbf N)$. Since $(\mathbf N, +)$ is a linearly orderable monoid, we get from Proposition \ref{prop:lin-orderable-H}\ref{it:prop:lin-orderable-H(0)} that $2 \ge |X| + |Y| - 1$, which is possible only if $X$ or $Y$ is a singleton. Together with Proposition \ref{prop:basic-properties-of-power-monoids}\ref{it:prop:basic-properties-of-power-monoids(ii)}, this proves that $\{0, x\}$ is an atom of $\mathcal P_\fuN(\mathbf N)$.

\ref{it:prop:valuations(v)} The ``if'' clause is trivial. As for the other direction, let $X, Y \in \mathcal P_{\fuN}(\mathbf N)$ such that $A := X + Y$ is a $3$-element set, but neither $X$ nor $Y$ is a unit. Then it is easily seen from Propositions \ref{prop:basic-properties-of-power-monoids}\ref{it:prop:basic-properties-of-power-monoids(ii)} and \ref{prop:lin-orderable-H}\ref{it:prop:lin-orderable-H(0)} that $|X| = |Y| = 2$, i.e., $X = \{0, x\}$ and $Y = \{0, y\}$ for some $x, y \in \mathbf N^+$. It follows $A=\{0, x, y, x+y\}$, which is only possible if $x = y$ (because $|A| = 3$ and $1 \le x, y < x+y$), so that $A = \{0, x, 2x\}$.
\end{proof}
\begin{proposition}
\label{prop:large_atoms}
Let $d, \ell, q \in \mathbf N^+$ and $A \in \mathcal P_\fin(\mathbf N)$ such that $\min(d, \min A) > \ell q$ and $x \equiv y \bmod d$ for all $x, y \in A$. Then $\bigl(q \cdot \llb 0, \ell \rrb\bigr) \cup A \notin \mathscr A(\mathcal P_\fuN(\mathbf{N}))$ if and only if $A = \{(\ell+k)q\}$ for some $k \in \fixed[-0.3]{\text{ }}\bigl\llb 1, \lceil \ell/2 \rceil \bigr\rrb$.
\end{proposition}
\begin{proof}
Set $B := \bigl(q \cdot \llb 0, \ell \rrb\bigr) \cup A$, and suppose first that $A = \{(\ell+k)q\}$ for some $k \in \fixed[-0.3]{\text{ }}\bigl\llb 1, \lceil \ell/2 \rceil \bigr\rrb$.
Then $k \le \ell - k + 1$ and $B = \{0, qk\} + C$ with $C := \bigl(q \cdot \llb 0, \ell-k \rrb\bigr) \cup \{\ell q\}$. Therefore, $B$ is not an atom of $\mathcal P_\fuN(\mathbf{N})$, because it is the sum of two elements of $\mathcal{P}_\fuN(\mathbf N)$ both different from $\{0\}$ (recall that $\mathcal P_\fuN(\mathbf{N})$ is a reduced monoid). So the ``if'' part of the statement is proved.

As for the other direction, let $B = X + Y$ for some non-zero $X, Y \in \PPc_\fuN(\mathbf N)$ (so, both $X^+$ and $Y^+$ are non-empty), and set $x_{\fixed[0.1]{\text{ }}\rm M} := \max X$ and $y_{\fixed[0.1]{\text{ }}\rm M} := \max Y$.
By symmetry, we can assume $1 \le x_{\fixed[0.1]{\text{ }}\rm M} \le y_{\fixed[0.1]{\text{ }}\rm M}$.

We claim $y_{\fixed[0.1]{\text{ }}\rm M} \le \ell q$. Indeed, suppose the contrary, and define $x_{\fixed[0.1]{\text{ }} \rm m} := \min X^+$.
Then $y_{\fixed[0.1]{\text{ }}\rm M}, x_{\fixed[0.1]{\text{ }} \rm m} + y_{\fixed[0.1]{\text{ }}\rm M} \in A$, which is only possible if $x_{\fixed[0.1]{\text{ }} \rm m} \ge d \ge \ell q+1$, since $x_{\fixed[0.1]{\text{ }} \rm m} = (x_{\fixed[0.1]{\text{ }} \rm m} + y_{\fixed[0.1]{\text{ }}\rm M}) - y_{\fixed[0.1]{\text{ }}\rm M} \ge 1$ and $x \equiv y \bmod d$ for all $x, y \in A$ (by hypothesis). 
In addition, points \ref{it:prop:valuations(i)} and \ref{it:prop:valuations(iii)} of Proposition \ref{prop:valuations} imply $q \in X \cup Y$ (note that $q = \min B^+$). So $x_m + q$ or $q + y_{\fixed[0.1]{\text{ }}\rm M}$ is in $X+Y = B$, and hence in $A$ (since $x_m + q$ and $q + y_{\fixed[0.1]{\text{ }}\rm M}$ are both $\ge \ell q + 1$). This is however a contradiction, in that 
$
q = (x_{\fixed[0.1]{\text{ }} \rm m} + q) - x_{\fixed[0.1]{\text{ }} \rm m} = (q + y_{\fixed[0.1]{\text{ }}\rm M}) - y_{\fixed[0.1]{\text{ }}\rm M}$ and $1 \le q < d$, but any two distinct elements in $A$ should have a distance $\ge d$.

It follows $B \subseteq \llb 0, x_{\fixed[0.1]{\text{ }}\rm M} + y_{\fixed[0.1]{\text{ }}\rm M} \rrb \subseteq \llb 0, 2\ell q \rrb$, hence $A \subseteq \llb \ell q+1, 2\ell q \rrb$. Since $1 \le 2\ell q - (\ell q + 1) < d$ and $A$ is non-empty, we can therefore conclude $|A| = 1$ (using again that $a \equiv b \bmod d$ for all $a, b \in A$). On the other hand, $y_{\fixed[0.1]{\text{ }}\rm M} \le \ell q$ yields, along with Proposition \ref{prop:valuations}\ref{it:prop:valuations(i)}, that $X, Y \subseteq q \cdot \llb 0, \ell \rrb$, hence $A \subseteq X + Y \subseteq q \cdot \mathbf N$.

So, putting it all together, we see that $A = \{(\ell+k)q\}$ for some $k \in \llb 1, \ell \rrb$.
Suppose for a contradiction that $\lceil \ell/2 \rceil < k \le \ell$. Then $k \ge 2$, and of course $x_{\fixed[0.1]{\text{ }}\rm M} \ge kq$, otherwise we would obtain
$$
(\ell+k)q = \max B = x_{\fixed[0.1]{\text{ }}\rm M} + y_{\fixed[0.1]{\text{ }}\rm M} < kq + \ell q,
$$
which is impossible. Moreover, we claim that
\begin{equation}
\label{equ:emptyness}
X \cap \llb x_{\fixed[0.1]{\text{ }}\rm M} - kq + 1, x_{\fixed[0.1]{\text{ }}\rm M} - 1 \rrb = Y \cap \llb y_{\fixed[0.1]{\text{ }}\rm M} - kq + 1, y_{\fixed[0.1]{\text{ }}\rm M} - 1 \rrb = \emptyset.
\end{equation}
In fact, if $x \in X \cap \llb x_{\fixed[0.1]{\text{ }}\rm M} - kq + 1, x_{\fixed[0.1]{\text{ }}\rm M} - 1 \rrb \ne \emptyset$ (the other case is similar), then $x + y_{\fixed[0.1]{\text{ }}\rm M}, (\ell+k)q \in B$ (as was already noted, we have $(\ell+k)q = x_{\fixed[0.1]{\text{ }}\rm M} + y_{\fixed[0.1]{\text{ }}\rm M}$), and actually
$$
(\ell+k)q > x + y_{\fixed[0.1]{\text{ }}\rm M} \ge (\ell+k)q - kq + 1 = \ell q+1.
$$
We thus get $x + y_{\fixed[0.1]{\text{ }}\rm M}, (\ell+k)q \in A$, which is impossible (since $A$ is a singleton) and leads to \eqref{equ:emptyness}.

Accordingly, we find that $X \subseteq \llb 0, x_{\fixed[0.1]{\text{ }}\rm M} - kq \rrb \cup \{x_{\fixed[0.1]{\text{ }}\rm M}\}$ and $Y \subseteq \llb 0, y_{\fixed[0.1]{\text{ }}\rm M} - kq \rrb \cup \{y_{\fixed[0.1]{\text{ }}\rm M}\}$, whence
\begin{equation}
\label{equ:holes}
q \cdot \llb 0, \ell \rrb = B \setminus \{(\ell+k)q\} = (X + Y) \setminus \{(\ell+k)q\} \subseteq \bigl(\llb 0, (\ell-k)q \rrb \cup \llb x_{\fixed[0.1]{\text{ }}\rm M}, \ell q \rrb\bigr) \cap (q \cdot \mathbf N).
\end{equation}
However, this is still a contradiction, because $x_{\fixed[0.1]{\text{ }}\rm M} - (\ell-k)q \ge kq - (\ell-k)q = (2k - \ell)q \ge 2q$, with the result that at least one multiple of $q$ in the interval $\llb 0, \ell q \rrb$ is missing from the right-most side of \eqref{equ:holes}.
\end{proof}
\begin{proposition}
\label{prop:atomic-unions}
Let $A \in \mathcal P_\fuN(\mathbf N)$ and $b, c \in \mathbf N^+$, and assume $b \ne 2c$ and $2 \sup A < c < b - \sup A$. Then $A \cup (A + b) \cup \{c\}$ is an atom of $\mathcal P_\fuN(\mathbf N)$.
\end{proposition}
\begin{proof}
Suppose to the contrary that $B := A \cup (A+b) \cup \{c\}$ is not an atom of $\mathcal P_{\fuN}(\mathbf N)$, i.e., there are $X, Y \in \mathcal P_{\fuN}(\mathbf N)$ such that $B = X + Y$ and $|X|, |Y| \ge 2$. We claim $c \in X \cup Y$.

Otherwise, $c = \bar{x} + \bar{y}$ for some $\bar{x} \in X^+$ and $\bar{y} \in Y^+$, and this can only happen if $\bar{x}, \bar{y} \in A$, because $X, Y \subseteq B$ and $\inf(A+b) = b > c$. It follows that $c = \bar{x} + \bar{y} \le 2 \sup A$, which is a contradiction, since $2 \sup A < c$ (by hypothesis). Thus $c \in X \cup Y$ (as claimed), and by symmetry we can assume $c \in X$.

Then $Y^+ \subseteq (A+b) \cup \{c\}$, because the assumptions made on $A$, $b$, and $c$ imply that $\sup A < c + a \le c + \sup A < b$ for all $a \in A$. In turn, this yields $Y = \{0, c\}$, since $Y \cap (A + b) \ne \emptyset$ would imply 
$$
\sup B = \sup X+ \sup Y \ge c + b > \sup(A+b) = \sup B, 
$$
which is, of course, impossible. So $B = X \cup (X + c)$, and therefore $X^+ \ne \emptyset$, because $\{0, b, c\} \subseteq B$ and $0 < c < b$ (whereas $X^+ = \emptyset$ would give $|B| = 2$). Accordingly, set $\bar z := \inf X^+$.  

Then $\bar z + c \in B$, and actually $\bar z + c \in A + b$, as we see by consider\-ing that $\bar z + c > c > \sup A$. But this is only possible if $A^+ = \emptyset$; otherwise, we would get from the above and Proposition \ref{prop:valuations}\ref{it:prop:valuations(iii)} that $\bar z = \inf A^+$, and hence $\bar z + c \le \sup A + c < b$, in contrast to the fact that $\bar z + c \in A + b$.

So putting it all together, we obtain that $B = \{0, b, c\}$, which, however, is still a contradiction, because $\{0, b, c\}$ is an atom, by Proposition \ref{prop:valuations}\ref{it:prop:valuations(v)} and the assumption that $2c \ne b > c$.
\end{proof}
We will also need a series of lemmas, the last of which (Lemma \ref{lem:fund-lemma}) is of crucial importance for the goals that we are pursuing (as summarized in \S{ }\ref{subsec:plan}).
\begin{lemma}
\label{lem:sidonicity}
Let $\alpha_1, \beta_1, \ldots, \alpha_n, \beta_n \in \mathbf N$ and $u_1, \ldots, u_n \in \mathbf N^+$ such that $\sum_{j=1}^i u_i \cdot \max(\alpha_i, \beta_i) < u_{i+1}$ for every $i \in \llb 1, n - 1 \rrb$.
Then $\sum_{i=1}^{n} \alpha_i u_i = \sum_{i=1}^{n} \beta_i u_i$ if and only if $\alpha_i = \beta_i$ for all $i \in \llb 1, n \rrb$.
\end{lemma}
\begin{proof}
The ``if'' part is obvious. As for the other, assume $\sum_{i=1}^n \alpha_i u_i = \sum_{i=1}^n \beta_i u_i$, set $E := \{i \in \llb 1, n \rrb: \alpha_i \ne \beta_i\}$, and suppose for a contradiction that $E \ne \emptyset$. Accordingly, let $i_0 := \max E$;
by symmetry, we can admit that $\alpha_{i_0} < \beta_{i_0}$. Then $\alpha_i = \beta_i$ for $i \in \llb i_0 + 1, n \fixed[0.2]{\text{ }} \rrb$, and we have
$\sum_{i=1}^{i_0} \alpha_i u_i = \sum_{i=1}^{i_0} \beta_i u_i$. This is however impossible, since our assumptions imply that
\begin{equation*}
\sum_{i=1}^{i_0} \alpha_i u_i \le \sum_{i=1}^{i_0 - 1} \alpha_i u_i + (\beta_{i_0} - 1) u_{i_0} < \beta_{i_0} u_{i_0} \le \sum_{i=1}^{i_0} \beta_i u_i.\qedhere
\end{equation*}
\end{proof}
\begin{lemma}
	\label{lem:basic}
	Given $u_1, \ldots, u_{n+1} \in \mathbf N^+$ such that
	\begin{enumerate*}[label={\rm (\alph{*})}]
		\item\label{it:lem:basic(a)} $u_1 + \cdots + u_i < \frac{1}{2} u_{i+1}$ for every $i \in \llb 1, n-1 \rrb$ and
		\item\label{it:lem:basic(b)} $2u_n < u_{n+1}$,
	\end{enumerate*}
	assume that $\sum_{i \in \mathcal{I}} u_i = \sum_{j \in \mathcal{J}} u_j + \sum_{k \in \mathcal{K}} u_k$ for some $\mathcal{I}, \mathcal{J}, \mathcal{K} \subseteq \llb 1, n+1 \rrb$. Then either $\mathcal{I} = \mathcal{J} \uplus \mathcal{K}$, or $n \in (\mathcal{J} \cap \mathcal{K}) \setminus \mathcal{I}$ and $n+1 \in \mathcal{I} \setminus (\mathcal{J} \cup \mathcal{K})$.
\end{lemma}
\begin{proof}
	Set $x := \sum_{j \in \mathcal{J}} u_j$, $y := \sum_{k \in \mathcal{K}} u_k$, and $z := \sum_{i \in \mathcal{I}} u_i$. We denote by $\delta_S$, for a fixed $S \subseteq \mathbf N$, the function $\mathbf N \to \{0, 1\} \subseteq \mathbf N$ defined by $\delta_S(i) := 1$ if $i \in S$ and $\delta_S(i) := 0$ otherwise. 
	Accordingly, for each $i \in \llb 1, n + 1 \rrb$ we take $\alpha_i := \delta_\mathcal{J}(i)$, $\beta_i := \delta_\mathcal{K}(i)$, and $\gamma_i := \delta_\mathcal{I}(i)$; and we let $E := \fixed[-0.15]{\text{ }}\bigl\{i \in \llb 1, n+1 \rrb: \alpha_i + \beta_i \ne \gamma_i\bigr\}$. We distinguish two cases:
	\vskip 0.1cm
	\textsc{Case 1:} $E = \emptyset$. We have $
	\alpha_i + \beta_i = \gamma_i$ for every $i \in \llb 1, n+1 \rrb$,
	which is clearly possible if and only if $\mathcal{I} = \mathcal{J} \uplus \mathcal{K}$ (in particular, note that $\gamma_i \le 1 < 2 = \alpha_i + \beta_i$ for every $i \in \mathcal{J} \cap \mathcal{K}$).
	\vskip 0.1cm
	\textsc{Case 2:} $E \ne \emptyset$. Let $i_0 := \max E$. Since $\alpha_i + \beta_i = \gamma_i$ for $i \in \llb i_0 + 1, n+1 \rrb$, we have
	\begin{equation*}
	\label{equ:remove-the-identical-parts}
	\sum_{i=1}^{i_0} \gamma_i u_i = \sum_{i=1}^{i_0} (\alpha_i + \beta_i) u_i.
	\end{equation*}
	On the other hand, we derive from \ref{it:lem:basic(a)} and \ref{it:lem:basic(b)} that $u_1 + \cdots + u_i < 2u_i < u_{i+1}$ for all $i \in \llb 1, n \rrb$. Thus, it is immediate that $\alpha_{i_0} + \beta_{i_0} < \gamma_{i_0}$; otherwise,
	\begin{equation*}
	\sum_{i=1}^{i_0} (\alpha_i + \beta_i) u_i \ge (\gamma_{i_0} + 1) u_{i_0} > \gamma_{i_0} u_{i_0} + \sum_{i=1}^{i_0-1} u_i \ge \sum_{i=1}^{i_0} \gamma_i u_i,
	\end{equation*}
	a contradiction. So $\alpha_{i_0} = \beta_{i_0} = 0$ and $\gamma_{i_0} = 1$; moreover, we must have $i_0 = n+1$, or else
	\begin{equation*}
	\sum_{i=1}^{i_0} (\alpha_i + \beta_i) u_i \le (\gamma_{i_0} - 1) u_{i_0} + \sum_{i=1}^{i_0-1} (\alpha_i + \beta_i) u_i \le (\gamma_{i_0} - 1) u_{i_0} + 2\sum_{i=1}^{i_0-1} u_i
	\fixed[-0.25]{\text{ }}\stackrel{\ref{it:lem:basic(a)}}{< }\fixed[-0.25]{\text{ }}
	\gamma_{i_0} u_{i_0} \le \sum_{i=1}^{i_0} \gamma_i u_i,
	\end{equation*}
	which is still impossible.
	It follows that $\gamma_n = 0$ and $\alpha_n = \beta_n = 1$, since
	\begin{equation*}
	\begin{split}
	\gamma_n u_n + u_{n+1} & \le \sum_{i=1}^{n+1} \gamma_i u_i = x+y
	\le (\alpha_n + \beta_n) u_n + 2\sum_{i=1}^{n-1} u_i \\
	& \stackrel{\ref{it:lem:basic(a)}}{< }\fixed[-0.25]{\text{ }}
	(\alpha_n + \beta_n + 1)u_n
	\fixed[-0.25]{\text{ }}\stackrel{\ref{it:lem:basic(b)}}{< }\fixed[-0.25]{\text{ }}
	u_n \cdot \min(\alpha_n, \beta_n) + u_{n+1}.
	\end{split}
	\end{equation*}
	To wit, $n \in (\mathcal{J} \cap \mathcal{K}) \setminus \mathcal{I}$ and $n+1 \in \mathcal{I} \setminus (\mathcal{J} \cup \mathcal{K})$.
\end{proof}
\begin{lemma}
	\label{lem:pre-sidonicity}
	Let $u_1, \ldots, u_{n+1} \in \mathbf N^+$ such that
	\begin{enumerate*}[label={\rm (\alph{*})}]
		\item\label{it:lem:pre-sidonicity(a)} $u_1 + \cdots + u_i < \frac{1}{2} u_{i+1}$ for every $i \in \llb 1, n-1 \rrb$ and
		\item\label{it:lem:pre-sidonicity(b)} $2u_n < u_{n+1}$.
	\end{enumerate*}
	Assume in addition that 
	\begin{equation}\label{equ:decomposition}
	\sum_{i=1}^{n+1} \{0, u_i\} = X + Y,
	\quad\text{for some } X, Y \in \mathcal P_{\fin,0}(\mathbf N);
	\end{equation}
	and set $I_X := \bigl\{i \in \llb 1, n+1 \rrb: u_i \in X\bigr\}$ and
	$I_Y := \bigl\{i \in \llb 1, n+1 \rrb: u_i \in Y \bigr\}$. The following hold:
	\begin{enumerate}[label={\rm (\roman{*})}]
		\item\label{it:lem:pre-sidonicity(i)} $\llb 1, n+1 \rrb = I_X \uplus I_Y$.
		\item\label{it:lem:pre-sidonicity(iii)} If $\sum_{j \in J} u_j \in X$ for some $J \subseteq \llb 1, n+1 \rrb$, then $J \setminus \{n\} \subseteq I_X$ \textup{(}and similarly for $Y$\textup{)}.
		\item\label{it:lem:pre-sidonicity(ii)} If $\sum_{j \in J} u_j \in X$, $\sum_{k \in K} u_k \in Y$, and $\sum_{i \in I} u_i = \sum_{j \in J} u_j + \sum_{k \in K} u_k$ for some $I, J, K \subseteq \llb 1, n+1 \rrb$, then either $I = J \uplus K$; or $n \notin I$, $J \cap K = \{n\}$, and $n+1 \in I \setminus (J \cup K)$.
	\end{enumerate}
\end{lemma}
\begin{proof}
\ref{it:lem:pre-sidonicity(i)} As in the proof of Lemma \ref{lem:basic}, it is easy to derive from conditions \ref{it:lem:pre-sidonicity(a)} and \ref{it:lem:pre-sidonicity(b)} that
\begin{equation}
\label{equ:easy-inequality:lem:pre-sidonicity}
u_i \le u_1 + \cdots + u_i < 2u_i < u_{i+1},
\quad \text{for all }i \in \llb 1, n \rrb.
\end{equation}
Consequently, we see (from \eqref{equ:decomposition}) that $2u_i \notin X + Y$, and hence $u_i \notin X \cap Y$, for all $i \in \llb 1, n+1 \rrb$.

Besides that, let $i_0 \in \llb 1, n \rrb$. Since $u_{i_0} \in X+Y$, there are $\mathcal J, \mathcal K \subseteq \llb 1, n+1 \rrb$ with $x := \sum_{j \in \mathcal J} u_j \in X$, $y := \sum_{k \in \mathcal K} u_k \in Y$, and $x + y = u_{i_0}$. Thus, we obtain from Lemma \ref{lem:basic} (applied with $\mathcal I = \{i_0\}$, $\mathcal J = J$, and $\mathcal K = K$) that $J \uplus K = \{i_0\}$, which is only possible if $J = \emptyset$ or $K = \emptyset$, namely, $u_{i_0} \in X \cup Y$. This, together with \eqref{equ:easy-inequality:lem:pre-sidonicity}, shows that
$$
\llb 1, n \rrb \subseteq I_X \cup I_Y
\quad\text{and}\quad
I_X \cap I_Y = \emptyset.
$$
In particular, we can assume (without loss of generality) that $u_n \in X
$, and it only remains to prove that $n+1 \in I_X \cup I_Y$. For, suppose the contrary and set $U := \left\{\sum_{i \in I} u_i: I \subseteq \llb 1, n-1\rrb\right\}$.
We distinguish two cases.
\vskip 0.1cm
\textsc{Case 1:} $Y \cap (U + u_n) = \emptyset$.
Because $u_{n+1} \in (X+Y) \setminus (X \cup Y)$, we must have that
$x + y = u_{n+1}$ for some $x \in X$ and $y \in Y$ with $x, y < u_{n+1}$. So we obtain that
$$
u_{n+1} = x+y
\fixed[-0.25]{\text{ }}\stackrel{\eqref{equ:easy-inequality:lem:pre-sidonicity}}{\le }\fixed[-0.25]{\text{ }}
\sum_{i=1}^n u_i + \sum_{i+1}^{n-1} u_i = u_n + 2 \sum_{i=1}^{n-1} u_i
\fixed[-0.25]{\text{ }}\stackrel{\ref{it:lem:pre-sidonicity(a)}}{< }\fixed[-0.25]{\text{ }}
2u_n
\fixed[-0.25]{\text{ }}\stackrel{\ref{it:lem:pre-sidonicity(b)}}{< }\fixed[-0.25]{\text{ }}
u_{n+1},
$$
which is impossible and completes the analysis of the present case.
\vskip 0.1cm
\textsc{Case 2:} $Y \cap (U + u_n) \ne \emptyset$.
Since $u_n + u_{n+1} \in X + Y$ and $u_{n+1} \notin X \cup Y$, there exist two index sets $J, K \subseteq \llb 1, n+1 \rrb$, none of which is equal to $\{n+1\}$, such that $x := \sum_{j \in J} u_j \in X$, $y := \sum_{k \in K} u_k \in Y$, and $x + y = u_n + u_{n+1}$. It follows that $n+1 \in J \cup K$; otherwise,
$$
x+y \le 2\sum_{i=1}^n u_i = 2u_n + 2\sum_{i=1}^{n-1} u_i
\fixed[-0.25]{\text{ }}\stackrel{\ref{it:lem:pre-sidonicity(b)}}{< }\fixed[-0.25]{\text{ }}
u_{n+1} + 2\sum_{i=1}^{n-1} u_i \fixed[-0.25]{\text{ }}\stackrel{\ref{it:lem:pre-sidonicity(a)}}{< }\fixed[-0.25]{\text{ }}
u_{n+1} + u_n,
$$
a contradiction. Therefore, we apply Lemma \ref{lem:basic} (with $\mathcal I = \{n,n+1\}$, $\mathcal J = J$, and $\mathcal K = K$) to find that $\{n,n+1\} = J \uplus K$.
On the other hand, recalling that $u_n \in X$ and $Y \cap (U + u_n) \ne \emptyset$, and taking $K_0$ to be any subset of $\llb 1, n+1 \rrb$ such that $n \in K_0$ and $\sum_{k \in K_0} u_k \in Y$, we get again from Lemma \ref{lem:basic} (applied first with $\mathcal{J} = J$ and $\mathcal{K} = K_0$, then with $\mathcal{J} = \{n\}$ and $\mathcal{K} = K$) that neither $J$ nor $K$ can be equal to $\{n,n+1\}$. But since $\{n,n+1\} = J \uplus K$, this is only possible if $J = \{n+1\}$ or $K = \{n+1\}$, and hence $u_{n+1} \in X \cup Y$, which is still a contradiction.
\vskip 0.1cm
\ref{it:lem:pre-sidonicity(iii)} Suppose that $x := \sum_{j \in J} u_j \in X$ for some $J \subseteq \llb 1, n+1 \rrb$, but $J \setminus \{n\} \not \subseteq I_X$, i.e., there exists an index $i \in J \setminus \{n\}$ such that $i \notin I_X$. Then $i \in I_Y$, by point \ref{it:lem:pre-sidonicity(i)}. So $x + u_i \in X + Y$, in contradiction to Lemma \ref{lem:basic} (applied with $\mathcal J = J$ and $\mathcal K = \{i\}$).
\vskip 0.1cm
\ref{it:lem:pre-sidonicity(ii)} Set $x := \sum_{j \in J} u_j$ and $y := \sum_{k \in K} u_k$, and assume that $\sum_{i \in I} u_i = x+y$, but $I \ne J \uplus K$. Then Lemma \ref{lem:basic} (applied with $\mathcal I = I$, $\mathcal J = J$, and $\mathcal K = K$) yields $n \in (J \cap K) \setminus I$ and $n+1 \in I \setminus (J \cup K)$. It thus follows from \ref{it:lem:pre-sidonicity(iii)} that $J \setminus \{n\} \subseteq I_X$ and $K \setminus \{n\} \subseteq I_Y$. On the other hand, we know from \ref{it:lem:pre-sidonicity(i)} that $I_X \uplus I_Y = \llb 1, n+1 \rrb$. So, putting it all together, we can conclude that $J \cap K = \{n\}$.
\end{proof}
\begin{lemma}
\label{lem:fund-lemma}
Let $u_1, \ldots, u_{n+1} \in \mathbf N^+$ be given so that
\begin{enumerate*}[label={\rm (\alph{*})}]
\item\label{it:lem:stronger-sidonicity(a1)} $u_1 + \cdots + u_n \le u_{n+1} - u_n$,
\item\label{it:lem:stronger-sidonicity(b1)} $2u_n \ne u_{n+1}$, and
\item\label{it:lem:stronger-sidonicity(c1)} $u_1 + \cdots + u_i < \frac{1}{2} u_{i+1}$ for all $i \in \llb 1, n-1 \rrb$.
\end{enumerate*}
Next, let $X, Y \in \mathcal P_\fuN(\mathbf N)$ such that
\begin{equation}
\label{equ:decomposition-into-sum}
\{0, u_1\} + \cdots + \{0, u_{n+1}\} = X + Y,
\end{equation}
and set
$
I_X := \bigl\{i \in \llb 1, n+1 \rrb: u_i \in X\bigr\}$ and
$I_Y := \bigl\{i \in \llb 1, n+1 \rrb: u_i \in Y \bigr\}$.
The following hold:
\begin{enumerate}[label={\rm (\roman{*})}]
\item\label{it:fund-lemma(iii)} $X \setminus \{u_1 + \cdots + u_n\} = \sum_{i \in I_X} \{0, u_i\}$ and $Y \setminus \{u_1 + \cdots + u_n\} = \sum_{i \in I_Y} \{0, u_i\}$.
\item\label{it:fund-lemma(ii)} If $X \ne \sum_{i \in I_X} \{0, u_i\}$ or $Y \ne \sum_{i \in I_Y} \{0, u_i\}$, then $n \ge 2$, $u_1 + \cdots + u_n = u_{n+1} - u_n$, and one of $X$ and $Y$ is equal to $\{0, u_n\}$.
\end{enumerate}
\end{lemma}
\begin{proof}
To start with, we note for future reference that conditions \ref{it:lem:stronger-sidonicity(a1)}-\ref{it:lem:stronger-sidonicity(c1)} yield
\begin{equation}
\label{equ:easy-inequality}
u_1 + \cdots + u_i < 2u_i < u_{i+1}
\quad \text{for all }i \in \llb 1, n \rrb,
\end{equation}
and for the sake of notation we set
$$
U := \sum_{i=1}^{n+1} \{0, u_i\},
\quad
U^\ast := \sum_{i=1}^{n-1} \{0, u_i\},
\quad
U^\prime := U^\ast + \{0, u_n\},
%\sum_{i=1}^n \{0, u_i\},
\quad\text{and}\quad
U^{\prime\prime} := U^\ast + \{0, u_{n+1}\}.
$$
To ease the exposition, we break up the proof into a series of claims. We will often use without comment that $X, Y \subseteq U$, as is implied by Proposition \ref{prop:valuations}\ref{it:prop:valuations(i)}.
Moreover, we assume, based on Lemma \ref{lem:pre-sidonicity}\ref{it:lem:pre-sidonicity(i)}, that
$u_n \in X$ (as the statements to be proved are symmetric with respect to $X$ and $Y$).
\begin{claim}
\label{cl:lem:stronger-sidonicity(E)}
Assume that $Y \cap (U^{\prime\prime} + u_n) \ne \emptyset$. Then the following hold:
\begin{enumerate}[label={\rm (\textsc{a}\arabic{*})}]
\item\label{it:cl:lem:stronger-sidonicity(E)(1)} $n \ge 2$, $Y \cap (U^{\prime\prime} + u_n) = \{u_1 + \cdots + u_n\} = \{u_{n+1} - u_n\}$, and $\llb 1, n-1 \rrb \subseteq I_Y$.
\item\label{it:cl:lem:stronger-sidonicity(E)(2)} $X = \{0, u_n\}$ and $I_Y = \llb 1, n + 1 \rrb \setminus \{n\}$.
\end{enumerate}
\end{claim}
\begin{proof}[Proof of Claim \textup{\ref{cl:lem:stronger-sidonicity(E)}}]
\ref{it:cl:lem:stronger-sidonicity(E)(1)} Let $K \subseteq \llb 1, n + 1 \rrb$ such that $n \in K$ and take $y := \sum_{k \in K} u_k \in Y$.
Since $u_n \in X$, we get from Lemma \ref{lem:pre-sidonicity}\ref{it:lem:pre-sidonicity(ii)} (applied with $\mathcal{J} = \{n\}$ and $\mathcal{K} = K$) that $K \subseteq \llb 1, n \rrb$ and $u_n + y \ge u_{n+1}$.
So, it follows from condition \ref{it:lem:stronger-sidonicity(a1)} that $u_n + y = u_{n+1}$, which is only possible if $n \ge 2$ (recall that $2u_n \ne u_{n+1}$) and $K = \llb 1, n \rrb$, i.e., $y = u_1 + \cdots + u_n$. Then $\llb 1, n-1 \rrb \subseteq I_Y$, by Lemma \ref{lem:pre-sidonicity}\ref{it:lem:pre-sidonicity(iii)}.

\ref{it:cl:lem:stronger-sidonicity(E)(2)} Let $x \in X^+$. Then $x = \sum_{j \in J} u_j \in X$ for some non-empty $J \subseteq \llb 1, n+1 \rrb$, and we get from Lemma \ref{lem:pre-sidonicity}\ref{it:lem:pre-sidonicity(ii)} $\bigl($applied with $\mathcal{J} = J$ and $\mathcal{K} = \llb 1, n \rrb\bigr)$ that $J \cap \llb 1, n \rrb = \emptyset$ or $n \in J \subseteq \llb 1, n \rrb$. In particular, the maximum of $X$ is $\le u_{n+1}$, and hence $X \setminus \{u_{n+1}\} \subseteq U^\prime$, because $J \cap \llb 1, n \rrb = \emptyset$ only if $J = \{n+1\}$.

Suppose that $J = \{n+1\}$, namely, $u_{n+1} \in X$. Then Lemma \ref{lem:pre-sidonicity}\ref{it:lem:pre-sidonicity(ii)} yields $Y \cap (U^\prime + u_{n+1}) = \emptyset$, and we find that $X+Y \subseteq (U^\prime + Y) \cup (X + U^\ast) \cup (u_1 + \cdots + u_{n+1})$, for we know from \ref{it:cl:lem:stronger-sidonicity(E)(1)} that $n \ge 2$ and $Y \setminus \{u_1 + \cdots + u_n\} \subseteq U^\ast$. On the other hand, we see that
\begin{equation*}
\max(U^\prime + Y) = 2(u_1 + \cdots + u_n)
\fixed[-0.25]{\text{ }}\stackrel{\ref{it:cl:lem:stronger-sidonicity(E)(1)}}{=}\fixed[-0.25]{\text{ }}
u_{n+1} + \sum_{i=1}^{n-1} u_i = \max(X + U^\ast)
\fixed[-0.25]{\text{ }}\stackrel{\eqref{equ:easy-inequality}}{< }\fixed[-0.25]{\text{ }}
u_{n+1} + u_n,
\end{equation*}
and it is clear that $u_n + u_{n+1} < u_1 + \cdots + u_{n+1}$ (because $n \ge 2$). Thus $u_n+u_{n+1} \notin X + Y = U$, which is, however, a contradiction.
So, putting it all together, we must conclude that $n \in J \subseteq \llb 1, n \rrb$. 

But we have from Lemma \ref{lem:pre-sidonicity}\ref{it:lem:pre-sidonicity(iii)} and \ref{it:cl:lem:stronger-sidonicity(E)(2)} that $J \setminus \{n\} \subseteq I_X$ and $\llb 1, n-1 \rrb \subseteq I_Y$; and from Lemma \ref{lem:pre-sidonicity}\ref{it:lem:pre-sidonicity(i)} that $I_X \uplus I_Y = \llb 1, n+1 \rrb$. So $J = \{n\}$, and since $x$ was an arbitrary element in $X^+$ and we are assuming that $u_n \in X$, it follows that $X = \{0, u_n\}$ and $I_Y = \llb 1, n+1 \rrb \setminus \{n\}$.
\end{proof}
\begin{claim}
\label{cl:lem:stronger-sidonicity(F)}
Let $J, K \subseteq \llb 1, n+1 \rrb$ such that $\sum_{j \in J} u_j \in X$ and $\sum_{k \in K} u_k \in Y$. Then one \textup{(}and only one\textup{)} of the following two cases occurs:
\begin{enumerate}[label={\rm (\textsc{b}\arabic{*})}]
\item\label{cl:lem:stronger-sidonicity(Q)(f1)} $J \subseteq I_X$, $K \subseteq I_Y$, and $J \cap K = \emptyset$.
\item\label{cl:lem:stronger-sidonicity(Q)(f2)} $J \subseteq I_X = \{n\}$, $K = \llb 1, n \rrb$, and conditions \ref{it:cl:lem:stronger-sidonicity(E)(1)} and \ref{it:cl:lem:stronger-sidonicity(E)(2)} of Claim \textup{\ref{cl:lem:stronger-sidonicity(E)}} are satisfied.
\end{enumerate}
\end{claim}
\begin{proof}[Proof of Claim \textup{\ref{cl:lem:stronger-sidonicity(F)}}]
Set $x := \sum_{j \in J} u_j$ and $y := \sum_{k \in K} u_k$. We distinguish two cases:
\vskip 0.1cm
\textsc{Case 1}: $K \subseteq I_Y$. 
We prove $J \cap I_Y = \emptyset$; this will give $J \subseteq I_X$ and $J \cap K = \emptyset$, since $J \subseteq \llb 1, n+1 \rrb$ and, by Lemma \ref{lem:pre-sidonicity}\ref{it:lem:pre-sidonicity(i)}, $I_X \uplus I_Y = \llb 1, n+1 \rrb$.
For, assume to the contrary that $J \cap I_Y$ is non-empty, and let $i_0 \in J \cap I_Y$. Then we infer from Lemma \ref{lem:pre-sidonicity}\ref{it:lem:pre-sidonicity(ii)} $\bigl($applied with $\mathcal{J} = J$ and $\mathcal{K} = \{i_0\}\bigr)$ that $i_0 = n$, and hence $u_n \in Y$, in contradiction to Claim \ref{cl:lem:stronger-sidonicity(E)}.
\vskip 0.1cm
\textsc{Case 2}: $K \not\subseteq I_Y$. Since $K \subseteq \llb 1, n+1 \rrb$ and, by Lemma \ref{lem:pre-sidonicity}\ref{it:lem:pre-sidonicity(i)}, $I_X \uplus I_Y = \llb 1, n+1 \rrb$, it is clear that $I_X \cap K \ne \emptyset$. Let $i_0 \in I_X \cap K$. Then Lemma \ref{lem:pre-sidonicity}\ref{it:lem:pre-sidonicity(ii)} $\bigl($applied with $\mathcal{J} = \{i_0\}$ and $\mathcal{K} = K\bigr)$ yields $i_0 = n$, which implies by Claim \ref{cl:lem:stronger-sidonicity(E)} that $X = \{0, u_n\}$ and $Y \cap (U^\ast + u_n) = \{u_1 + \cdots + u_n\}$.
So $J \subseteq I_X = \{n\}$ and $y = u_1+\cdots+u_n$, and by \eqref{equ:easy-inequality} and Lemma \ref{lem:sidonicity} this is possible only if $K = \llb 1, n \rrb$.
\end{proof}
\begin{claim}
\label{cl:lem:stronger-sidonicity(G)}
%Suppose that $\sum_{i \in I} u_i < u_{n+1}$ for a g
Given $I \subseteq \llb 1, n+1 \rrb$, there exist $J, K \subseteq \llb 1, n+1 \rrb$ for which $\sum_{j \in J} u_j \in X$, $\sum_{k \in K} u_k \in Y$, and $\sum_{i \in I} u_i = \sum_{j \in J} u_j + \sum_{k \in K} u_k$. Moreover, one \textup{(}and only one\textup{)} of the following holds:
\begin{enumerate}[label={\rm (\textsc{c}\arabic{*})}]
\item $J \uplus K = I$, $J \subseteq I_X$, and $K \subseteq I_Y$.
\item $J \subseteq I_X = \{n\}$ and $K = \llb 1, n \rrb$.
\end{enumerate}
\end{claim}
\begin{proof}[Proof of Claim \textup{\ref{cl:lem:stronger-sidonicity(G)}}]
Set $z := \sum_{i \in I} u_i$. Then $z \in X+Y = U$, and hence there exist $J, K \subseteq \llb 1, n+1 \rrb$ such that
$x := \sum_{j \in J} u_j \in X$,
$y : = \sum_{k \in K} u_k \in Y$, and $z = x + y$.
If $K \subseteq I_Y$, then $J \cap K = \emptyset$ and $J \subseteq I_X$ by point \ref{cl:lem:stronger-sidonicity(Q)(f1)} of Claim \ref{cl:lem:stronger-sidonicity(F)}, hence $J \uplus K = I$ by Lemma \ref{lem:basic} (applied with $\mathcal I = I$, $\mathcal J = J$, and $\mathcal K = K$). Otherwise, point \ref{cl:lem:stronger-sidonicity(Q)(f2)} of Claim \ref{cl:lem:stronger-sidonicity(F)} yields $J \subseteq I_X = \{n\}$ and $K = \llb 1, n \rrb$.
\end{proof}
\begin{claim}
\label{cl:lem:stronger-sidonicity(H)}
$\sum_{i \in \mathcal{I}_X} u_i \in X$ for every $\mathcal{I}_X \subseteq I_X$, and $\sum_{i \in \mathcal{I}_Y} u_i \in Y$ for every $\mathcal{I}_Y \subseteq I_Y$.
\end{claim}
\begin{proof}[Proof of Claim \textup{\ref{cl:lem:stronger-sidonicity(H)}}]
We just prove the statement relative to $X$, as the other is similar. For,
let $I \subseteq I_X$, and set $z := \sum_{i \in I} u_i$. The claim is obvious if $|I| \le 1$ (by the very definition of $I_X$), so assume $|I| \ge 2$.

Since $z \in U = X + Y$, there exist $J, K \subseteq \llb 1, n+1 \rrb$ such that $x := \sum_{j \in J} u_j \in X$, $y := \sum_{k \in K} u_k \in Y$, and $z = x + y$. Because $|I_X| \ge |I| \ge 2$, we thus obtain from Claim \ref{cl:lem:stronger-sidonicity(G)} that $J \uplus K = I$ and $K \subseteq I_Y$. But this is possible only if $K = \emptyset$, because $I_X \cap I_Y = \emptyset$ by Lemma \ref{lem:pre-sidonicity}\ref{it:lem:pre-sidonicity(i)} and $K \subseteq I \subseteq I_X$. So $I = J$, and hence $z = x \in X$.
\end{proof}
With all this in hand, we are ready to conclude. In fact, we get from Claim \ref{cl:lem:stronger-sidonicity(G)} that
$$
X \subseteq \sum_{i \in I_X} \{0, u_i\}
\quad\text{and}\quad
Y \setminus \{u_1 + \cdots + u_n\} \subseteq \sum_{i \in I_Y} \{0, u_i\},
$$
and from Claim \ref{cl:lem:stronger-sidonicity(H)} that
$$
\sum_{i \in I_X} \{0, u_i\} \subseteq X
\quad\text{and}\quad
\sum_{i \in I_Y} \{0, u_i\} \subseteq Y \setminus \{u_1 + \cdots + u_n\},
$$
with the result that $X = \sum_{i \in I_X} \{0, u_i\}$ and $Y \setminus \{u_1 + \cdots + u_n\} = \sum_{i \in I_Y} \{0, u_i\}$. This proves point \ref{it:fund-lemma(iii)}, while \ref{it:fund-lemma(ii)} follows from  Claim \ref{cl:lem:stronger-sidonicity(E)} (recall that we are assuming without loss of generality that $u_n \in X$).
\end{proof}
The next step is to determine the set of lengths of $X$ for some special choices of the set $X \in \mathcal P_\fuN(\mathbf N)$. Consistently with the notation introduced in \S{ }\ref{sec:factorizations} (in the special case of reduced, commutative monoids), we will identify a word $\mathfrak{c} \in \mathscr F(\mathscr A(\mathcal P_{\fin,0}(\mathbf N)))$ with the congruence class $\llb \mathfrak c \rrb_{\mathscr C_{\mathcal P_{\fin,0}(\mathbf N)}}$.
\begin{proposition}
\label{prop:set_of_lengths_of_intervals}
$\LLs\fixed[-0.2]{\text{ }}\bigl(\llb 0, n \rrb\bigr) = \llb 2, n \rrb$ for every $n \ge 2$.
\end{proposition}
\begin{proof}
As was noted before, $\mathcal P_\fuN(\mathbf{N})$ is a reduced \BF-monoid.
So the claim is trivial if $n = 2$, because if $\llb 0, 2 \rrb = X + Y$ for some $X, Y \subseteq \mathcal{P}_\fuN(\mathbf N) \setminus \bigl\{\{0\}\bigr\}$, then it is clear that $X = Y = \llb 0, 1 \rrb$.

Accordingly, suppose the claim is true for a fixed $n \ge 2$, and observe that $\llb 0, n+1 \rrb = \llb 0, 1 \rrb + \llb 0, n \rrb$. Since ${\sf L}(X) + {\sf L}(Y) \subseteq {\sf L}(X+Y)$ for all $X, Y \in \mathcal P_\fuN(\mathbf{N})$, it follows that
\begin{equation}
\label{equ:inclusion}
\LLs\fixed[-0.2]{\text{ }}\bigl(\llb 0, n+1 \rrb\bigr) \supseteq 1 + \LLs\fixed[-0.2]{\text{ }}\bigl(\llb 0, n \rrb\bigr) = \llb 3, n+1 \rrb.
\end{equation}
On the other hand, let $A := \{0, 2\}$ if $n = 2$ and $A := \{0, 1\} \cup \bigl\{k \in \llb 2, n \rrb: k \equiv n \bmod 2\bigr\}$ otherwise. Then $A$ is an atom by Propositions \ref{prop:valuations}\ref{it:prop:valuations(iv)} and \ref{prop:large_atoms} (apply the latter with $d = 2$ and $\ell = q = 1$),
and we have $\llb 0, n+1 \rrb = \{0, 1\} + A$, which implies, together with \eqref{equ:inclusion}, that $\llb 2, n+1 \rrb \subseteq \LLs\fixed[-0.2]{\text{ }}\bigl(\llb 0, n+1 \rrb\bigr)$.

So we are done, since $(\mathbf N, +)$ is a linearly orderable monoid, and therefore we get from Theorem \ref{th:normable_implies_atomic}\ref{it:th:normable(iv)} and Proposition \ref{prop:lin-orderable-H}\ref{it:prop:lin-orderable-H(ii)} that $\sup \LLs\fixed[-0.2]{\text{ }}\bigl(\llb 0, n+1 \rrb\bigr) \le \bigl|\llb 0, n+1 \rrb\bigr| - 1 = n+1$.
\end{proof}
\begin{proposition}
\label{lem:unicity}
Let $v_1, \ldots, v_\ell \in \mathbf N^+$ such that $v_1 + \cdots + v_i < \frac{1}{2} v_{i+1}$ for every $i \in \llb 1, \ell - 2 \rrb$ and, if $\ell \ge 2$, $v_1 + \cdots + v_{\ell-1} < v_\ell - v_{\ell-1}$. Then $\mathsf{Z}\fixed[-0.15]{\text{ }} \bigl(\{0, v_1\} + \cdots + \{0, v_\ell\}\bigr) \fixed[-0.15]{\text{ }} = \fixed[-0.15]{\text{ }} \bigl\{\{0, v_1\} \ast \cdots \ast \{0, v_\ell\}\bigr\}$ in $\mathcal P_\fuN(\mathbf{N})$.
\end{proposition}
\begin{proof}
If $\ell = 1$, the conclusion is trivial, since every two-element set in $\mathcal P_{\fin,0}(\mathbf N)$ is an atom by Proposition \ref{prop:valuations}\ref{it:prop:valuations(iv)}. So let $\ell \ge 2$ and assume that the following condition is verified:
%, serving as an inductive hypothesis, is satisfied:
\begin{enumerate}[label={\rm (\textsc{h})}]
\item\label{item:condition(H)} If $t \in \llb 1, \ell-1\rrb$ and $x_1, \ldots, x_t \in \mathbf N^+$ are such that $x_1 + \cdots + x_i < \frac{1}{2}x_{i+1}$ for all $i \in \llb 1, t-2 \rrb$ and, when $t \ge 2$, $x_1 + \cdots + x_{t-1} < x_t - x_{t-1}$, then $\mathsf{Z}\fixed[-0.15]{\text{ }} \bigl(\{0, x_1\} + \cdots + \{0, x_t\}\bigr) \fixed[-0.15]{\text{ }} = \fixed[-0.15]{\text{ }} \bigl\{\{0, x_1\} \ast \cdots \ast \{0, x_t\}\bigr\}$.
\end{enumerate}
Next, suppose that $V := \sum_{i=1}^\ell \{0, v_i\} = X + Y$ for some non-unit $X, Y \in \mathcal P_\fuN(\mathbf N)$, and set
$$
I_X := \bigl\{i \in \llb 1, \ell \fixed[0.15]{\text{ }} \rrb: v_i \in X\bigr\}
\quad\text{and}\quad
I_Y := \bigl\{i \in \llb 1, \ell \fixed[0.15]{\text{ }} \rrb: v_i \in Y \bigr\}.
$$
By Lemma \ref{lem:fund-lemma} (applied with $n = \ell - 1$ and $u_1 = v_1, \ldots, u_{n+1} = v_\ell$), we see that
\begin{equation}\label{equ:splitting-equ}
X = \sum_{i \in I_X} \{0, v_i\},
\quad
Y = \sum_{i \in I_Y} \{0, v_i\},
\quad\text{and}\quad
I_X \uplus I_Y = \llb 1, \ell \rrb.
\end{equation}
In particular, $\emptyset \ne I_X, I_Y \subsetneq \llb 1, \ell \rrb$, because $X$ and $Y$ are both different from $\{0\}$.

Put $m := |I_X|$, and let $i_1, \ldots, i_m$ be the natural enumeration of $I_X$.
Since $v_{i_1}, \ldots, v_{i_m}$ is a subsequence of $v_1, \ldots, v_\ell$, we have $v_{i_1} + \cdots + v_{i_k} < \frac{1}{2}v_{i_{k+1}}$ for all $k \in \llb 1, m-2 \rrb$ and, for $m \ge 2$, $v_{i_1} + \cdots + v_{i_{m-1}} < v_{i_m} - v_{i_{m-1}}$. Therefore, we derive from condition \ref{item:condition(H)} that
$\mathsf{Z}(X) = \fixed[-0.15]{\text{ }} \bigl\{\{0, v_{i_1}\} \ast \cdots \ast \{0, v_{i_m}\}\bigr\}$.
Likewise, if $n := |I_Y|$ and $j_1, \ldots, j_n$ is the natural enumeration of $I_Y$, then $\mathsf{Z}(Y) = \fixed[-0.15]{\text{ }} \bigl\{\{0, v_{j_1}\} \ast \cdots \ast \{0, v_{j_n}\}\bigr\}$. 

So, putting it all together and recalling from \eqref{equ:splitting-equ} that $I_X \uplus I_Y = \llb 1, \ell \rrb$, we conclude by Lemma \ref{lem:products-of-factorizations} that $\mathsf{Z}(V) = \fixed[-0.15]{\text{ }} \bigl\{\{0, v_1\} \ast \cdots \ast \{0, v_\ell\}\bigr\}$.
\end{proof}
\begin{proposition}
\label{prop:final-step}
Let $n \in \mathbf N_{\ge 2}$, and let $u_1, \ldots, u_{n+1} \in \mathbf N^+$ such that
\begin{enumerate*}[label={\rm (\alph{*})}]
\item\label{it:prop:sidonicity(b)} $u_1 + \cdots + u_n = u_{n+1} - u_n$ and
\item\label{it:prop:sidonicity(c)} $u_1 + \cdots + u_i < \frac{1}{2} u_{i+1}$ for every $i \in \llb 1, n-1 \rrb$.
\end{enumerate*}
Set
$$
U := \{0, u_1\} + \cdots + \{0, u_{n+1}\}
\quad\text{and}\quad
A := \left\{ \sum_{i \in I} u_i: I \subseteq \llb 1, n-1 \rrb \right\}.
$$
Then the following hold:
\begin{enumerate}[label={\rm (\roman{*})}]
\item\label{it:prop:two_factorizations(i)} $B := A \cup (A + u_{n+1}) \cup \{u_1 + \cdots + u_n\} \in \mathscr{A}(\mathcal P_\fuN(\mathbf{N}))$ and $|B| \ge 3$.
\item\label{it:prop:two_factorizations(ii)} $\mathsf{Z}(U) = \fixed[-0.15]{\text{ }} \bigl\{\{0, u_n\} \ast B \fixed[0.2]{\text{ }}, \fixed[0.2]{\text{ }} \{0, u_1\} \ast \cdots \ast \{0, u_{n+1}\}\bigr\}$.
\end{enumerate}
In particular, ${\sf L}(U) = \{2, n+1\}$, $\Delta(U) = \{n-1\}$, and ${\sf c}(U) = n$.
\end{proposition}
\begin{proof}
The ``In particular'' part of the statement is a straightforward consequence of point \ref{it:prop:two_factorizations(ii)}, so we can definitely focus on the proof of \ref{it:prop:two_factorizations(i)} and \ref{it:prop:two_factorizations(ii)}.

\ref{it:prop:two_factorizations(i)} Clearly $|B| \ge 3$, and hence  $B \ne \{0, u_i\}$ for every $i \in \llb 1, n+1 \rrb$, because $\{0, u_1, u_1 + \cdots + u_n\} \in B$ and $0 < u_1 < u_1 + \cdots + u_n$ (here we use that $n \ge 2$). Moreover, we have
$$
2\sup A = 2(u_1 + \cdots + u_{n-1}) \fixed[-0.25]{\text{ }}\stackrel{\ref{it:prop:sidonicity(c)}}{<}\fixed[-0.25]{\text{ }} u_n \le u_1 + \cdots + u_n \fixed[-0.25]{\text{ }}\stackrel{\ref{it:prop:sidonicity(b)}}{=}\fixed[-0.25]{\text{ }} u_{n+1} - u_n \fixed[-0.25]{\text{ }}\stackrel{\ref{it:prop:sidonicity(c)}}{<}\fixed[-0.25]{\text{ }} u_{n+1} - \sup A.
$$
Therefore, we infer from Proposition \ref{prop:atomic-unions} (applied with $b = u_{n+1}$ and and $c = u_1 + \cdots + u_n$) that $B$ is an atom of $\mathcal P_\fuN(\mathbf N)$, since it is clear from \ref{it:prop:sidonicity(b)} that $2(u_1 + \cdots + u_n) - u_{n+1} = $

\ref{it:prop:two_factorizations(ii)} Observe that $U$ is not an atom of $\mathcal P_{\fuN}(\mathbf N)$, and recall that $\mathcal P_\fuN(\mathbf N)$ is a reduced, commutative \BF-monoid (by Corollary \ref{cor:P_fuN(N)}). Accordingly, let $U = X + Y$ for some non-unit $X, Y \in \mathcal P_\fuN(\mathbf N)$, and set
$$\mathsf{Z}(X,Y) := \bigl\{\mathfrak{a} \ast \mathfrak{b}: (\mathfrak{a}, \mathfrak{b}) \in \mathsf{Z}(X) \times \mathsf{Z}(Y)\bigr\} \subseteq \mathsf Z(\mathcal P_\fuN(\mathbf N)).$$
Moreover, take 
$
I_X := \bigl\{i \in \llb 1, n+1 \rrb: u_i \in X\bigr\}$ and $
I_Y := \bigl\{i \in \llb 1, n+1 \rrb: u_i \in Y \bigr\}$.
By Lemma \ref{lem:fund-lemma},  
we have $I_X \uplus I_Y = \llb 1, n+1 \rrb$, and there are only two cases:
\vskip 0.1cm
\textsc{Case 1:} $X = \{0, u_n\}$ and $Y = B$ (up to rearrangement). By Proposition \ref{prop:valuations}\ref{it:prop:valuations(iv)} and point \ref{it:prop:two_factorizations(i)}, both $X$ and $Y$ are atoms, hence $\mathsf{Z}(X,Y) = \bigl\{\{0, u_n\} \ast B\bigr\}$.
\vskip 0.1cm
\textsc{Case 2:} $X = \sum_{i \in I_X} \{0, u_i\}$ and $Y = \sum_{i \in I_Y} \{0, u_i\}$. Let $i_1, \ldots, i_h$ be the natural enumeration of $I_X$ and $j_1, \ldots, j_k$ the natural enumeration of $I_Y$, where $h := |I_X|$ and $k := |I_Y|$ $\bigl($it is clear that $h, k \in \mathbf N^+$, because $X, Y \ne \{0\}\bigr)$.
Since $u_{i_1}, \ldots, u_{i_h}$ is a \textit{proper} subsequence of $u_1, \ldots, u_{n+1}$, it holds
$$u_{i_1} + \cdots + u_{i_s} < \frac{1}{2} u_{i_{s+1}},
\quad \text{for all }s \in \llb 1, h-2 \rrb,
$$
and
$$
u_{i_1} + \cdots + u_{i_{h-1}} < u_{i_h} - u_{i_{h-1}},\quad\text{for } h \ge 2.
$$
So, we get from Proposition \ref{lem:unicity} (applied with $\ell = h$ and $v_1 = u_{i_1}, \ldots, v_\ell = u_{i_h}$) that $\mathsf{Z}(X) = \bigl\{\{0, u_{i_1}\} \ast \cdots \ast \{0, u_{i_h}\}\bigr\}$. And in a similar way, we obtain that $\mathsf{Z}(Y) = \bigl\{\{0, u_{j_1}\} \ast \cdots \ast \{0, u_{j_k}\}\bigr\}$. Hence, using that $\llb 1, n+1 \rrb = \{i_1, \ldots, i_h\} \uplus \{j_1, \ldots, j_k\}$, we find
$\mathsf{Z}(X,Y) = \bigl\{\{0, u_1\} \ast \cdots \ast \{0, u_{n+1}\}\bigr\}$.
\vskip 0.1cm
We are now in the position to finish the proof of point \ref{it:prop:two_factorizations(ii)}, as we infer from the above and Lemma \ref{lem:products-of-factorizations} that $\mathsf{Z}(U) = \fixed[-0.15]{\text{ }} \bigl\{\{0, u_n\} \ast B \fixed[0.2]{\text{ }}, \fixed[0.2]{\text{ }} \{0, u_1\} \ast \cdots \ast \{0, u_{n+1}\}\bigr\}$.
\end{proof}
Finally, we have all the ingredients we need to prove the main result of this section.
\begin{theorem}
\label{th:main-theorem}
Let $H$ be a Dedekind-finite, non-torsion monoid. Then:
\begin{enumerate}[label={\rm (\roman{*})}]
\item\label{it:th:main-theorem(i)} $\mathscr{L}(\mathcal P_\fin(H)) \supseteq \mathscr{L}(\mathcal P_{\fun}(H)) \supseteq \mathscr{L}(\mathcal P_\fuN(\mathbf{N}))$.
\item\label{it:th:main-theorem(ii)} $\mathscr{U}_k(\mathcal P_\fin(H)) = \mathscr{U}_k(\mathcal P_{\fun}(H)) = \mathscr{U}_k(\mathcal P_\fuN(\mathbf N)) = \mathbf N_{\ge 2} $ for every $k \ge 2$.
\item\label{it:th:main-theorem(iii)} $\Delta(\mathcal P_\fin(H)) = \Delta(\mathcal P_{\fun}(H)) = \Delta(\mathcal P_\fuN(\mathbf{N})) = \mathbf N^+$.
\item\label{it:th:main-theorem(iv)} $\cat(\mathcal P_\fin(H)) \supseteq \cat(\mathcal P_{\fun}(H)) \supseteq \cat(\mathcal P_\fuN(\mathbf{N})) = \mathbf N^+$.
\end{enumerate}
In particular, if $H$ is a linearly orderable \BF-monoid, then the inclusions in point \ref{it:th:main-theorem(iv)} are equalities.
\end{theorem}
\begin{proof}
To ease notation, we will write $P$ in place of $\mathcal P_{\fin,\times}(H)$ and $P_0$ in place of $\mathcal P_{\fin,0}(\mathbf N)$.

Clearly,
\ref{it:th:main-theorem(i)} follows from Theorem \ref{th:transfer} and Proposition \ref{prop:basic-properties-of-power-monoids}\ref{it:prop:basic-properties-of-power-monoids(iv)}; \ref{it:th:main-theorem(ii)} from \ref{it:th:main-theorem(i)} and Proposition \ref{prop:set_of_lengths_of_intervals}; and \ref{it:th:main-theorem(iii)} from \ref{it:th:main-theorem(i)} and Proposition \ref{prop:final-step}.
As for \ref{it:th:main-theorem(iv)}, we need some more work.

To start with, we get from Proposition \ref{prop:final-step} that $\mathbf N_{\ge 2} \subseteq \cat(P_0)$, and since $P_0$ is a BF-monoid, it is evident that $\cat(P_0) \subseteq \mathbf N^+$. This yields $\cat(P_0) = \mathbf N^+$, as it is easy to check that
\begin{equation}
\label{equ:final-equation}
\mathsf{Z}\bigl(\llb 0, 6 \rrb \setminus \{4\}\bigr) = \fixed[-0.15]{\text{ }} \bigl\{\{0, 1\} \ast \{0,2,5\} \fixed[0.2]{\text{ }}, \fixed[0.2]{\text{ }} \{0, 1\} \ast \{0,1,2,5\}\bigr\} \subseteq \mathsf Z(P_0).
\end{equation}
On the other hand, Proposition \ref{prop:basic-properties-of-power-monoids}\ref{it:prop:basic-properties-of-power-monoids(iv)} yields $\cat(P) \subseteq \cat(\mathcal P_\fin(H))$. So we are left to show that $\cat(P_0) \subseteq \cat(P)$, as the ``In particular'' part of the statement is a consequence of \ref{it:th:main-theorem(iv)} and Proposition \ref{prop:lin-orderable-H}.
%, and Remark \ref{rem:set-of-catenary-degrees-when-H-is-BF}.

For, pick $n \in \mathbf N^+$ and let $\Phi$ be the same homomorphism of Theorem \ref{th:transfer}. We set $U_n := \llb 0, 6 \rrb \setminus \{4\}$ if $n = 1$; and $U_n := \sum_{i=1}^{n+1} \{0, u_i\}$ otherwise, where $u_1, \ldots, u_{n+1} \in \mathbf N^+$, $u_1 + \cdots + u_n = u_{n+1} - u_n$, and $u_1 + \cdots + u_i < \frac{1}{2} u_{i+1}$ for every $i \in \llb 1, n-1 \rrb$.
Also, we define ${\sf c}_n := {\sf c}_{P_0}(U_n)$ and ${\sf c}_n^\star := {\sf c}_{P}(\Phi(U_n))$.

By \eqref{equ:final-equation} and Proposition \ref{prop:final-step}, there exist atoms $A_0, \ldots, A_{n+1} \in \mathscr{A}(P_0)$ such that $|A_i| \ne |A_0|$ for all $i \in \llb 1, n+1 \rrb$ and $
\mathsf{Z}_{P_0}(U_n) = \{A_0 \ast A_1, A_1 \ast \cdots \ast A_{n+1}\} \subseteq \mathsf Z(P_0)$.
So it is evident that $\mathsf c_n = n$.

On the other hand, we know from Theorem \ref{th:transfer} that $\Phi$ is actually an injective equimorphism. 
Consequently, it follows from the above and condition \ref{covariant-transfer(3)} of Definition \ref{def:equimorphisms} that 
$$
\mathsf{Z}_{P}(\Phi(U_n)) = \bigl\{\llb \Phi(A_0) \ast \Phi(A_1) \rrb_{\mathscr C_P}, \llb \Phi(A_1) \ast \cdots \ast \Phi(A_{n+1}) \rrb_{\mathscr C_P} \bigr\} \subseteq \mathsf Z(P),
$$
Besides, the injectivity of $\Phi$ implies that $|\Phi(A_i)| = |A_i| \ne |A_0| = |\Phi(A_0)|$ for every $i \in \llb 1, n+1 \rrb$, with the result that
$(\Phi(A_0) \ast \Phi(A_1)) \wedge_P (\Phi(A_1) \ast \cdots \ast \Phi(A_{n+1})) = n$.
So, putting it all together, we conclude from Lemma \ref{lem:fundamental-lemma-for-distance} that ${\sf c}_n^\star = {\sf c}_n = n$.
This finishes the proof, because $n \in \mathbf N^+$ was arbitrary.
\end{proof}
We close the section by proving that there is little chance that the arithmetic results summarized in Theorem \ref{th:main-theorem} can be also obtained via ``standard transfer techniques''.
\begin{proposition}
\label{prop:not_a_transfer_Krull_monoid}
Let $H$ be a Dedekind-finite, non-torsion monoid. Then neither $\mathcal P_{\fin}(H)$ nor $\mathcal P_\fun(H)$ is equimorphic to a cancellative monoid \textup{(}in particular, neither is a transfer Krull monoid\textup{)}.
\end{proposition}
\begin{proof}
By Proposition \ref{prop:final-step} (applied with $r = 2$), there are $A, B, C, D \in \mathscr{A}(\mathcal P_\fuN(\mathbf N))$ such that $A + B = A + C + D$. So, if $\Phi$ is the equimorphism of Theorem \ref{th:transfer}, then $\bar{A} := \Phi(A)$, $\bar{B} := \Phi(B)$, $\bar{C} := \Phi(C)$, and $\bar{D} := \Phi(D)$ are atoms of $\mathcal P_\fun(H)$. In addition, $\bar{A} \bar{B} = \bar{A}\fixed[0.2]{\text{ }} \bar{C} \fixed[-0.1]{\text{ }} \bar{D}$.

Building on these premises, suppose for a contradiction that there is an equimorphism $\varphi: \mathcal P_{\fin}(H) \to K$ (respectively, $\varphi: \mathcal P_\fun(H) \to K$) for which $K$ is a cancellative monoid.
It follows
$$
\varphi(\bar{A}) \fixed[0.15]{\text{ }} \varphi(\bar{B}) = \varphi(\bar{A}) \fixed[0.15]{\text{ }} \varphi(\bar{C}) \fixed[0.15]{\text{ }} \varphi(\bar{D}),
$$
which, by cancellativity of $K$, yields $\varphi(\bar{B}) = \varphi(\bar{C}) \fixed[0.15]{\text{ }} \varphi(\bar{D})$.
However, we know from Proposition \ref{prop:basic-properties-of-power-monoids}\ref{it:prop:basic-properties-of-power-monoids(iv)} that $\mathcal P_\fun(H)$ is a divisor-closed submonoid of $\mathcal P_\fin(H)$, and this implies, by the above and Proposition \ref{prop:divisor-closed-sub}, that $\bar{B}$, $\bar{C}$, and $\bar{D}$ are also atoms of $\mathcal P_\fin(H)$. So, using that $\varphi$ is atom-preserving, we conclude that $\varphi(\bar{B})$, $\varphi(\bar{C})$, and $\varphi(\bar{D})$ are all atoms of $K$, in contradiction to the fact that $\varphi(\bar{B}) = \varphi(\bar{C}) \fixed[0.15]{\text{ }} \varphi(\bar{D})$.
\end{proof}
\section{Prospects for future research}
\label{sec:future}
We conjecture that, if $H$ is a Dedekind-finite, non-torsion monoid, then the systems of sets of lengths of $\mathcal P_\fin(H)$ and $\mathcal P_\fun(H)$ contain every non-empty finite subset of $\mathbf N_{\ge 2}$.
%, with equality if $H$ is linearly orderable and BF.
Note that, by Theorem \ref{th:transfer} and Proposition \ref{prop:lin-orderable-H}\ref{it:prop:lin-orderable-H(ii)}-\ref{it:prop:lin-orderable-H(iii)}, it is sufficient to show that 
$$
\mathscr{L}(\mathcal P_\fuN(\mathbf N)) = \bigl\{\{0\} \fixed[0.15]{\text{ }}, \{1\}\bigr\} \cup \mathcal P_\fin\fixed[-0.2]{\text{ }}\bigl(\mathbf N_{\ge 2}\bigr).
$$
The conjecture is probably difficult, and we hope it will stimulate further work in the subject. Analogous conclusions are known to hold for certain cancellative commutative monoids, see, e.g., \cite[Theorem 1]{Ka90} or  \cite[Corollary 4.1]{FrNaRi17}.
\section*{Acknowledgements}
\label{subsec:acks}
The authors are indebted to Alfred Geroldinger for invaluable comments and enlightening conversations; to Daniel Smertnig for useful discussions (in particular, we owe him the observation that a monoid has a length function only if it is unit-cancellative) and his help with a campaign of numerical experiments that have eventually led to the formulation of Lemma \ref{lem:fund-lemma}; to Benjamin Steinberg for answering a question related to Lemma \ref{lem:basic-properties-atoms-units}\ref{it:lem:basic-properties-atoms-units(i)} on MathOverflow (see \href{http://mathoverflow.net/a/261870/16537}{http://mathoverflow.net/questions/261850/}); and to an anonymous referee for some precious remarks.

\end{document}